\documentclass[10pt, reqno]{amsart}
 
\usepackage{amsmath,amsthm,amssymb,amscd,graphicx}

\usepackage{enumerate}
\usepackage{mathtools}
\usepackage{mathrsfs}
\usepackage{color}

 \usepackage{hyperref}

\renewcommand{\phi}{\varphi}
\def\C{\mathbb C}
\def\R{\mathbb R}

\def\Z{\mathbb Z}

\def\N{\mathbb N}
\def\E{\mathcal E}
\def\ET{\widetilde{\mathcal{E}}}

\def\e{e}
\def\eps{\varepsilon}
\def\dis{\displaystyle}   

\def\EE{\exp}

\newcommand{\pa}[1]{\left(#1\right)}
\newcommand{\pac}[1]{\left[#1\right]}
\newcommand{\paa}[1]{\left\{#1\right\}}
\newcommand{\bra}[1]{\left\langle #1 \right\rangle}
\newcommand{\abs}[1]{\left|#1\right|}
\newcommand{\absabs}[1]{\left\|#1\right\|}

\newcommand{\wt}{\widetilde}

\renewcommand{\Re}{  {\mathfrak{Re}}  }
\renewcommand{\Im}{   {\mathfrak{Im}} }
\newcommand{\ov}{  \overline  }
\newcommand\<{\langle}
\renewcommand\>{\rangle}
 
\newcommand{\simx}[2]{\underset{#1\to#2}{\thicksim}}

\catcode`\Ž = \active\def Ž{\'e} \catcode`\ƒ = \active\def ƒ{\'E}  \catcode`\Ë = \active\def Ë{\`A}\catcode`\ = \active\def {\`e}
\catcode`\ = \active\def {\c{c}} \catcode`\ˆ = \active\def
ˆ{\`a} \catcode`\ = \active \def {\^e} \catcode`\ = \active
\def {\`u} \catcode`\™ = \active \def ™{\^o} \catcode`\ž =
\active \def ž{\^u} \catcode`\" = \active \def "{\^{\i}}
\catcode`\‰ = \active \def ‰{\^a} \catcode`\š = \active \def
š{{\"o}}

\newtheorem{theorem}{Theorem}[section]

\newtheorem{lemme}[theorem]{Lemma}
\newtheorem{prop}[theorem]{Proposition}
\theoremstyle{definition}
\newtheorem{df}[theorem]{Definition}
\theoremstyle{remark}
\newtheorem{remark}[theorem]{Remark}

\numberwithin{equation}{section}

\begin{document}

\author{Pierre Germain}
\address{Department of Mathematics, Huxley building, South Kensington campus, Imperial College London, London SW7 2AZ, United Kingdom} %
\email{pgermain@ic.ac.uk}

\author{Valentin Schwinte}
\address{Universit\'e de Lorraine, CNRS, IECL, F-54000 Nancy, France}
\email{valentin.schwinte@univ-lorraine.fr}

\author{Laurent Thomann }
\address{Universit\'e de Lorraine, CNRS, IECL, F-54000 Nancy, France}
\email{laurent.thomann@univ-lorraine.fr}

\title[On the stability of the Abrikosov lattice in the LLL]{On the stability of the Abrikosov lattice in the Lowest Landau Level}


\begin{abstract}
We study the Lowest Landau Level equation set on simply and doubly-periodic domains (in other words, rectangles and strips with appropriate boundary conditions). To begin with, we study well-posedness and establish the existence of stationary solutions. Then we investigate the linear stability of the lattice solution and prove it is stable for  the (hexagonal) Abrikosov lattice, but unstable for rectangular lattices. 
\end{abstract}

\maketitle

\tableofcontents

\renewcommand{\labelenumi}{(\roman{enumi}).}

\bigskip

\section{Introduction and main results}

\subsection{The equation}

Consider  the Lowest Landau Level equation
\begin{equation}\label{LLL}\tag{LLL}
\left\{
\begin{aligned}
 &   i \partial_t u = \Pi (|u|^2 u),\\
 &u(0,\cdot)=  u_0,
\end{aligned}
\right.
\end{equation}
   where $\Pi$ is the projector in $L^2(\C)$ on the Fock-Bargmann space $$\mathcal{E} =\Big \{\, u(z) =  \EE\big( -{|z|^2}/2\big) f(z)\;,\;f \; \mbox{entire\ holomorphic}\,\Big\}\cap L^2(\C )\ ,$$ which is given by the kernel
\begin{equation}\label{defPi}
[\Pi u](z) = \frac{1}{\pi} \EE\Big( -{|z|^2}/2\Big)  \int_\mathbb{C}  \EE\Big( \ov  w z - |w|^2/2\Big) u(w) \,dL(w),
\end{equation}
($L$ being simply the Lebesgue measure on $\C$). Notice that $\Pi $ extends as an operator from $\mathscr{S}'(\C )$ onto the space 
$$\widetilde {\mathcal E}=\big\{ u(z) =\EE\big( -{|z|^2}/2\big) f(z)\;,\;f \; \mbox{entire\ holomorphic}\; : \; \exists M ,\ \vert u(z)\vert \lesssim \< z\>^M \, \big\} \ ,$$
on which $\Pi $ is the identity operator. For this reason, we shall extend equation (\ref{LLL}) to $\widetilde {\mathcal E}$. 

\subsection{The Physics of \eqref{LLL}}

The reviews~\cite{Aftalion,Fetter} provide very good overviews of the physics of rotating Bose-Einstein condensates.
Different experiments showed that they exhibit triangular arrays of vortices~\cite{ARVK,MCWD,SCEMC}, known as Abrikosov lattices, which are stable and support oscillations known as Tkachenko waves~\cite{BWCB}. These patterns are interpreted as minimizers of the Gross-Pitaevskii energy to which a trapping term and a rotation term are added. It can be written as follows
$$
 E(u) = \int_{\mathbb{R}^2} \left[ \frac{1}{2} \left| (\nabla - i \Omega A) u \right|^2 + \frac{1}{2}(1-\Omega^2) |x|^2 |u|^2 + \frac{1}{2} |u|^4 \right] \,dx.
$$
Here $A = \begin{pmatrix} -x_2 \\ x_1 \end{pmatrix}$, $\Omega$ is the speed of rotation, and the other physical constants were scaled out. 

In order to make sense of these arrays, further simplifications are needed. One possibility is the Thomas-Fermi regime, examined in~\cite{FB}. Another possibility, which will occupy us here, is to consider the lowest Landau level regime, which arises under two physical assumptions. To start with, we require that $\Omega \to 1$, so that the second term in $E$ becomes negligible compared to the first one. Furthermore, we assume that the first term in $E$ is dominant compared to the third, and that the energy levels of $\Delta_A = \nabla_A \cdot \nabla_A$ with $\nabla_A=\nabla -i \Omega A$ are well separated. This means that states of low energy~$E$ will be in the ground state of $\Delta_A$.

This ground state is very degenerate; this is the lowest Landau level which is well-known in quantum Hall Physics~\cite{GJ,PericeThese}.   Let us    describe the lowest Landau level in the case $\Omega = 1$. Observe that, writing $z=x_1+ix_2$, 
$$
\left\| (\nabla - i A) u \right\|^2_{L^2} = 2 \| u \|_{L^2}^2 + \| (2 \partial_{\overline{z}} + z) u \|_{L^2}^2.
$$
 This formula implies that the ground state of $\Delta_A$ (when $\Omega=1)$ is the whole Fock-Bargmann space~$\mathcal{E}$.  Summarizing, we reduced the problem to the energy
$$
H(u) = \int_{\mathbb{R}^2} \Big[ |u|^4 + (1-\Omega^2) |z|^2|u|^2 \Big]\,dL(z), \qquad u \in \mathcal{E}.
$$
 Considering now the time-dependent problem, it turns out that the Hamiltonian flow given by $\dis \int |z|^2|u|^2$ is trivial, and that it commutes with the Hamiltonian flow given by $\dis \int |u|^4$, which is~\eqref{LLL}; thus the full time-dependent problem reduces to~\eqref{LLL}.

In the lowest Landau level regime, the energy has been considerably simplified, and the state space has a very rigid structure; this enabled theoretical progress on the distribution of vortices of minimizers in~\cite{ABD,Ho,WBP}. Tkachenko waves can also be considered in the lowest Landau level regime~\cite{Baym,Sonin}.

\medskip

It is in the context of superconductors that regular arrays of vortices~\cite{Abrikosov} and the Tkachenko waves they support~\cite{Tkachenko} were first observed and theoretically described. These systems are described by the Ginzburg-Landau equations of superconductivity.
Focusing on the wave function and neglecting the self-consistent magnetic field, one finds the same energy as above, also defined on the Bargmann-Fock space, but there is a major difference: the relevant time-dependent problem is not Hamiltonian, but rather a gradient flow.
For this reason, the results derived in the present paper apply to the Ginzburg-Landau equations as long as they only pertain to properties of $H$ on the space $\mathcal{E}$ (typically, its Hessian at critical points). But the results describing properties of the Hamiltonian system~\eqref{LLL} do not have a good counterpart in the context of time-dependent Ginzburg-Landau.

\subsection{Mathematical results}

The mathematical analysis of \eqref{LLL} was initiated in~\cite{ABN} for the time-independent and \cite{Nier1} for the time-dependent problem. These works were later extended in \cite{GGT, Sch}. In \cite{SchTho, Tho1}, systems of coupled \eqref{LLL} equations were studied and travelling waves were constructed, providing examples of solutions with growing Sobolev norms.

Many works were devoted to the analysis of rotating Bose-Einstein condensates \cite{AB,BR,CRY,Rou2,Rou1,RSY2}, studying the variational problem at the level of the Gross-Pitaevskii equation and also proving its convergence to the time-independent version of \eqref{LLL}.

The existence and stability of doubly periodic solutions to the Ginzburg-Landau equations has been studied in~\cite{Sigal-Tza1, Sigal-Tza2, Sigal-Tza3}. These doubly periodic solutions appear in a regime corresponding to the lowest Landau level reduction; but, as noticed above, the time-dependent problem is parabolic and hence very different from the Hamiltonian system \eqref{LLL}.

The equation~\eqref{LLL} also arises as the completely resonant system for the cubic nonlinear Schr\"odinger equation. The  completely resonant system was derived from the cubic (NLS) on the two-dimensional torus in  \cite{FGH}, and was studied in~\cite{GHT, GHT2}. An analogous derivation can be followed from the confined Gross-Pitaevskii equation~\cite{MeSpa,FMS}. Finally, the completely resonant system can also be interpreted as the modified scattering limit of the cubic nonlinear Schr\"odinger equation with partial harmonic confinement~\cite{H-T}. In all these works, vortex arrays have not been reported.

We also refer to \cite{GT, BBCE, BiBiCrEv,BiBiEv, Clerck-Evnin, Sch} for more results on~\eqref{LLL} and related equations.
 
\subsection{The Theta function and the Abrikosov lattice} We now turn to the mathematical description of the Abrikosov lattice, which relies on the Jacobi Theta functions.

Let $\tau=\tau_1+i \tau_2 \in \C$ with $\tau_2>0$.  The Jacobi Theta function on the lattice $ \mathcal{L}_\tau = \mathbb{Z}\oplus \tau \mathbb{Z}$ is defined (see \cite[Chapter V]{Cha}) by 
\begin{equation*}
\Theta_{\tau}(z)=-i \sum_{n=-\infty}^{+\infty} (-1)^n \EE\Big( {i\pi \tau (n+1/2)^2}+{i(2n+1)\pi z}\Big), \qquad z\in \C.
\end{equation*}
The $\Theta_{\tau}$ function vanishes exactly on  $\mathcal{L}_\tau$ and is such that 
\begin{equation}\label{star} 
\Theta_{\tau}(-z)=-\Theta_{\tau}(z),\quad \Theta_{\tau}(z+1)=-\Theta_{\tau}(z),\quad \Theta_{\tau}(z+\tau)=-e^{-i \pi \tau}e^{-i 2\pi z}\Theta_{\tau}(z).
\end{equation}
For more results on Theta functions, we refer to \cite[$\S$ 4]{Godement} (in this book, the notation $\theta_1(z, \tau)=\Theta_{\tau}(z)$ is used) and we refer as well to the book \cite{Mumford}.  We also mention the expository paper \cite{Nier2} of F. Nier making connections with various points of view. 

To every Theta function, we associate the function 
$$
\Phi_0(z) =  \EE\Big( \frac{z^2}2 - \frac{{i \pi}z}  {\gamma} - \frac{|z|^2}2\Big) \Theta_\tau \left( \frac{z}{\gamma} - \frac{\tau-1}{2} \right).
$$
For the choice of $\gamma^2 \tau_2 = \gamma^2 \mathfrak{Im} \, \tau= \pi$, this function is $L^\infty$ and belongs to $\widetilde {\mathcal E}$. Furthermore, it is a stationary solution of the \eqref{LLL} equation (see Remark \ref{rem14} below). Finally, its period is the lattice $\mathcal{L}_{\tau,\gamma}=\gamma (\mathbb{Z} + \tau \mathbb{Z})$; 
if $\tau = j = \EE\big({{2i \pi}/{3}}\big)$ and $\gamma^2 = {2\pi}/{\sqrt{3}}$, this period is given by the hexagonal array known as the Abrikosov lattice!

\medskip

Introduce now the \textit{magnetic translation} in the direction $\alpha \in \mathbb{C}$
$$
R_\alpha u (z) = \EE\Big(\frac{\alpha \ov{z} - \ov{\alpha} z}2\Big)   u(z+\alpha).
$$
For any $\alpha\in \C$, $R_\alpha$ is a symmetry of \eqref{LLL} (and in particular, it leaves $\mathcal{E}$ invariant).  The symmetries of the $\Theta_{\tau}$ function translate into the following identity for $\Phi_0$:
\begin{equation}
\label{symPhi}
R_{\gamma} \Phi_0 = R_{\gamma \tau} \Phi_0 = \Phi_0.
\end{equation}

\subsection{Results obtained}
The main theme of the present paper is the stability of vortex lattices under the flow of \eqref{LLL}. Since it is quite difficult to make progress on this question in full generality, we simplify the problem by considering perturbations which respect the symmetries \eqref{symPhi} of $\Phi_0$, resulting in doubly- or simply-periodic functions.

\medskip

In \textit{Section \ref{section2}}, we investigate the case of doubly periodic functions, namely elements of $\mathcal{E}$ which satisfy $R_\gamma u = R_{\gamma \tau} u = u$ with the necessary quantization condition $\gamma^2 \tau_2 \in \pi \mathbb{N}$. The equation \eqref{LLL} turns into a finite-dimensional Hamiltonian system, for which we propose a convenient coordinate system, and establish some elementary properties.

\medskip

In \textit{Section \ref{section3}}, we consider simply periodic functions, namely elements of $\mathcal{E}$ such that $R_\gamma u = u$; they can be thought of as being defined on a vertical strip of width~$\gamma$ and extended by periodicity. We develop the theory of \eqref{LLL} in this setting by establishing local well-posedness and characterizing stationary solutions which decay at infinity. We also exhibit a Hilbert basis 
$$
\psi_n(z) = R_{{in\pi}/{\gamma}} \psi_0 (z), \quad \mbox{with} \quad \psi_0(z) = \left( {2}/({\pi \gamma^2}) \right)^{1/4} \EE\big( z^2/2 -  |z|^2/2 \big)   
$$
for periodic functions in $\mathcal{E}$ which are square integrable on the strip. This basis provides a very natural coordinate system for \eqref{LLL}: the formulation is simple and stationary solutions are very easy to express.

\medskip

In \textit{Section \ref{rectangular_lattice}}, we consider the linearization of \eqref{LLL} around $\Phi_0$ with $\tau = {i \pi}/{\gamma^2}$, which corresponds to rectangular lattices. The linearized problem is given by
\begin{equation}
\label{linearizedpb}
i \partial_t v + \lambda v = \Pi \big[2 |\Phi_0|^2 v + \Phi_0^2 \overline{v}\big].
\end{equation}
We are able to completely analyze this problem by viewing it in the Hilbert basis $(\psi_n)$, identifying a convolution structure, and then taking the Fourier transform. This leads to the following result.

\begin{theorem}[Instability for rectangular lattices]
For any $\gamma>0$ and $\tau = {i \pi}/{\gamma^2}$, the linearized problem~\eqref{linearizedpb} is exponentially unstable.
\end{theorem}

We refer to Theorem \ref{thmunstable} for a precise statement.

\medskip

Finally, in \textit{Section \ref{section5}}, we turn to the linearization of \eqref{LLL} around Abrikosov lattices, namely~$\Phi_0$ with $\tau = \EE\big({{2i\pi}/{3}}\big)$. When viewed in the appropriate functional framework, this linearization is stable in $L^2$ and even exhibits decay in $L^\infty$. The approach is similar to the case of rectangular lattices: switching to the Hilbert basis, identifying a convolution structure, and taking the Fourier transform; but the analysis is more involved. 

\begin{theorem}[Stability for the Abrikosov lattice]
Consider a solution $v$ of the linearized problem~\eqref{linearizedpb} around the Abrikosov lattice, expand it in the Hilbert basis $(\psi_n)$ and take the Fourier transform of the coefficients to obtain $f(t,\xi)$, with $\xi \in [0,1]$, and define $g(t,\xi)=\ov{f}(t,-\xi)$. In particular, $f_0$ corresponds to the initial value $v(t=0)$ viewed through this transformation.
Then $L^2$ stability holds in the following form
$$
\left\| \frac{f+g}{\mu} \right\|_{L^2([0,1])} + \| f \|_{L^2([0,1])} \lesssim \left\| \frac{f_0+g_0}{\mu} \right\|_{L^2([0,1])} + \| f_0 \|_{L^2([0,1])}
$$
while $L^\infty$ decay can be captured as follows
$$
\| v(t) \|_{L^\infty(\C)} \lesssim \frac{1}{t^{1/3}} \left( \left\| \frac{f_0+g_0}{\mu} \right\|_{H^1([0,1])} + \| f_0 \|_{H^1([0,1])} \right),
$$
 where $\mu$ is a smooth 1-periodic function such that $\mu(\xi) \sim c \xi^2$ when $\xi \to 0$. 
\end{theorem}

We refer to Theorem \ref{thmL2stable} and Theorem \ref{decay_linf} for more precise statements.

\subsection{Perspectives}

Preliminary computations seem to indicate that the Abrikosov lattice ($\gamma^2 = {2\pi}/{\sqrt 3}$, $\tau =j$) is the only stable lattice, at least when perturbations are restricted to simply periodic functions, but it is still an open question. We refer to Remark \ref{rk_tau_general} for a short discussion of this question.
 
It seems very natural to try and prove nonlinear stability by building upon the decay proved for the linearized problem around the Abrikosov lattice. However, the rather weak decay rate $t^{-1/3}$, which is optimal, makes this problem quite challenging.

Finally, the stability of the Abrikosov lattice, both at the linearized and at the nonlinear level, remains an outstanding problem.

\subsection*{Acknowledgements} Pierre Germain was supported by a Wolfson fellowship from the Royal Society and the Simons collaboration on Wave Turbulence. Valentin Schwinte was supported by the grant DrEAM from Universit\'e de Lorraine. Laurent Thomann was supported by the ANR project ``SMOOTH'' ANR-22-CE40-0017.

 \section{(LLL) on doubly periodic domains (cells)}
 \label{section2}
 
In this section, we construct solutions to \eqref{LLL} such that $\vert u\vert$ is doubly periodic. We will rely on results of Aftalion-Serfaty~\cite[Section 3]{AftaSerfa}. In the sequel, we work with the general lattice 
$$\mathcal{L}_{\tau,\gamma} = \gamma(\mathbb{Z}\oplus \tau \mathbb{Z})$$
 with $\gamma>0 $, and we consider the space
\begin{multline*}
\mathcal{E}_{\tau,\gamma}=\big\{u\in \widetilde{\mathcal{E}}: R_\gamma u=u,\quad   R_{\gamma \tau} u=u\big\}\\
=\Big\{u\in \widetilde{\mathcal{E}}: u(z+\gamma)=\EE\Big( {\frac{\gamma}2(z-\ov{z})}\Big) u(z), \\
  u(z+\gamma\tau)=\EE\Big( \frac{\gamma}2 \big(\tau_1(z-\ov{z})-i\tau_2(z+\ov{z})\big)\Big) u(z)=\EE\Big( \frac{\gamma}{2}(\ov{\tau}z-\tau\ov{z})\Big) u(z)\Big\},
\end{multline*}
so that in particular for all $u \in \mathcal{E}_{\tau,\gamma}$ and $z\in\mathbb{C}$, $\vert u(z+\gamma)\vert= \vert u(z+\gamma\tau)\vert=\vert u(z)\vert$. \medskip

We define the fundamental cell of $\mathcal{L}_{\tau,\gamma}$ by
\begin{equation*}
    K_{\tau,\gamma}=\big\{z=\gamma(r_{1}+r_{2}\tau),\quad r_{1},r_{2}\in[0,1]\big\}.
\end{equation*}

\subsection{Multiplicative description of $\E_{\tau,\gamma}$}

\begin{prop}[\cite{AftaSerfa}]\label{prop41}
Let $\gamma>0$ and $\tau \in \C$ with $\tau_2=\Im \tau >0$.
    \begin{enumerate}[$(i)$]
\item     If $\gamma^2 \tau_2 \notin\pi \mathbb{N}$, then $\mathcal{E}_{\tau,\gamma} = \{ 0 \}$.
 \item   If $\gamma^2 \tau_2 = \pi N$ for some $N \in \mathbb{N}^*$, then $\mathcal{E}_{\tau,\gamma}$ is a complex vector space of dimension~$N$. It can be described as the set of functions vanishing $N$ times modulo~$\mathcal{L}_{\tau,\gamma}$ (counting multiplicity), which can be written under the form
    \begin{equation}\label{formprop}
        u(z)=\lambda \EE\Big(  z^2/2  + b z-   |z|^2/2\Big) \prod_{k=1}^N\Theta_{\tau}\Big(({z-z_k})/{\gamma} \Big)
    \end{equation}
where $\lambda \in \C$ and $(z_k)_{1\leq k\leq N}\in \C$ are representatives of the zeroes of $u$ modulo~$\mathcal{L}_{\tau,\gamma}$ and finally $b \in \mathbb{C}$ and the $(z_k)$ satisfy the relations
\begin{align}\label{somme}
& \gamma b = i(- N + 2j) \pi \nonumber \\
&\sum_{k=1}^N z_k = \frac{\gamma}{2} (\tau - 1) N - j \tau \gamma + \ell \gamma
\end{align}
for some $j,\ell \in \mathbb{Z}$.
        \end{enumerate}
\end{prop}

A few remarks are in order:
\begin{itemize}
\item This theorem appears in~\cite{AftaSerfa} with a minor misprint in the equation for $b$; we include the proof, which follows that in~\cite{AftaSerfa}, in Appendix \ref{appendix-B}.
\item Up to selecting the representatives modulo $\mathcal{L}_{\tau,\gamma}$ of the zeroes of $u$, we can ensure that $\dis \sum_{k=1}^N z_k = {\gamma N}(\tau-1)/2+\gamma \ell$, or in other words $j=0$.
\item The condition on $\tau$ can be understood in the following way: let $u \in \mathcal{E}_{\tau,\gamma}$, $u \neq 0$. Then clearly $\dis R_\gamma R_{\gamma \tau} u= R_{\gamma \tau}R_\gamma u =u$. Next, from \eqref{commut} we obtain the quantization condition $\gamma^2 \tau_2 \in\pi \mathbb{N}$.
\item Geometrically, this theorem is stating that $\mathcal{E}_{\tau,\gamma}$ is non empty when  the fundamental cell has area $N\pi$, and elements of the space vanish $N$ times on it.
\item By Lemma \ref{lem.transla},  for $N=1$, $\alpha=\beta=\pi$ in \eqref{def-d}, $\nu=\gamma^{-1}$, $h=1$ we recover the function defined in \cite[Proposition~4.1]{ABN}.
\end{itemize}

In the sequel we assume that 
\begin{equation*}
    \tau=\tau_1+\frac{i\pi N}{\gamma^2}\quad \mbox{for some}\quad N \in \N.
    \end{equation*}
    
\subsection{Additive description of $\E_{\tau,\gamma}$}

    For $0\leq k\leq N-1$  define
    \begin{multline*}
\mathcal{E}_{\tau,\gamma,k}=\Big\{u\in \widetilde{\mathcal{E}}: R_\gamma u=u,\quad   R_{{\gamma \tau}/N} u= \EE\big( {-{2ik\pi}/N}\big) u\Big\}\\[5pt]
=\Big\{u\in \widetilde{\mathcal{E}}: u(z+\gamma)=\EE\Big(\frac{\gamma}2(z-\ov{z})\Big) u(z),\quad \qquad \\[4pt]
  u(z+\frac{\gamma \tau}{N})= e^{-{2ik\pi}/N}  \EE\Big(\frac{\gamma}{2N}   (\ov{\tau}z-\tau\ov{z})\Big)u(z)\Big\}.
\end{multline*}

We observe immediately that
$$
\mathcal{E}_{\tau,\gamma,k} \subset \mathcal{E}_{\tau,\gamma},
$$
as follows by iterating the periodicity condition in the direction $\tau$ and using~\eqref{propRalpha}.
Define next
\begin{equation}\label{defPhi0}
 \Phi_{0}(z)=\EE\Big( {\frac{1}{2} z^2 -\frac{i\pi}{\gamma} z - \frac{1}{2} |z|^2} \Big) \Theta_{{\tau}/{N}}\Big((z-z_{0})/{\gamma}\Big) , \quad z_0=\frac{\gamma}2(\frac{\tau}N-1),
 \end{equation}
 which satisfies  $R_\gamma  \Phi_{0}= \Phi_{0}$ and    $R_{{\gamma \tau}/N}  \Phi_{0}= \Phi_{0}$. Then 
    for $0\leq k\leq N-1$  set
\begin{equation*} 
\Phi_{k}(z)=R_{{k \gamma}/N} \Phi_{0}(z)=e^{-{ik\pi}/N}\EE\Big( {\frac{1}{2} z^2 -\frac{i\pi}{\gamma} z - \frac{1}{2} |z|^2} \Big)\Theta_{{\tau}/{N}}\Big((z-z_{k})/\gamma\Big),
 \end{equation*}
 \begin{equation*} 
z_k=\frac{\gamma}2(\frac{\tau}N-1)-\frac{k}{N}\gamma,
 \end{equation*}
so that $R_\gamma  \Phi_{k}= \Phi_{k}$ and    $R_{{\gamma \tau}/N}  \Phi_{k}=\EE\big({-{2ik\pi}/N}\big) \Phi_{k}$. \medskip

The $\mathcal{E}_{\tau,\gamma,k}$ turn out to provide an orthogonal decomposition of $\mathcal{E}_{\tau,\gamma}$, as is stated in the following proposition; this result appeared in \cite{Nier2} in a slightly different guise.  See also \cite{Perice,Nguyen-Rougerie} for more results in this direction.

\begin{prop} \label{prop72} Recall that $\gamma > 0$ and $\gamma^2 \tau_2 = \pi N$ for some $N \in \mathbb{N}^*$.
    \begin{enumerate}[$(i)$]
        \item For any $k$, the space $\mathcal{E}_{\tau,\gamma,k}$ is generated by $\Phi_k$
 \begin{equation*}
 \mathcal{E}_{\tau,\gamma,k}= \operatorname{span}_{\C} \big\{\Phi_{k}\big\},
 \end{equation*}      
 and the space $\mathcal{E}_{\tau,\gamma}$ is the direct sum of the $\mathcal{E}_{\tau,\gamma,k}$
 \begin{equation*}
 \mathcal{E}_{\tau,\gamma}=\bigoplus_{k=0}^{N-1} \mathcal{E}_{\tau,\gamma,k}.
 \end{equation*}
 \item The family $(\Phi_{k})$ is  orthogonal  in $L^{2}(K_{\tau, \gamma})$ and furthermore for any $k \in \Z$,
$$
\int_{K_{\tau, \gamma}}|\Phi_{k}(z)|^{2} \, dL(z)=\gamma N \sqrt{\frac{\pi}{ {2}}} \EE\Big( \frac{\pi^2}{2\gamma^2}\Big).
$$
\item For any $k \in \Z$, 
\begin{equation*}
\int_{K_{\tau, \gamma}} |\Phi_{k}(z)|^4 \, dL(z)= \frac{N\gamma^2}2\EE\Big( \frac{\pi^2}{ \gamma^2}\Big) \sum_{j , \ell\in \Z}\EE\Big( {-\gamma^2 |j\frac{\tau}N-\ell|^{2}}\Big).
\end{equation*}
\end{enumerate}
\end{prop} 
  
\begin{proof}
$(i)$ The proof follows closely that of Proposition \ref{prop41}, given in Appendix \ref{App1}. Consider $u\in \mathcal{E}_{\tau,\gamma,k}$. Let $\{z_j\}_{1 \leq j \leq p}$ be the zeros of $u$ in the fundamental cell $K_{{\tau}/{N},\gamma}$ of the lattice $\mathcal{L}_{{\tau}/{N},\gamma}$. With the second periodicity condition, the $z_{j,\ell} = z_j + {\ell \gamma \tau}/{N}$ are the zeroes of $u$ in the fundamental cell $K_{\tau,\gamma}$ of the lattice $\mathcal{L}_{\tau,\gamma}$, so that~$u$ has $pN$ zeroes in $K_{\tau,\gamma}$. From Proposition~\ref{prop41} and $\mathcal{E}_{\tau,\gamma,k} \subset \mathcal{E}_{\tau,\gamma},$ we get $p=1$. Then $u$ has only one zero in $K_{{\tau}/{N},\gamma}$ and we write as in \eqref{eqA1}, 
    $$ u(z) = \lambda \EE \Big( {-\frac{|z|^2}{2} + \alpha z^2 + \beta z}\Big)  \Theta_{{\tau}/{N}}\Big((z-z_{1})/\gamma\Big), $$
where $\alpha,\beta, \lambda \in \mathbb{C}$. We take for simplicity $\lambda =1$ in the following. The first periodicity condition $R_\gamma u = u$ of~$\mathcal{E}_{\tau,\gamma,k}$ requires that $\alpha = {1}/{2}$ and $\beta = -{i\pi}/{\gamma} + {2i\ell\pi}/{\gamma}$, with $\ell \in \mathbb{Z}$.
The second periodicity condition of $\mathcal{E}_{\tau,\gamma,k}$ differs from Proposition \ref{prop41}: the relation contains a $R_{{\gamma\tau}/{N}}$ magnetic translation and an additional phase factor $\EE\big({-{2ik\pi}/{N}}\big)$. Since $\tau_2 = {\pi N}/{\gamma^2},$ the coefficient on $z$ in this periodicity condition already matches, and gives no further information. Nevertheless, by matching the constant terms, one gets 
$$ z_1 = \frac{\gamma}{2}\pa{\frac{\tau}{N}-1} - \frac{k}{N}\gamma + \pa{L-\ell\frac{\tau}{N}}\gamma = z_k + \pa{L-\ell\frac{\tau}{N}}\gamma,$$
where $L \in \Z$. Taking~\eqref{star} into account to get rid of $ \pa{L-\ell {\tau}/{N}}\gamma$, we obtain a multiplicative factor $\lambda \EE\big( {-{2i\pi \ell z}/{\gamma}}\big)$, where $\lambda \in \C$ is a constant. Overall, $$ \mathcal{E}_{\tau,\gamma,k} \subset \text{span}_\C\{\Phi_k\}.$$
The reverse inclusion is easily obtained from the definition of $\Phi_k$. 
It follows from the periodicity condition that the spaces $\mathcal{E}_{\tau,\gamma,k}$ are in direct sum; and this direct sum equals $\mathcal{E}_{\tau,\gamma}$ by comparing the dimensions.\medskip

\noindent $(ii)$ Denoting $K^1_{\tau,\gamma}=\big\{z=\gamma\big(r_{1}+r_{2}{\tau}/{N}\big),\;\, r_{1},r_{2}\in[0,1]\big\}$, for $0 \leq k, \ell \leq N-1$, we have
\begin{eqnarray*} 
 \int_{K_{\tau,\gamma}}\Phi_{k}(z)\ov{\Phi_{\ell}(z)} \, dL(z)&=& \sum_{j=0}^{N-1}\int_{\substack{z=\gamma(r_{1}+r_{2}{\tau}/N)\\0\leq r_{1}\leq1,j\leq r_{2}\leq j+1}}\Phi_{k}(z)\ov{\Phi_{\ell}(z)} \, dL(z) \\ 
&   = &\sum_{j=0}^{N-1}\int_{K^1_{\tau,\gamma}}R_{{\gamma \tau j}/{N}}\Phi_{k}(z)\ov{R_{{\gamma \tau j }/{N}}\Phi_{\ell}(z)} \, dL(z) \\
&=& \big(\sum_{j=0}^{N-1}e^{-{2ij(k-\ell)\pi}/ N}\big)\int_{K^1_{\tau,\gamma}}\Phi_{k}(z)\ov{\Phi_{\ell}(z)}\, dL(z)\\
&   =& \delta_{k,\ell}N\int_{K^1_{\tau,\gamma}}|\Phi_{k}(z)|^{2}dL(z).
\end{eqnarray*}
To compute the above right-hand side, matters reduce to the case $k=0$, by periodicity and since $\Phi_{k}=R_{{k \gamma}/ N} \Phi_{0}$. Next, we observe that 
$$R_{z_0} \Phi_0(z) = \EE \Big( {\frac{1}{2} z_0^2- \frac{i\pi}{\gamma} z_0- \frac{1}{2} |z_0|^2} \Big)\EE\Big( {\frac{1}{2} z^2 - \frac{1}{2} |z|^2}\Big) \Theta_{{\tau}/{N}}\big(\frac{z}{\gamma}\big).$$
Still using periodicity, and also that $\mathfrak{Re} \big(  z_0^2/2  - {i\pi} z_0/ {\gamma} -   |z_0|^2/2 \big) = {\pi^2}/({4 \gamma^2})$, 
\begin{eqnarray*}
\int_{K^1_{\tau, \gamma}} |\Phi_{k}(z)|^2dL(z) &=& \int_{K^1_{\tau, \gamma}} |\Phi_{0}(z)|^2dL(z) \\
& =&  \EE\Big( \frac{\pi^2}{2\gamma^2}\Big)    \int_{K^1_{\tau, \gamma}} \big| \EE\Big( { z^2/2 -   |z|^2/2}\Big)  \Theta_{{\tau}/{N}}\big(\frac{z}{\gamma}\big)\big|^2dL(z).
\end{eqnarray*}
From \cite[Last equality on page 681]{ABN}, we have 
$$\int_{K^1_{\tau,\gamma}} |\Phi_{k}(z)|^2dL(z)  =  \gamma  \sqrt{\frac{\pi}{ {2}}}  \EE\Big( \frac{\pi^2}{2\gamma^2}\Big).$$
\medskip

\noindent $(iii)$ Arguing as in item $(ii)$, we obtain
\begin{align*}
\int_{K_{\tau, \gamma}} |\Phi_{k}(z)|^4dL(z) &=N \int_{K^1_{\tau, \gamma}} |\Phi_{k}(z)|^4dL(z) = N \int_{K^1_{\tau, \gamma}} |\Phi_{0}(z)|^4dL(z) \\
& = N \EE\Big( \frac{\pi^2}{\gamma^2}\Big)    \int_{K^1_{\tau,\gamma}} \big| \EE\Big( { z^2/2 -   |z|^2/2}\Big)     \Theta_{{\tau}/{N}}\big(\frac{z}{\gamma}\big)\big|^4dL(z),
\end{align*}
and from \cite[First equality on page 682]{ABN}, we have 
$$\int_{K^1_{\tau,\gamma}} |\Phi_{k}(z)|^4dL(z)  = \frac{\gamma^2}2\EE\Big( \frac{\pi^2}{\gamma^2}\Big) \sum_{j , \ell\in \Z}\EE\Big( {-\gamma^2 |j\frac{\tau}N-\ell|^{2}}\Big),$$
as claimed.
\end{proof}

In the work \cite{ABN}, the authors consider functions $v \in  \widetilde{\mathcal{E}}$ such that $| v(z+\gamma)|=  | v(z+\gamma\tau)|= |v(z)|$. The next result, which is proved in \cite[Paragraph~3]{Nier2}, shows that we can reduce to the study of the space $\mathcal{E}_{\tau,\gamma}$.

\begin{lemme}\label{lem.transla}
Let $\gamma>0$ and $\tau \in \C$ such that $\tau_2=\Im \tau>0$. Assume that $v \in \widetilde{\mathcal{E}}$ and $v \not \equiv 0$ satisfies
$$| v(z+\gamma)|=  | v(z+\gamma\tau)|= |v(z)|.$$
Then there exists $\delta \in \C$ such that $u =R_{\delta} v \in \mathcal{E}_{\tau,\gamma}$.
\end{lemme}

\begin{proof}
We have for all $z \in \C$, $|R_{\gamma}u(z)|=|u(z)|$. Thus there exists $\alpha(z)\in  \R$ such that $R_{\gamma}u(z)=\EE\big(i \alpha(z)\big)u(z)$. Since $\alpha$ is  entire, $\alpha$ is constant. Similarly, there exists $\beta \in \R$ such that $R_{\gamma \tau}u(z)=\EE\big( {i \beta}\big) u(z)$. Now  apply $R_{\gamma \tau}$ to the first relation and $R_{\gamma}$ to the second, by \eqref{commut} we get that there exists $N \in \N$ such that $\tau_2={\pi N}/{\gamma^2}$. Set 
\begin{equation}\label{def-d}
\delta=\frac{\gamma}{2N\pi}(\beta-\alpha \tau_1)-i \frac{\alpha}{2 \gamma}=\frac{\gamma}{2\pi N} \big( \beta -\alpha   \tau\big),
\end{equation}
then $u=R_{\delta}v\in \mathcal{E}_{\tau,\gamma}$.
\end{proof}

\subsection{The nonlinear term}

The aim of this subsection is to understand the structure of the nonlinear term $\Pi \big(|u|^2 u\big)$ in the space $\mathcal{E}_{\tau,\gamma}$, more specifically in the basis provided by the $\Phi_k$. The first step is to realize that $\Pi$ can be interpreted as an orthogonal projector, as will now be explained. Define the space
\begin{multline*}
    \mathcal{F}_{\tau,\gamma}=\Big\{u\in L^2 (K_{\tau,\gamma}), \;\; u(z+\gamma)=\EE\Big( {\frac{\gamma}2(z-\ov{z})}\Big) u(z),\;\; \\
     u(z+\gamma\tau) =\EE\Big(  \frac{\gamma}2(\ov{\tau}z-\tau\ov{z})\Big) u(z)\Big\},
\end{multline*}
so that
$$
\mathcal{E}_{\tau,\gamma} =  \mathcal{F}_{\tau,\gamma} \cap \widetilde{\mathcal{E}}.
$$

\begin{lemme} 
\label{lemmaproj}
    For all $u,v\in \mathcal{F}_{\tau,\gamma}$
    \begin{equation*}
        \int_{K_{\tau,\gamma}} \Pi u(z) \ov{v(z)}dL(z)= \int_{K_{\tau,\gamma}}  u(z) \ov{\Pi(v(z))}dL(z).
    \end{equation*} 
\end{lemme}

\begin{proof} With the change of variable $\xi=w-k\tau \gamma-\ell \gamma$ and the fact that 
$$u(\xi+k\tau \gamma+\ell \gamma)=(-1)^{k \ell N}\EE\Big({\ell \gamma(\xi-\ov\xi)/2+k\gamma(\ov{\tau}\xi-\tau\ov\xi)/2}\Big)u(\xi)$$
we get
\begin{align*}
\Pi u(z)&=\frac1{\pi}\EE\big({-|z|^{2}/2}\big)\int_{\C}\EE\Big(z\ov{w}-|w|^{2}/2\Big)u(w) \, dL(w)\\
&=\frac1{\pi}\EE\big({-|z|^{2}/2}\big)\sum_{k,\ell\in \Z}\int_{w\in K_{\tau,\gamma}+k\tau \gamma+\ell \gamma}\EE\Big(z\ov{w}-|w|^{2}/2\Big)u(w) \, dL(w)\\
&=\frac1{\pi}\int_{  K_{\tau,\gamma}}\sum_{k,\ell\in \Z}(-1)^{k \ell N} \EE\big({-|z|^{2}/2}\big) \EE\Big({-{\gamma^2}|k\tau+\ell|^{2}/2+(k\ov\tau+\ell)\gamma z}\Big)\\
&\hspace{4cm}\times \EE\Big({(z-k\gamma \tau-\ell\gamma)\ov{w}-|w|^{2}/2}\Big)u(w) \, dL(w).
\end{align*}
Then 
\begin{multline*}
 \int_{K_{\tau,\gamma}} \Pi u(z) \ov{v(z)}\,dL(z)=\\
 \begin{aligned}
 &=\frac1{\pi}\int_{  K_{\tau,\gamma}\times K_{\tau,\gamma}}\sum_{k,\ell\in \Z}(-1)^{k \ell N}\EE\big({-|z|^{2}/2}\big) \EE\Big({-{\gamma^2}|k\tau+\ell|^{2}/2+(k\ov\tau+\ell)\gamma z}\Big)\\
&\hspace{3.5cm}\times \EE\Big({\big(z-k\gamma \tau-\ell\gamma\big)\ov{w}-|w|^{2}/2}\Big)u(w)\ov{v(z)}\,dL(w) \, dL(z)\\
&=\int_{K_{\tau,\gamma}}  u(w) \ov{\Pi(v(w))}\,dL(w)
 \end{aligned}
\end{multline*}
hence the result.
\end{proof}

Given an element $u$ of $\mathcal{F}_{\tau,\gamma}$, it can be thought of as a function on $\mathbb{C}$, or as a function on $K_{\tau,\gamma}$. Therefore we introduce the following notations.

\begin{df}
Given an element $u$ of $\mathcal{F}_{\tau,\gamma}$, we write $u^{ext}$ (extended), if $u$ is considered as  a function on $\mathbb{C}$ and $u^{res}$ (restricted) if $u$    is considered as  a function on~$K_{\tau,\gamma}$.   
\end{df}

With these notations, we have the following result:

\begin{lemme} \label{lemmaproj2}
The projector $\Pi$ on $\mathcal{F}_{\tau,\gamma}$ can be interpreted as the orthogonal projector~$\Pi'$ in $L^2(K_{\tau,\gamma})$ on $\mathcal{E}_{\tau,\gamma}$. In other words,
$$
(\Pi u^{ext})\big|_{K_{\tau,\gamma}} = \Pi' u^{res}.
$$
\end{lemme}

\begin{proof} On the one hand, if $u \in \mathcal{E}_{\tau,\gamma}$, then both $\Pi$ and $\Pi'$ act as the identity. It is clear as far as $\Pi'$ is concerned. Turning to $\Pi$, we have that $\Pi u^{ext} \in \widetilde{\mathcal{E}}$ and $\Pi u^{ext} \in \mathcal{F}_{\tau,\gamma}$ (since $R_\alpha$ and $\Pi$ commute), hence $\Pi u$ belongs to $\widetilde{\mathcal{E}} \cap \mathcal{F}_{\tau,\gamma} = \mathcal{E}_{\tau,\gamma}$. Furthermore, by the previous lemma, $\langle \Pi u^{ext}, v \rangle_{L^2(K_{\tau,\gamma})} = \langle u^{res}, v \rangle_{L^2(K_{\tau,\gamma})}$ for any $v \in \mathcal{E}_{\tau,\gamma}$, hence $u = \Pi u$.

On the other hand, choose $u$ in the orthogonal set to $\mathcal{E}_{\tau,\gamma}$ in $L^2(K_{\tau,\gamma})$. By definition of $\Pi'$, we get $\Pi' u^{res} = 0$. Turning to $\Pi u^{ext}$, it is in $\mathcal{E}_{\tau,\gamma}$ by the same argument as above, and satisfies $\langle \Pi u^{ext}, v \rangle_{L^2(K_{\tau,\gamma})} = \langle u^{res}, v \rangle_{L^2(K_{\tau,\gamma})}=0$ for any $v \in \mathcal{E}_{\tau,\gamma}$; hence $\Pi u^{ext} = 0$.
\end{proof}

By the previous lemma, we will identify henceforth $\Pi$ and $\Pi'$, and we will also ignore the distinction between $u^{ext}$ and $u^{res}$.

\begin{prop}
If $u_j \in \mathcal{E}_{\tau,\gamma,j}$ for $j=k, \ell, m$, then $\Pi(u_k \ov{u_{\ell}} {u_m})\in \mathcal{E}_{\tau,\gamma,n} $ with $n=k-\ell+m$ mod $N$.
Furthermore,
$$
\Pi (\Phi_{k} \ov{\Phi_{\ell}} \Phi_{m}) = \lambda \Phi_{n} \quad \mbox{with} \quad \lambda = \frac{1}{\| \Phi_0 \|_{L^2}^2} \int_{K_{\tau, \gamma}} \Phi_{k}(z) \ov{\Phi_{\ell}(z)} \Phi_{m}(z)\ov{\Phi_{n}(z)} \,dL(z).
$$
\end{prop}

\begin{proof} By~\eqref{propRalpha}, for all $\alpha \in \C$, 
$$R_\alpha\big( \Pi(u_k \ov{u_\ell} {u_m})\big) =\Pi(R_\alpha u_k \ov{R_\alpha u_\ell} {R_\alpha u_m}).$$
Thus, 
\begin{eqnarray*}
R_{{\gamma \tau}/{N}} \Pi(u_k \ov{u_\ell} {u_m})&=& \Pi\big(R_{{\gamma \tau}/{N}} u_k \ov{ R_{{\gamma \tau}/{N}} u_\ell }{R_{{\gamma \tau}/{N}} u_m}\big)\\
&=& e^{-{2i(k-\ell+m)\pi}/N}\Pi\big(u_k \ov{u_\ell} {u_m}\big),
\end{eqnarray*}
which gives the first assertion. The formula for $\lambda$ follows from the fact that $(\Phi_k)$ is an orthogonal basis and Lemma \ref{lemmaproj}.
\end{proof}

\subsection{Dynamical consequences}

    \begin{prop}\label{prop34}
Recall that $\gamma > 0$ and $\gamma^2 \tau_2 = \pi N$ for some $N \in \mathbb{N}^*$. Then, for all $\kappa \in \C$ and all $k \in \C$, the function
$$ u(t,z) = \kappa  \EE\big( {-i\lambda_0 |\kappa|^2t}\big)\Phi_k(z)$$
is a stationary solution of \eqref{LLL}, where 
\begin{equation}\label{def-lambda}
\lambda_0=\frac{\int_{K_{\tau, \gamma}} |\Phi_{k}(z)|^4dL(z)}{\int_{K_{\tau, \gamma}} |\Phi_{k}(z)|^2dL(z)}= \frac{\gamma}{\sqrt{2\pi}}  \EE\Big( \frac{\pi^2}{2\gamma^2}\Big)\sum_{j , \ell\in \Z}\EE\Big( {-\gamma^2 |j\frac{\tau}N-\ell|^{2}}\Big).
\end{equation}
In particular,
\begin{itemize}
\item  If $N=1$ and $\tau={i \pi }/{\gamma^2}$, which corresponds to the rectangular lattice, then
\begin{equation}\label{lambda-carre}
\lambda_0= \frac{1}{\sqrt{2}} \EE\Big( \frac{\pi^2}{2\gamma^2}\Big)\Big(\sum_{q\in\Z}\EE\big({-\frac{\pi^2q^2}{\gamma^2}}\big)\Big)^2.  
\end{equation}
\item If $N=1$, $\tau=\EE\big({{2 i\pi}/3}\big)$ and $\gamma={\sqrt{2\pi}}/ {3^{1/4}}$, which corresponds to the hexagonal lattice, then
\begin{equation}\label{lambda-hexa}
\lambda_0=\frac{1}{\sqrt{2}} \EE\Big( \frac{\pi^2}{2\gamma^2}\Big)\big( I^2+2IJ-J^2\big),
\end{equation}
with 
$$I =  \sum_{j \in \Z}\EE\Big(-  \frac{\pi^2(2j)^2}{\gamma^2}  \Big) \quad\mbox{and} \quad J=\sum_{j \in \Z}\EE\Big(-  \frac{\pi^2(2j+1)^2}{\gamma^2}  \Big).$$
\end{itemize}
\end{prop} 

\begin{proof}
The result of Proposition \ref{prop34} directly follows from Proposition \ref{prop72} $(ii)$ and~$(iii)$. Let us detail the cases of interest.  

\medskip

$\bullet$ If $N=1$ and $\tau={i \pi }/{\gamma^2}$, then by \eqref{def-lambda}
\begin{eqnarray*}
 \sum_{j , \ell\in \Z}\EE\Big({-\gamma^2 |j{\tau}-\ell|^{2}}\Big)&=& \sum_{j , \ell\in \Z}\EE\Big({-\gamma^2 (\frac{j^2  \pi^2}{\gamma^4}+ \ell^2)}\Big) \\
 &=& \Big(\sum_{j\in\Z}\EE\big(-\frac{\pi^2j^2}{\gamma^2}\big)\Big) \Big(\sum_{\ell\in\Z}\EE\big(- \gamma^2 \ell^2\big)\Big).
 \end{eqnarray*}
By the Poisson summation formula \eqref{PF1} with $\alpha=\gamma^2$ and $z=0$, 
$$\sum_{\ell\in\Z}\EE\big(- \gamma^2 \ell^2\big)=\frac{\sqrt{\pi}}\gamma  \sum_{\ell\in\Z}\EE\big(-\frac{\pi^2\ell^2}{\gamma^2}\big),$$
which implies \eqref{lambda-carre}.

\medskip

$\bullet$ If $N=1$, $\tau= \EE\big( {2 i\pi/3}\big)=-1/2+ i \sqrt{3}/2$ and $\gamma={\sqrt{2\pi}}/{3^{1/4}}$, then 
$$A:=  \sum_{j , \ell\in \Z}\EE\big(-\gamma^2 |j{\tau}-\ell|^{2}\big)= \sum_{j , \ell\in \Z}\EE\Big(-\gamma^2 \big((\ell+\frac12 j)^2+ \frac 34 j^2\big)\Big) $$
From \eqref{PF1} with $\alpha =\gamma^2$ and $z=j/2$, we deduce that 
$$ \sum_{ \ell\in \Z}\EE\Big(-\gamma^2 (\ell+\frac12 j)^2\Big) =\frac{\sqrt{\pi}}\gamma \sum_{ n\in \Z}\EE\Big(-\frac{\pi^2n^2}{\gamma^2}+i \pi n j\Big). $$
On the other hand, using that $3 \gamma^2/4= {\pi^2}/{\gamma^2}$, we find 
\begin{eqnarray*}
 A &=& \frac{\sqrt{\pi}}\gamma   \sum_{j \in \Z}\Big(\EE\big(-  \frac{\pi^2j^2}{\gamma^2 }\big)     \sum_{ n\in \Z}\EE\big(-\frac{\pi^2n^2}{\gamma^2}+i \pi n j \big)\Big) \\
 &=& \frac{\sqrt{\pi}}\gamma  \big(I(I+J)+J(I-J)\big) \\
 &=& \frac{\sqrt{\pi}}\gamma  \big(I^2+2IJ-J^2\big),
\end{eqnarray*}
which was the claim.
\end{proof}

    \begin{theorem}\label{prop42}
    The equation   \eqref{LLL} is globally well-posed in $ \mathcal{E}_{\tau,\gamma}$: for all $u_0 \in  \mathcal{E}_{\tau,\gamma}$ there exists a unique $u \in \mathcal{C}^{\infty}(\R ; \mathcal{E}_{\tau,\gamma})$ which depends smoothly on $u_0$. Moreover, the following quantities are conserved: for all $t\in \R$
    \begin{equation*}
\int_{K_{\tau,\gamma}}|u(t,z)|^{2}dL(z)= \int_{K_{\tau,\gamma}}|u_0(z)|^{2}dL(z), 
\end{equation*}
    and 
        \begin{equation*}
\int_{K_{\tau,\gamma}}|u(t,z)|^{4}dL(z)=\int_{K_{\tau,\gamma}}|u_0(z)|^{4}dL(z).
\end{equation*}
  Furthermore, the   solution can be represented in any of the following forms:
  \begin{enumerate}[$(i)$]
    \item Assume that $(z_{j,0})_{1\leq j\leq N}$ satisfy \eqref{somme} for some $k, \ell \in \Z$, then  the solution of~\eqref{LLL} with  data 
    $$
    u_0(z) = \mu_0 \EE\Big(  \frac12 z^2  +\frac{i\pi  } {\gamma} (-N+2 k ) z  -  \frac12   |z|^2\Big)  \prod_{j=1}^N\Theta_{\tau}\Big(\frac{1}{\gamma}(z-z_{j,0})\Big),
    $$
    takes the form
    $$
    u(t,z) = \mu(t) \EE\Big(  \frac12 z^2  +\frac{i\pi  } {\gamma} (-N+2 k ) z  -  \frac12   |z|^2\Big)\prod_{j=1}^N\Theta_{\tau}\Big(\frac{1}{\gamma}(z-z_{j}(t))\Big),
    $$
    and we have for all $t \in \R$
    \begin{equation}\label{somme1}
        \sum_{j=1}^Nz_j(t)=\sum_{j=1}^Nz_{j,0} = \frac{\gamma}{2} (\tau - 1) N - k \tau \gamma + \ell \gamma.
    \end{equation}
       \item Assume that $u_0 \in  \mathcal{E}_{\tau,\gamma}$ reads $\dis u_0(z) =\sum_{j=0}^{N-1} \lambda_{j,0} \Phi_j(z),$  with $\dis (\lambda_{j,0})_{0 \leq j \leq N-1} \in \C^N$, then  
               $$u(t,z) =\sum_{j=0}^{N-1} \lambda_j(t) \Phi_j(z),$$ 
               where $\dis (\lambda_{j}(t))_{0 \leq j \leq N-1}$ satisfy for all $0\leq j\leq N-1$
                  \begin{equation}\label{systL}
 i\dot{\lambda}_j=C_N\sum_{\substack {0\leq k, \ell, m\leq N-1\\ k-\ell+m=j\;\;[N]}}\Big(\int_{K_{\tau,\gamma}} \Phi_{k}(z)\ov{\Phi_{\ell}(z)} \Phi_{m}(z)\ov{\Phi_{j}(z)}\,dL(z)\Big) \lambda_{k}\ov{\lambda_{\ell}}\lambda_{m},
 \end{equation}         
   with 
   $$\dis C_N= \Big(\int_{K_{\tau, \gamma}}|\Phi_{j}(z)|^{2}dL(z)\Big)^{-1}=\frac1{\gamma N} \sqrt{\frac{2}{ {\pi}}} \EE\Big(- \frac{\pi^2}{2\gamma^2}\Big)$$         and we have for all $t \in \R$
    \begin{equation*} 
        \sum_{j=0}^{N-1}|\lambda_j(t)|^2=  \sum_{j=0}^{N-1}|\lambda_{j,0}|^2.
    \end{equation*}
          \end{enumerate}

          \end{theorem}

\begin{remark}\label{rem14}
$\bullet$ Notice that we have  the stationary solutions 
$$u(t,z) =  \kappa \EE( -i \mu t)\EE\big( {z^2}/2+\frac{2i \pi k}{\gamma}z    - {|z|^2}/2\big), \qquad \kappa \in \C, \;\; k \in \Z, $$
for some $\mu >0$, found in \cite[Proposition\;6.1]{GGT}, which  satisfies $R_{\gamma}u=u$. This would correspond to the case $N=0$ in Theorem \ref{prop42}.\smallskip

$\bullet$  When $N=1$, we obtain the stationary solution
\begin{equation*}
    u(t,z)=  e^{i\mu t}\Phi_0(z) =e^{i\mu t}\EE\Big( {  \frac{z^2}2 -\frac{i\pi}{\gamma} z     -  \frac{|z|^2}2}\Big) \Theta_{\tau}\Big(\frac{1}{\gamma}(z-z_0)\Big) , \quad z_0 = \frac{\gamma}{2}(\tau-1).
\end{equation*}
This stationary solution has already been obtained in \cite{ABN}. Indeed, set $\gamma=1/{\nu}$, then the function $f_\tau$ which is defined in \cite[Theorem 1.4]{ABN} is the function  $v(z)=\EE\big(z^2/2-|z|^2/2\big)\Theta_\tau({z}/{\gamma})$. Set $u:=R_{-z_0}v$, where $z_0 = {\gamma}(\tau-1)/2$, then $u \in \mathcal{E}_{\tau,\gamma}$ with $\tau_2=\pi/\gamma^2$. Notice however that $v \notin \mathcal{E}_{\tau,\gamma}$ because $ R_\gamma v=-v$ and $ R_{\gamma \tau} v=-v$.\smallskip

$\bullet$ The conservation law \eqref{somme1} is a condition such that $u$ belongs to the space $\E_{\tau,\gamma}$. By the previous result, for all $t\in \R$, any solution of \eqref{LLL} has exactly $N$ zeros (counted with multiplicity) in a fundamental cell $K_{\tau,\gamma}$. We can understand this fact in the following way: by the contour formula, the number of zeros of $u$ is locally preserved, until one zero crosses $\partial K_{\tau,\gamma}$, but then in this case, the number of zeros in $K_{\tau,\gamma}$ has to be conserved by the quasi-periodicity condition $\vert u(z+\gamma)\vert= \vert u(z+\gamma\tau)\vert=\vert u(z)\vert$.\smallskip

$\bullet$  Assume that $v\in \E_{\tau,\gamma}$, then $\dis R_{{\gamma}/N}v \in \E_{\tau,\gamma}$. As a consequence we can show that if $u$ is solution to~\eqref{LLL}, then $R_{{\gamma}/N}u$ is also solution with initial condition $R_{{\gamma }/N}u_0$. This can be reformulated for the system \eqref{systL} as: if $(\lambda_0(t),\lambda_1(t), \dots, \lambda_{N-1}(t))$ satisfies~\eqref{systL}, then $(\lambda_1(t), \dots, \lambda_{N-1}(t),\lambda_{0}(t))$ also.
\end{remark}

\begin{proof}[Proof of Theorem \ref{prop42}]
This is a direct consequence of part $(ii)$ of Proposition \ref{prop72}. For $u_{0}\in \mathcal{E}_{\tau, \gamma}$ we deduce that  the  map $\dis u\longmapsto u_{0}-i\int_{0}^{t} \Pi(\vert u\vert^2u)(s)ds$ sends $\mathcal{E}_{\tau, \gamma}$ into itself, and the fixed point argument implies that for all $t\in \R$, $u(t)\in \mathcal{E}_{\tau, \gamma}$.

The second part of the result follows from the characterisation of $\mathcal{E}_{\tau,\gamma}$ given in Proposition \ref{prop41}.
\end{proof}

\section{(LLL) on simply periodic domains (strips)}
\label{section3}

\subsection{Periodic functions in the Fock-Bargmann space}
Denote by $\mathcal{F}_{\gamma}$  the space
\begin{equation*}
    \mathcal{F}_{\gamma}=\Big\{u\in \widetilde{{\mathcal{E}}}:   \, R_{\gamma}u=u\Big\}.
\end{equation*}
 A subspace of $\mathcal{F}_\gamma$ is given by functions which are $L^p$-integrable on a fundamental domain associated to the translation $z \mapsto z+\gamma$. For convenience, we fix the vertical strip
$$
S_\gamma = \big\{ z \in \mathbb{C}, \, -{\gamma}/{2} < \mathfrak{Re} z \leq {\gamma}/{2} \big\},
$$
and we consider the natural $L^p(S_\gamma)$ norm. Similarly, $L^{p,\alpha}(S_\gamma) \cap \mathcal{F}_{\gamma} $ is the subspace of $\mathcal{F}_\gamma$ for which the following norm is finite
$$
\| u \|_{L^{p,\alpha}(S_\gamma)} = \| \langle z \rangle^\alpha u \|_{L^p(S_\gamma)}.
$$

Denote
$$\psi_0(z):=\left(\frac{2}{\pi \gamma^2}\right)^{1/4}\EE\Big({\frac{z^2}2-\frac{|z|^2}2}\Big),$$
and for $n\in \Z$
\begin{equation*} 
\psi_n(z) = R_{{i n \pi}/ \gamma}\psi_0(z) = \EE\Big({-\frac{\pi^2 n^2}{\gamma^2}}\Big) \EE\Big(\frac{2in\pi}{\gamma} z\Big)\psi_0(z).
\end{equation*}

  \begin{remark}
  Alternatively, we could have defined
$$
\wt{\psi}_n(z) =R_{{n\tau  \gamma}/ N}\psi_0(z)
$$
with the usual quantization condition on $\tau$ and $N$. This results in
$$
\wt{\psi}_n(z) =R_{{n\tau  \gamma}/  N}\psi_0(z)= \EE\Big( \frac{i n^2 \pi \tau_1}{N} -\frac{\pi^2 n^2}{\gamma^2}+ \frac{2in\pi}{\gamma} z\Big) \psi_0(z).
$$
So the only dependence on $\tau$ is through the number $\EE\big( {i {n^2 \pi \tau_1}/ {N}}\big)$ which has modulus 1 and is thus irrelevant (here we denote by $\tau_1=\Re \tau $).
  \end{remark}


\subsection{On the restriction of the projector $\Pi$ on   $\mathcal{F}_\gamma$}

\begin{lemme}\label{lemme_extension} Let $\Pi$ be the orthogonal projector from the space of $R_\gamma$-invariant functions to ${L^2(S_\gamma) \cap \mathcal{F}_{\gamma}}$, and denote $K(z,w)$ its kernel, for $(z,w) \in S_\gamma \times S_\gamma$.
  \begin{enumerate}[$(i)$]
\item The kernel $K$ is given by the formula
\begin{multline}\label{perio}
K(z,w)=\sqrt{\frac{2}{\pi \gamma^2}} \EE\Big( {-\frac{\pi^2}{2\gamma^2}} \Big)\EE\Big( {\frac{1}{2} z^2 - \frac{1}{2} |z|^2+\frac{1}{2} \ov{w}^2  - \frac{1}{2} |w|^2 -\frac{i\pi}{\gamma}  (z-  \ov{w})}\Big)\\
\times \Theta_{{2i\pi}/{\gamma^2}}\Big(\frac{1}{\gamma}(z-\ov{w}-\frac{i\pi}{\gamma}+\frac{\gamma}{2})\Big).
\end{multline}
\item  It enjoys the bound
\begin{equation}\label{borneK}
|K(z, w)| \lesssim \EE\Big(-\frac{(\mathfrak{Im}z-\mathfrak{Im}w)^2}2\Big).
\end{equation}
\item  For all $u,v\in \mathcal{F}_{\gamma} \cap L^2(S_{\gamma})$
\begin{equation}\label{aa}
        \int_{S_{\gamma}} \Pi u(z) \ov{v(z)}\, dL(z)= \int_{S_{\gamma}}  u(z) \ov{\Pi v(z)}\, dL(z).
\end{equation} 
\item The projector $\Pi$ on $\{ u, \; R_\gamma u = u \} \cap L^2(S_\gamma)$ can be identified with the orthogonal projector from $L^2(S_\gamma)$ to $\mathcal{F}_\gamma \cap L^2(S_\gamma)$.
\item For any $p \geq 2$ and $\alpha \geq 0$, the orthogonal projector $\Pi$ has a unique extension to $L^p(S_\gamma)$ and $L^{p,\alpha}(S_\gamma)$.
\end{enumerate}
\end{lemme}

\begin{proof} 
$(i)$ We show that the periodized projection operator in the whole space agrees with the projection operator given by \eqref{perio}. To get a formula for the periodized projection operator, we write, if $u$ is $R_\gamma$-invariant,
$$
u(z+k\gamma) = \EE\Big({{k\gamma}(z-\ov{z})/2 }\Big) u(z).
$$
The change of variable $w = w'+ k \gamma$ results in
$$
\EE\Big( {\ov{w} z - \frac{1}{2}|w|^2}\Big)  u (w) = \EE\Big(   \overline{w'} z - \frac{1}{2} |w'|^2 + k \gamma (z-\overline{w'}) - \frac{1}{2} k^2 \gamma^2\Big)  u(w').
$$
Making this change of variable in the integral after splitting it into vertical strips gives
\begin{align*}
\Pi u(z) & = \frac{1}{\pi} \EE\big( -{|z|^2}/2\big)  \int_\mathbb{C}  \EE\Big( \ov  w z - |w|^2/2\Big) u(w) \,dL(w) \\
&=  \frac{1}{\pi} \EE\big( -{|z|^2}/2\big) \sum_{ k \in \Z} \int_{w \in S_{\gamma} + k\gamma} \dots \\
& = \frac{1}{\pi} \EE\big( -{|z|^2}/2\big) \int_{S_{\gamma}}  \EE\big( {\ov w z - |w|^2/2 }\big) \\
&\hspace{4cm} \times \left[ \, \sum_{ k \in \Z}  \EE\big( {k \gamma (z-\overline{w}) -  k^2 \gamma^2/2}\big) \right] u(w) dL(w).
\end{align*}
Comparing with the above formula, we need to check that
\begin{multline*}
\sqrt{\frac{2}{\pi \gamma^2}} \EE\Big( {z^2/2  +\ov{w}^2/2 }\Big)\sum_{   k \in \Z }  \EE\Big( {\frac{2ik\pi}{\gamma}(z-\ov{w}) -\frac{2\pi^2k^2}{\gamma^2}}\Big)= \\
= \frac{1}{\pi} e^{\ov w z } \sum_{k \in \Z}  \EE\Big(  {k \gamma (z-\overline{w}) - \frac{1}{2} k^2 \gamma^2}\Big),
\end{multline*}
 but this follows from the Poisson formula \eqref{PF1} with $x={(z-\ov{w})}/ \gamma$ and $\alpha={\gamma^2}/2$. This proves \eqref{perio}.

\medskip

\noindent $(ii)$ Let $N=2$, $\tau = {2i\pi}/{\gamma^2}$, and $z_1 = z_2 = {\gamma}(\tau-1)/2$. Then, the function 
$$ u(z) = \EE\Big( {z^2/2-{2i\pi}z/ {\gamma}-|z|^2/2}\Big) \Theta^2_{{2i\pi}/{\gamma^2}}\Big( (z-z_1)/\gamma\Big) $$
is of the form \eqref{formprop}. Therefore, $u$ is periodic over the lattice $\mathcal{L}_{\tau,\gamma}$, so that it is bounded.  This implies that
$$ \abs{\Theta_{{2i\pi}/{\gamma^2}}\pa{\frac{1}{\gamma}(z-\ov{w}-\frac{i\pi}{\gamma}+\frac{\gamma}{2})}\EE \Big(  {-\frac{i\pi}{\gamma}(z-\ov{w})+\frac14(z-\ov{w})^2-\frac14|z-\ov{w}|^2}   \Big)} \lesssim 1.$$

Then, using $(i)$ and 
$$ \frac14(z-\ov{w})^2-\frac14|z-\ov{w}|^2 = \frac14z^2-\frac14|z|^2+\frac14\ov{w}^2-\frac14|w|^2-\frac14(2z\ov{w}-zw-\ov{z}\ov{w}), $$
we get 
$$ \abs{K(z,w)} \lesssim \abs{\EE \Big(  \frac14z^2-\frac14|z|^2+\frac14\ov{w}^2-\frac14|w|^2+\frac14(2z\ov{w}-zw-\ov{z}\ov{w}) \Big)}.$$
We write, for $z=x+iy, w=a+ib \in S_{\gamma}$,
\begin{multline*}
\frac14z^2-\frac14|z|^2+\frac14\ov{w}^2-\frac14|w|^2+\frac14(2z\ov{w}-zw-\ov{z}\ov{w})\\
\begin{aligned}
&  = \frac{i}{2}xy-\frac12y^2-\frac{i}{2}ab-\frac{1}{2}b^2 + yb +\frac12i(ya-xb) \\
& = -\frac12(y-b)^2 + \frac{i}2(y-b)(x+a),
\end{aligned}
\end{multline*}
and obtain the bound
$$ \abs{K(z,w)} \lesssim \EE\Big(-\frac{(\mathfrak{Im}z-\mathfrak{Im}w)^2}2\Big).$$
 
\medskip

\noindent $(iii)$ Can be proved like Lemma \ref{lemmaproj}.

\medskip

\noindent $(iv)$ Can be proved like Lemma \ref{lemmaproj2}.

\noindent $(v)$	 By the assertion $(ii)$, we can write for $u\in  L^p_{\gamma}$
		\begin{align*}
		\|\Pi u\|_{L^{p,\alpha}(S_{\gamma})} &= \absabs{\bra{z}^\alpha\int_{S_{\gamma}}K(z,w)u(w)dL(w)}_{L^p(S_{\gamma})}\\
			& \lesssim \absabs{\int_{S_{\gamma}}    \EE\Big(-\frac{(\mathfrak{Im}z-\mathfrak{Im}w)^2}2\Big) \big(\bra{w}^\alpha+\bra{z-w}^\alpha\big)u(w)dL(w)}_{L^p(S_{\gamma})} \\ 
			& \lesssim \absabs{\bra{z}^\alpha u}_{L^p(S_{\gamma})} = \|u\|_{L^{p,\alpha}(S_{\gamma})},
		\end{align*}
		and the result follows. 
\end{proof}

We shall now prove hypercontractivity estimates for all $L^p-L^q$. 
\begin{lemme}
	Let $u\in \mathcal{F}_{\gamma}$. Then for any $1 \leq p \leq q \leq +\infty$,
 	\begin{equation}\label{Carlen_bande}
 		\|u\|_{L^{q}(S_{\gamma})} \lesssim \|u\|_{L^{p}(S_{\gamma})}.
 	\end{equation}
\end{lemme}

\begin{proof}
By interpolation, it is sufficient to prove the result for $p \geq 1$ and $q= +\infty$. If $u\in \mathcal{F}_\gamma$,
 $$u(z)= \Pi u(z)=\int_{S_\gamma} K(z,w) u(w) dL(w).$$
By \eqref{borneK}, the kernel $K$ and the H\"older inequality, for all $z \in S_{\gamma}$,
    $$|u(z)|\leq   \| K(z, \cdot)\|_{L^{p'}(S_{\gamma})} \| u\|_{L^p(S_{\gamma})} \lesssim  \|u\|_{L^p(S_{\gamma})},$$
which is the desired result. 
\end{proof}

 \subsection{The family $(\psi_n)_{n \in \Z}$ is a  Hilbertian basis of $L^2(S_\gamma) \cap \mathcal{F}_\gamma$}

\begin{lemme}\label{lem98}
The family $(\psi_n)_{n \in \Z}$ is  orthonormal and forms a  Hilbertian basis of $L^2(S_\gamma) \cap \mathcal{F}_\gamma$.
\end{lemme}

In particular, the function $\Phi_0$ defined in \eqref{defPhi0} has a natural expansion in the family $(\psi_n)_{n \in \Z}$, namely a direct computation shows that, with $z_0={\gamma}({\tau}/N-1)/2$,

\begin{eqnarray}\label{exp1}
  \Phi_{0}(z)&=&\EE\Big(  \frac{z^2}2 -\frac{i\pi z} {\gamma} -  \frac{|z|^2}2\Big)  \Theta_{{\tau}/{N}}\Big((z-z_{0})/\gamma\Big) \nonumber \\
  &=&   e^{-i{\pi \tau}/{(4N)}}\pa{{\pi\gamma^2}/{2}}^{1/4}\sum_{n=-\infty}^{+\infty} e^{i{n^2\pi\tau_1}/{N}}\psi_n(z).   
  \end{eqnarray}

\begin{proof} It is clear that $\psi_0 \in \mathcal{F}_\gamma$, and this remains true for $\psi_n$ since $\big[R_{{i n \pi}/\gamma},R_\gamma\big] = 0$. Furthermore,
\begin{multline*}
 \int_{S_{\gamma}} \psi_{k}(z)\ov{\psi_{\ell}} dL(z) =\\=\left(\frac{2}{\pi \gamma^2}\right)^{1/2}\EE\Big({-\frac{\pi^2}{\gamma^2} (k^2+\ell^2)}\Big)
        \int_{-\infty}^\infty  \int_0^\gamma \EE\Big(-2y^2+\frac{2i\pi}{\gamma}(k-\ell)x-\frac{2\pi}{\gamma}(k+\ell)y\Big) dx  dy \\
        = \delta_{k,\ell} \Big(\frac{2}{\pi }\Big)^{1/2} \EE\Big({-\frac{2\pi^2}{\gamma^2} k^2}\Big)  \int_{-\infty}^\infty \EE\Big(-2y^2-\frac{4\pi}{\gamma}ky\Big) dy = \delta_{k,\ell},
\end{multline*}
which shows that the family is orthonormal. \medskip

We now show that the family is complete. Let $u \in L^2(S_\gamma) \cap \mathcal{F}_\gamma$. For $z=x+iy$, we set $v(x,y)=u(z)\EE\big(-ixy\big)$. Then for all $y \in \R$, the function  $x \mapsto v(x,y)$ is $\gamma-$periodic. We expand it in Fourier series $\dis v(x,y)=\sum_{k \in \Z} c_k(y)\EE\big( {{2i \pi k x}/\gamma}\big)$ and we have 
$$\|u \|^2_{L^2(S_\gamma)}=\|v\|^2_{L^2(S_\gamma)}=\sum_{k \in \Z} \int_{\R}|c_k(y)|^2 dy <\infty. $$
Therefore, for all $k \in \Z$, $c_k \in L^2(\R)$. We now expand $c_k$ in the Hilbertian basis of~$L^2(\R)$ given by  (translated) Hermite functions. Namely, for all $k\in \Z$, there exists a family of polynomial $(P_{j,k})_{j \geq 0}$ where $P_{j,k}$ has degree $j$ and such that 
$$c_k(y)=\EE\Big( -(y+\frac{k \pi}\gamma)^2\Big)\sum_{j=0}^{+\infty}P_{j,k}(y).$$

Coming back to $u$, it can be written as the following convergent series
$$u(z)=\sum_{k \in \Z} \sum_{j\geq 0} \EE\Big( {ixy+\frac{2i \pi k x}\gamma -(y+\frac{k \pi}\gamma)^2} \Big) P_{j,k}(y),$$
where the summands are orthogonal. To show that the family $(\psi_k)$ is complete, using the continuity of $\Pi$ on $L^2(S_\gamma)$ by Lemma \ref{lemme_extension} $(v)$,  it suffices to show that the orthogonal projection of each summand can be written as a linear combination of these functions.

We will do so by showing that, for all $k \in \Z$ and $j \geq 0$, there exists $C_{j,k} \in \C$ such that 
\begin{equation}\label{claim1}
A_{j,k}=\Pi\Big(    \EE\Big( {-(y+\frac{k \pi}\gamma)^2}\Big)   (y+\frac{k \pi}\gamma)^j \EE\Big( {ixy+\frac{2ik\pi x}{\gamma} }  \Big)\Big)= C_{j,k}  \psi_k(z).
\end{equation}
By the formula for the projection 
\begin{multline*}
A_{j,k} =\\
\begin{aligned}
&= \frac{1}{\pi} e^{-|z|^2/2} \int_\mathbb{C} \EE\Big( {\ov  w z - |w|^2/2}\Big)  \EE\Big( {-(\frac{w -\ov w}{2i}+\frac{k \pi}\gamma)^2} \Big)(\frac{w -\ov w}{2i}+\frac{k \pi}\gamma)^j   \\
&\hspace{6cm}\times  \EE\Big( {\frac14(w^2-\ov  w^2)  +\frac{i \pi k}{\gamma}(w+\ov w)}   \Big) \,dL(w)\\
&= \frac{1 }{\pi (2i)^j}  \EE\Big(-\frac{\pi^2 k^2}{\gamma^2}-\frac{|z|^2}2\Big)   \int_\mathbb{C} \EE \Big( {-|w|^2 + \ov  w z +\frac{2i \pi k}{\gamma} w+\frac12 w^2 } \Big)   ({w -\ov w}+\frac{2i k \pi}\gamma)^j       \,dL(w)\\
&= \frac{2^{j/2+1}}{\pi (2i)^j}  \EE\Big(-\frac{\pi^2 k^2}{\gamma^2}-\frac{|z|^2}2\Big)\int_\mathbb{C} \EE\Big( {-2|w|^2 + \sqrt{2}\ov  w z +\frac{2 \sqrt{2}i \pi k}{\gamma} w+ w^2 }  \Big) \\
&\hspace{6cm}\times   ({w -\ov w}+\frac{\sqrt{2}i k \pi}\gamma)^j       \,dL(w).
\end{aligned}
\end{multline*}
Let $t\in \R$ and let us compute 
\begin{multline*}
 \EE\Big(\frac{\pi^2 k^2}{\gamma^2} \Big)\sum_{j=0}^{+\infty}  \frac{\pi (2i)^j}{2^{j/2+1}}\frac{t^j A_{j,k}}{j!}=\EE\Big(   -\frac{|z|^2}2 + \frac{\sqrt{2}i k \pi t } \gamma  \Big)\\
\times \int_\mathbb{C} \EE\Big( {-2|w|^2 +(\frac{2 \sqrt{2}i \pi k}{\gamma} +t)w + (\sqrt{2}z-t)\ov w   + w^2 }      \Big) \,dL(w).
\end{multline*}
From \cite[Section 6]{GGT} we recall the formula
\begin{equation*}
 \int \EE\Big( -2 |w|^2 + aw + b\overline{w} + cw^2 \Big)\,dL(w) 
=  \frac{\pi}{2} \EE\Big({b(cb+2a)}/{4}\Big).
\end{equation*}
We set $a={2 \sqrt{2}i \pi k}/{\gamma} +t$, $b=\sqrt{2}z-t$ and $c=1$, thus 
\begin{multline}\label{series}
 \EE\Big(\frac{\pi^2 k^2}{\gamma^2} \Big)\sum_{j=0}^{+\infty}  \frac{\pi (2i)^j}{2^{j/2+1}}\frac{t^j A_{j,k}}{j!}=\frac{\pi}{2}  \EE\Big(   -\frac{|z|^2}2 + \frac{\sqrt{2}i k \pi t } \gamma  \Big)\\
\times \int_\mathbb{C} \EE\Big( {-2|w|^2 +(\frac{2 \sqrt{2}i \pi k}{\gamma} +t)w + (\sqrt{2}z-t)\ov w   + w^2 }  \Big)     \,dL(w) \\
=  \frac{\pi}{2} e^{- t^2/4} \EE\Big( {\frac{z^2}{2}-\frac{|z|^2}2+\frac{2i k\pi z}{\gamma}}\Big) = C_k e^{-t^2/4}\psi_k(z).
\end{multline}
Finally, by identifying the powers of $t$ in the expansion \eqref{series} we get  \eqref{claim1}.
\end{proof}

\begin{remark}

 Since the $(\psi_k)$ form a Hilbertian basis of $L^2(S_\gamma) \cap \mathcal{F}_\gamma$, we can recover the expression of the kernel $K$ as follows:
   \begin{multline*}
   K(z,w)=\\
  \begin{aligned}
 &= \sum_{   k \in \Z } \psi_k(z)\ov{\psi_k(w)}\\
&= \sqrt{\frac{2}{\pi \gamma^2}} \EE\Big( {\frac12z^2-\frac12|z|^2+\frac12\ov{w}^2-\frac12|w|^2}\Big) \sum_{   k \in \Z }  \EE\Big( {\frac{2ik\pi}{\gamma}(z-\ov{w}) -\frac{2\pi^2k^2}{\gamma^2}}\Big)\\
&=\sqrt{\frac{2}{\pi \gamma^2}} \e^{-{\pi^2}/{(2\gamma^2)}} \Theta_{{2i\pi}/{\gamma^2}}\Big(\frac{1}{\gamma}(z-\ov{w}-\frac{i\pi}{\gamma}+\frac{\gamma}{2})\Big) \\
&\hspace{4cm} \times \EE\Big( {-\frac{i\pi}{\gamma}(z-\ov{w}) +\frac12z^2-\frac12|z|^2+\frac12\ov{w}^2-\frac12|w|^2}\Big).
\end{aligned}
 \end{multline*}

\end{remark}

 \subsection{The equation (LLL) on $\mathcal{F}_{\gamma}$: basic structure}

\subsubsection{The equation in physical space}

\begin{lemme}\label{lem_conserv}
The following quantities are formally conserved by the flow of \eqref{LLL}
\begin{align*}
& \mathcal{H}(u) = \int_{S_{\gamma}}|u(z)|^{4} \, dL(z) \\
& M(u) = \int_{S_{\gamma}}|u(z)|^{2}\, dL(z) \\
&  P(u) =  \int_{S_{\gamma}}\ov{u(z)} \Gamma_1 u(z)\, dL(z)=  i \int_{S_{\gamma}}(z-\ov{z})|u(z)|^{2}\, dL(z)
\end{align*}
(Hamiltonian, mass and momentum respectively).
\end{lemme}

 \begin{proof}
The proof of the first point and the two first conservation laws are similar to what we did previously. 

  Let us  check the equality of the two different expressions of $P(u)$. We write $u(z)=f(z)\EE\big({-{|z|^2}/2}\big)$, where $f$ is an holomorphic function. Then 
       \begin{eqnarray*}
 \int_{S_{\gamma}}\ov{u(z)} \Gamma_1 u(z)\, dL(z)&=& i  \int_{S_{\gamma}}\ov{f(z)}\big(z f(z) -\partial_z f(z)\big)  e^{-|z|^2} \, dL(z) \\
 &=& i  \int_{S_{\gamma}}    z    |f(z)|^2 e^{-|z|^2}   dL(z)+   i  \int_{S_{\gamma}}   |f(z)|^2    \partial_z  \big(   e^{-|z|^2} \big) \, dL(z) \\
 &=&   i \int_{S_{\gamma}}(z-\ov{z})|u(z)|^{2}\, dL(z).
       \end{eqnarray*}

The conservation of $P(u)$ follows from the action of $R_{\alpha}$ on $\mathcal{F}_{\gamma}$, but we can give a direct proof: Set $\partial_{z}=(\partial_{x}-i\partial_{y})/2$, $\partial_{\overline{z}}=(\partial_{x}+i\partial_{y})/2$. First we check that for all $u,v\in \mathcal{F}_{\gamma}$
 \begin{equation*}
 \int_{S_{\gamma}} \partial_z u(z) \ov{v(z)}dL(z)= -\int_{S_{\gamma}}  u(z) \ov{\partial_{\ov{z}}v(z)}dL(z).
 \end{equation*} 
Writing $u(z)=f(z)\EE\big( {-\vert z\vert^2/2}\big)$, we have  $\Pi(\ov{z}u )=\partial_z f(z)\EE\big( {-\vert z\vert^2/2}\big)$, which implies that
\begin{multline*}
\int_{S_{\gamma}} \vert {u(z)} \vert^2\ov{u(z)}\Pi(\ov{z}u )dL(z) = \int_{S_{\gamma}} \ov{f(z)}^2f(z) \partial_z f(z) e^{-2\vert z\vert^2}dL(z)\nonumber \\
  =\frac12 \int_{S_{\gamma}} \ov{f(z)}^2  \partial_z\big( f^2(z)\big) e^{-2\vert z\vert^2}dL(z) = \int_{S_{\gamma}}\ov{z}\vert f(z)\vert^4 e^{-2\vert z\vert^2}dL(z)\nonumber \\
  = \int_{S_{\gamma}}\ov{z}\vert u(z)\vert^4  dL(z).
\end{multline*}
Together with \eqref{aa} , this implies that 
\begin{eqnarray*}
  \frac{d}{dt}\int_{S_{\gamma}}(z-\ov{z})|u(z)|^{2}dL(z) &=& -2i \Re \int_{S_{\gamma}}(z-\ov{z})\ov{u(z)} \Pi(|u|^{2}u)(z)dL(z) \\
& =& -2i \Re \int_{S_{\gamma}}\ov{\Pi\big((\ov{z}-z)u(z)\big)}|u|^{2}u(z)\,dL(z) \\
& =&-2i \Re \int_{S_{\gamma}}(z-\ov{z})|u(z)|^{4}\,dL(z)=0
\end{eqnarray*}
hence the result.
\end{proof}

The symmetries of the equation \eqref{LLL} on a strip are as follows:
\begin{itemize}
\item Phase rotation: $u \mapsto \EE\big({i \theta}\big) u$, for $\theta\in \mathbb{R}$.
\item Horizontal magnetic translation: $u \mapsto R_\theta u$, for $\theta\in \mathbb{R}$.
\item Vertical magnetic translation: $\dis u \mapsto R_{{i\pi}/{\gamma}} u$.
\item Symmetry around the origin: $u(z) \mapsto u(-z)$.
\end{itemize}

By Noether's theorem, the two continuous symmetries (phase rotation and horizontal magnetic translation) are related to the conserved quantities (mass and momentum respectively).

\subsubsection{The equation in the Hilbertian basis $\dis (\psi_k)_{k \in \Z}$}

From now on, we denote
\begin{equation}\label{def-A}
A_{k,\ell,m} =  \frac{1}{\gamma\sqrt{\pi}}\EE\Big( -\frac{\pi^2}{\gamma^2}\big((\ell-k)^2+(\ell-m)^2\big)\Big).
\end{equation}
Then, the following result holds true.

\begin{lemme}\label{lem-ortho} 
If $k,\ell,m,n \in \mathbb{Z}$, then
\begin{equation*}
\int_{S_{\gamma}} \psi_{k}(z)\ov{\psi_{\ell}(z)} \psi_{m}(z)\ov{\psi_{n}(z)}\,dL(z)
=\begin{cases} 
\displaystyle A_{k,\ell,m} & \mbox{if $k - \ell + m - n = 0$} \\
0 & \mbox{if $k - \ell + m - n \neq 0$}.
\end{cases}
\end{equation*}

In other words, when $n = k - \ell + m$, we have 
        \begin{equation}\label{proj3}
            \Pi\big(\psi_{k}\ov{\psi_{\ell}} \psi_{m}\big)
            =A_{k,\ell,m}   \psi_{n}.
        \end{equation}
\end{lemme}

 \begin{proof} We write
    \begin{multline*}
 \int_{S_{\gamma}} \psi_{k}\ov{\psi_{\ell}} \psi_{m}\ov{\psi_{n}}(z)dL(z)=\\
 = \frac{2}{\pi \gamma^2}\int_{0}^{\gamma}\int_{\R}  
 \EE\Big( \frac{2i \pi}{\gamma}(k-\ell+m-n)x -\frac{2\pi}{\gamma}(k+\ell+m+n)y\Big)\\
\times  \EE\Big(-4y^2-\frac{\pi^2}{\gamma^2}(k^2+\ell^2+m^2+n^2)\Big) dy \, dx
  \end{multline*}
which vanishes if $k-\ell+m-n \neq 0$. Assuming that $k-\ell+m-n =0$, 
    \begin{multline*}
 \int_{S_{\gamma}} \psi_{k}\ov{\psi_{\ell}} \psi_{m}\ov{\psi_{n}}(z)dL(z)=\\
     \begin{aligned}
 &= \frac{2}{\pi \gamma}  \int_{\R}  \EE\Big(  -\frac{2\pi}{\gamma}(k+\ell+m+n)y-4y^2-\frac{\pi^2}{\gamma^2}(k^2+\ell^2+m^2+n^2)\Big) dy\\
 &= \frac{2}{\pi \gamma}\EE\Big( \frac{\pi^2}{4 \gamma^2}(k+\ell+m+n)^2-\frac{\pi^2}{\gamma^2}(k^2+\ell^2+m^2+n^2)\Big)  \int_{\R}  \e^{ -4y^2} dy,
     \end{aligned}
       \end{multline*}
which gives the desired formula.
  \end{proof}

Expanding a solution of~\eqref{LLL} in the Hilbertian basis
$$
u(t,z) = \sum_{k \in \mathbb{Z}} \lambda_k(t) \psi_k(z),
$$
it follows from the above proposition that the coordinates $(\lambda_n)$ satisfy the equation
\begin{equation}\label{eq_lambda_k}
i\frac{d}{dt} \lambda_n(t) = \sum_{k - \ell + m = n} A_{k,\ell,m} \lambda_{k}(t) \ov{\lambda_{\ell}(t)} \lambda_{m}(t), \quad n \in \Z, \quad t\in \R.
\end{equation}

The Hamiltonian, the mass and the momentum can be expressed in the Hilbertian basis as
\begin{align*}
& \mathcal{H}(u) = \frac14 \sum_{k - \ell + m = n} A_{k,\ell,m} \lambda_{k} \ov{\lambda_{\ell}} \lambda_{m} \ov{\lambda_{n}} \\
& M(u) = \sum_{k \in \mathbb{Z}} | \lambda_k|^2 \\
& P(u) =\frac{2\pi}{\gamma} \sum_{k \in \mathbb{Z}} k | \lambda_k|^2,
\end{align*}
as follows by expressing these quantities in the Hilbertian basis $(\psi_k)_{}$.  \medskip

As for symmetries in this new coordinate system, they can be expressed as follows:
\begin{itemize}
\item Phase rotation: $\lambda_k \mapsto  \EE\big({i\theta}\big) \lambda_k$, for $\theta\in \mathbb{R}$.
\item Horizontal magnetic translation: $\dis \lambda_k \mapsto \EE\big({2\pi i k \theta}/{\gamma}\big) \lambda_k$, for $\theta \in \mathbb{R}$ \big(since $R_\theta \psi_k = \EE\big( {2\pi i k \theta/ \gamma}\big) \psi_k$\big).
\item Vertical magnetic translation: $\lambda_k \mapsto \lambda_{k+1}$ (since $R_{{i\pi}/{\gamma}} \psi_k = \psi_{k+1}$).
\item Symmetry around the origin: $\lambda_k \mapsto \lambda_{-k}$ (since $\psi_k(-z) = \psi_{-k}(z)$).
\end{itemize}

\subsection{Local and global solutions of periodic LLL}

For any $\alpha \geq 0$, we define the Banach space~$\ell^{\infty,\alpha}$ constituted of sequences, given by the norm
$$ \|(\lambda_k)\|_{\ell^{\infty,\alpha}} = \sup_{k\in\mathbb{Z}} \bra{k}^\alpha |\lambda_k|.$$
We also define $L^{p, \alpha}(S_\gamma)$ the weighted Lebesgue space by the norm
$$ \|f\|_{L^{p,\alpha}(S_\gamma)} =   \|\< z\>^\alpha f\|_{L^p(S_\gamma)}.$$
In this section, we will prove well-posedness results. We also refer to~\cite[Section 3]{GGT} where similar results are obtained for~\eqref{LLL} in the whole space.

\begin{prop}\label{LWP}
	We have local well-posedness results on the following spaces:
	\begin{enumerate}[$(i)$]
		\item For any $p \in [1,\infty],$ the equation \eqref{LLL} is locally well-posed in $\mathcal{F}_{\gamma} \cap L^p(S_{\gamma})$: for any data $u_0 \in \mathcal{F}_{\gamma} \cap L^p(S_{\gamma})$ there exists $T > 0$ and a unique solution $u \in \mathcal{C}\big([0, T], \mathcal{F}_{\gamma} \cap L^p(S_{\gamma})\big)$, which depends smoothly on $u_0$.
		\item For any $p \in [1,\infty], \alpha \geq 0,$ the equation \eqref{LLL} is locally well-posed in $L^{p,\alpha}(S_{\gamma})$: for any data $u_0 \in \mathcal{F}_{\gamma}  \cap L^{p,\alpha}\big(S_{\gamma}\big)$ there exists $T > 0$ and a unique solution $u \in \mathcal{C}\big([0, T], \mathcal{F}_{\gamma} \cap L^{p,\alpha}(S_\gamma)\big)$, which depends smoothly on $u_0$.
		\item By writing $\dis u= \sum_{k\in\Z} \lambda_k\psi_k$, the equation \eqref{LLL} is locally well-posed in $\ell^{\infty,\alpha}$ for $\alpha \geq 0$.
	\end{enumerate}
\end{prop}

\begin{proof}
$(i)$ and $(ii)$: It suffices to prove that the map $ u \mapsto \Pi(|u|^2u)$ is smooth on the weighted spaces $L^{p,\alpha}(S_\gamma)$, for $2 \leq p \leq +\infty$ and $\alpha \geq 0$, with a differential bounded on bounded subsets. This follows from both Lemma \ref{lemme_extension}, item $(v)$, and the hypercontractivity estimates \eqref{Carlen_bande}:
\begin{align*}
\|\bra{z}^\alpha\Pi(f\ov{g}h)\|_{L^{p}(S_{\gamma})} &\lesssim \|\bra{z}^\alpha f\ov{g}h\|_{L^{p}(S_{\gamma})} \lesssim \|\bra{z}^\alpha f\|_{L^{p}(S_{\gamma})} \|g\|_{L^{\infty}(S_{\gamma})}\|h\|_{L^{\infty}(S_{\gamma})} \\
& \lesssim \|\bra{z}^\alpha f\|_{L^{p}(S_{\gamma})}\|\bra{z}^\alpha g\|_{L^{p}(S_{\gamma})}\|\bra{z}^\alpha h\|_{L^{p}(S_{\gamma})} \\
& = \| f\|_{L^{p,\alpha}(S_{\gamma})} \| g\|_{L^{p,\alpha}(S_{\gamma})} \| h\|_{L^{p,\alpha}(S_{\gamma})}.
\end{align*}

$(iii)$ The equation in $(\lambda_n)$ coordinates can be written
\begin{align*}
i \frac{d}{dt} \lambda_n &=  \frac{1}{\gamma\sqrt{\pi}}\sum_{k-\ell+m=n}\EE \Big({-\frac{\pi^2}{\gamma^2}\pa{(\ell-k)^2+(\ell-m)^2}}\Big) \lambda_{k} \ov{\lambda_{\ell}} \lambda_{m} \\
&= \frac{1}{\gamma\sqrt{\pi}}\sum_{p\in\Z}\sum_{m\in\Z}\EE \Big(  {-\frac{2\pi^2}{\gamma^2}\Big(\big(n-\frac{p}{2}\big)^2+\big(m-\frac{p}{2}\big)^2}\Big)\Big) \lambda_{p-m}  \ov{\lambda_{p-n}}\lambda_{m}\\
&  =: \mathcal{R}(\lambda,\lambda,\lambda).
\end{align*}
We need to prove that $\mathcal{R}$ maps $(\ell^{\infty,\alpha})^3$ to $\ell^{\infty,\alpha}$, that is to say 
\begin{align*}
&\sigma_n = \sum_{p\in\Z}  \EE\Big( {-\frac{2\pi^2}{\gamma^2}\big( {n-\frac{p}{2}}}\big)^2\Big) \frac{1}{\bra{p-n}^\alpha}  \sum_{m\in\Z} \EE\Big( {-\frac{2\pi^2}{\gamma^2}\big( {m-\frac{p}{2}}}\big)^2\Big)     \frac{1}{\bra{m}^\alpha\bra{p-m}^\alpha} \\
 &\lesssim \frac{1}{\bra{n}^\alpha}.
\end{align*}
We first deal with the sum over $m \in \Z$: 
\begin{eqnarray*}
\sum_{m\in\Z}\EE\Big( {-\frac{2\pi^2}{\gamma^2}\big( {m-\frac{p}{2}}}\big)^2\Big)  \frac{1}{\bra{m}^\alpha\bra{p-m}^\alpha} &\lesssim& \frac{1}{\bra{p}^\alpha} \sum_{m\in\Z}\EE\Big( {-\frac{2\pi^2}{\gamma^2}\big( {m-\frac{p}{2}}}\big)^2\Big)   \\
&\lesssim &\frac{1}{\bra{p}^\alpha}.
\end{eqnarray*}
It remains the sum over $p  \in \Z$. It is similarly bounded:
$$ \sigma_n \lesssim \sum_{p\in\Z} \EE\Big( {-\frac{2\pi^2}{\gamma^2}\big( {n-\frac{p}{2}}}\big)^2\Big)    \frac{1}{\bra{p-n}^\alpha\bra{p}^\alpha} \lesssim \frac{1}{\bra{n}^\alpha}, $$
which is the result.
\end{proof}

Finally, we state a global well-posedness result for \eqref{LLL} on the strip, which is very similar to the case of the equation posed in the whole space, therefore we refer to \cite[Proposition 3.7]{GGT} for the proof.
\begin{prop}
Assume that $2\leq p\leq 4$. The equation~\eqref{LLL} is globally well-posed for data $u_0 \in \mathcal{F}_{\gamma} \cap L^p(S_{\gamma})$ and such data lead to solutions in $\mathcal{C}^\infty \big(\R, \mathcal{F}_{\gamma} \cap L^p(S_{\gamma}) \big)$, depending smoothly on~$u_0$. 
\end{prop}

We stress that the global well-posedness for data $u_0 \in  \mathcal{F}_{\gamma} \cap L^{\infty}(S_{\gamma}) $ is an open problem.

\subsection{Classification of stationary solutions with a finite number of zeros in a strip}

\begin{df}
An $M$-stationary wave is a solution of~\eqref{LLL} of the form 
$$u(t) = e^{-i a t} u_0, \quad  \mbox{where $a \in \mathbb{R}$, $u_0 \in \mathcal{F}_{\gamma}$}.$$

An $MQ$-stationary wave is  a solution of~\eqref{LLL} of the form
$$u(t) = e^{-i a t} R_{-\alpha t}u_0, \quad \mbox{where $a \in \mathbb{R}$, $\alpha=\alpha_1+i\alpha_2 \in \mathbb{C}$, $u_0 \in \mathcal{F}_{\gamma}$}.$$
\end{df}

They are given, respectively, by the solutions of
\begin{equation}
\label{M} \tag{$M$}
a u = \Pi \big( |u|^2 u\big)
\end{equation}
\begin{equation}
\label{MQ} \tag{$MQ$}
a u + \alpha  \cdot  \Gamma u = \Pi \big( |u|^2 u\big),
\end{equation}
where, by \eqref{infinitesimal}, $\alpha  \cdot  \Gamma =\alpha_1    \Gamma_1+\alpha_2   \Gamma_2 $ is the operator on $\mathcal{F}_\gamma$ defined by 
  \begin{equation*} 
\Gamma_1=i(z -\partial_z-\frac{\ov z}2), \qquad \Gamma_2=(z +\partial_z+\frac{\ov z}2).
\end{equation*}
Since $\Gamma_1 \psi_k = \frac{2 k \pi}\gamma \psi_k$, the operator $\Gamma_1$ can be expressed in the Hilbertian basis as
$$
\Gamma_1 \big((\lambda_k)_k\big) = \left(\big( {2\pi k}/{\gamma} \big) \lambda_k\right)_k.
$$

We define the set 
\begin{equation*}
\mathcal{G}=\Big\{u\in \mathcal{F}_\gamma : u\;\text{has a finite number of zeros modulo $\gamma$} \Big\}.
\end{equation*}
 
To begin with, we provide a complete characterization of the space $\mathcal{G}$.
\begin{lemme} \label{lem77}
We have the following descriptions: 
\begin{enumerate}[$(i)$]
\item (multiplicative description) The function $u \in \mathcal{G}$ has $N$ zeros modulo $\gamma$ if and only if it can be written under the form  
\begin{equation*} 
u(z)=\kappa \EE\Big( {\frac{1}{2} z^2 +\frac{2i\pi L}{\gamma}z- \frac{1}{2} |z|^2} \Big) \prod_{k=1}^N\Big(e^{{i\pi z}/{\gamma} }\sin\big(\frac{\pi}{\gamma}(z-z_k)\big)\Big)
\end{equation*}
where $L \in\mathbb{Z}$, $\kappa \in \C$ and $(z_k)_{1\leq k\leq N}$ in $S_\gamma$ modulo $\gamma$.
\item (additive description)  The function $u \in \mathcal{G}$ has $N$ zeros if it can be written under the form
\begin{equation*} 
u(z)=\kappa  \EE\Big(  \frac{z^2}2 -   \frac{|z|^2}2\Big) Z^L P(Z), 
\end{equation*}
where $Z = \EE\big({{2\pi i z}/{\gamma}}\big)$ and $P$ is a complex polynomial of degree $N$ such that ${P(0) \neq 0}$.
\end{enumerate}
\end{lemme}

Notice that this proves the embedding $\mathcal{G} \subset L^2(S_{\gamma})$.

\begin{proof}
The proof of the multiplicative description is similar to the proof of Proposition~\ref{prop41}: denote by $(z_k)_{1\leq k\leq N}$  the zeroes of $u$ in the strip $\dis -\gamma/2< \Re z\leq  \gamma/2$ and write
$$
\dis u(z) = \EE\big(-|z|^2/2\big) \varphi(z) \prod_{k=1}^N\Big(e^{{i\pi z }/{\gamma}}\sin\big(\frac{\pi}{\gamma}(z-z_k)\big)\Big)
.
$$
    By construction, $\varphi$ does not vanish on $\mathbb{C}$, and is entire; thus, it can be written $\varphi = \EE\big( \Psi\big)$, with $\Psi$ entire. Furthermore, denoting $\dis A = \mathbb{C} \setminus \cup_{\substack{k \in \{1,\dots,N\} \\ n \in \mathbb{Z}}} B(z_k+n\gamma,\eps)$ for $\eps>0$ small enough, it is easy to see that $|\varphi(z)| \leq C \EE\big( C|z|^2\big)$ on~$A$, hence on~$\mathbb{C}$ by the maximum modulus principle. This implies that $\mathfrak{Re} \Psi (z) \leq C |z|^2$ on $\mathbb{C}$, and  applying the Borel-Caratheodory theorem gives that $\Psi$ is a polynomial of degree at most 2. Therefore, we can write
$$
 u(z) = \kappa \EE\Big(   {-|z|^2/2 + \alpha z^2 + \beta z}\Big) \prod_{k=1}^N\Big(e^{{i\pi z}/ {\gamma} }\sin\big(\frac{\pi}{\gamma}(z-z_k)\big)\Big).
$$
Finally, the periodicity condition $R_\gamma u = u$ leads to $\alpha = {1}/2$ and $\beta = {2i\pi L}/{\gamma}$.

To prove the additive description, notice that $\dis \EE\big(  {{i\pi z}/{\gamma} }\big)\sin\big({\pi}(z-z_k)/\gamma\big)$ can be written under the form $aZ+b$, with $b \neq 0$. This proves that any function which in the multiplicative form above can also be expressed in the additive form. Conversely, if a function can be put in the additive form, it belongs to $\mathcal{G}$ and has $N$ zeros modulo~$\gamma$ since $z \mapsto Z$ is a bijection from $S_\gamma$ to $\mathbb{C} \setminus \{ 0 \}$.
\end{proof}

\begin{prop}
The $M$-stationary solutions to \eqref{LLL} in $\mathcal{G}$ are given by the formula
$$
u(t,z)=\kappa \psi_k(z) \EE\Big( -i \frac{|\kappa|^2 t }{\gamma\sqrt{\pi}}\Big), \quad\quad k \in \Z , \; \kappa \in \C.
$$
There are no $MQ$-stationary solutions beyond the above examples.
\end{prop}

Notice that the previous waves belong to $L^2(S_{\gamma})$ and that for 
$$u(t,z)=\kappa \psi_k(z) \EE\Big( -i \frac{|\kappa|^2 t }{\gamma\sqrt{\pi}}\Big)$$
$$M(u)=|\kappa|^2, \qquad P(u)=\frac{2\pi k}{\gamma}|\kappa|^2, \qquad \mathcal{H}(u)=\frac14 \|u\|^4_{L^4(S_{\gamma})}=\frac{|\kappa|^4}{4\gamma\sqrt{\pi}}.$$
\begin{proof}
Let $u \in \mathcal{G}$ be a $MQ$-stationary solution to \eqref{LLL}, namely 
$$u(t,z)=e^{-i at}R_{-\alpha t}U(z)$$  for some  $a \in \R$, $\alpha \in \C$ and  $U \in \mathcal{G}$. Thus there exists $L \in \Z$ such that 
$$\dis U(z) =c \EE\big( {\frac{1}{2} z^2 - \frac{1}{2} |z|^2}\big) Z^L P(Z)$$
 where the polynomial $P$ satisfies the hypotheses of Lemma \ref{lem77}.  Next, $U$ satisfies the equation  $a U+\alpha \cdot     \Gamma U =  \Pi \big( |U|^2 U\big)$. Let $\beta \in \C$, then by \eqref{com0} and \eqref{r-1}, 
$$\big(a+2 \Im (\alpha \ov{\beta})\big) R_\beta U+(\alpha \cdot     \Gamma) R_\beta U  =  \Pi \big( |R_\beta U|^2 R_\beta U\big).$$
After that, observe that 
$$
R_{\pm {i\pi}/{\gamma}} \left[ e^{  z^2/2 -   |z|^2/2} Z^L P(Z) \right] = \mu^L Z^{L\pm 1} P(\mu Z) e^{- {\pi^2}/{\gamma^2}} e^{  z^2/2 -   |z|^2/2}, 
$$
with $\mu = \EE\big({\mp {2\pi^2}/{\gamma^2}}\big)$, thus choosing $\beta=- {i\pi}L/\gamma$ and setting $V=R_{\beta}U$, it suffices to treat the case $L=0$. Moreover $[R_{\gamma}, R_{\beta}]=0$, so that $V \in \mathcal{G}$. Then, $V$ can be expanded as 
\begin{equation*}
V(z)= e^{ z^2/2-   |z|^2/2} \sum_{j=0}^Nc_j e^{{2i j \pi z }/{\gamma} }= \sum_{j=0}^Nd_j  \psi_j(z),
\end{equation*}
with $d_0 \neq 0$ and $d_N \neq 0$. Therefore
\begin{equation*}
|V|^2V(z)=  \sum_{0 \leq k, \ell, m \leq N}d_{k}\ov{d_{\ell}}d_{m}  \psi_{k}\ov{\psi_{\ell}}\psi_{m}(z).
\end{equation*}
We compute the coefficient of highest degree in $\Pi(|V|^2V)$: it is given by  $k=m=N$, $\ell=0$. By~\eqref{proj3}, 
$$\Pi\big(\psi^2_N\ov{\psi_0}\big)=\frac{1}{\gamma \sqrt{\pi}}e^{-{2\pi^2N^2}/{\gamma^2}}\psi_{2N}.$$ On the other hand, we observe that $\Gamma_2 \psi_n=(2z + {2i n\pi}/{\gamma}) \psi_n    $. Then, if $V$ satisfies $\big(a+2 \Im (\alpha \ov{\beta})\big)  V +(\alpha \cdot     \Gamma)    V = \Pi \big( |V|^2 V\big)$, then we must have $N=0$, and $\alpha_2=0$. Finally, $U=\kappa R_{{i\pi L }/{\gamma }}\psi_0=\kappa \psi_{L}$, hence the result.
\end{proof}

\section{Linearized analysis around rectangular lattices}\label{rectangular_lattice}

From now on and in the following sections, we assume that $N=1$. In the present section, we assume furthermore that the lattice is rectangular, namely that $\tau_1=0$ and $\tau_2={ \pi}/{\gamma^2}$, so that $\tau= { i\pi}/{\gamma^2}$.

Recall the definition of the strip
$$
S_\gamma = \big\{ z =x+iy: \;\; -{\gamma}/{2} < x \leq {\gamma}/{2} \big\},
$$
which we regard as a fundamental domain for functions in $\mathcal{F}_\gamma$.

Also recall the Hilbert basis
 $$
 \psi_0(z)=\big(\frac{2}{\pi \gamma^2}\big)^{1/4}\EE\Big( \frac{z^2}2-\frac{|z|^2}2\Big), $$
 $$   \psi_k(z)=R_{{ik\pi}/{\gamma}}\psi_0(z)=\big(\frac{2}{\pi \gamma^2}\big)^{1/4}\EE\Big( \frac{2ik\pi} {\gamma} z  -\frac{\pi^2k^2}{\gamma^2}+\frac{z^2}{2}-\frac{|z|^2}2\Big), \;\; k \in \mathbb{Z}.
 $$
It is such that $R_{\alpha}\psi_k=\EE\big({{2ik\pi \alpha}/ {\gamma}}\big) \psi_k$ if $\alpha \in \mathbb{R}$ and $\Gamma_1 \psi_k=\big( {2 \pi k}/{\gamma} \big)\psi_k$. 

By \eqref{exp1} we have the expansion 
\begin{eqnarray*}
\Phi_0(z)&=& \big(\frac{\pi \gamma^2}{2}\big)^{1/4}  \EE\Big(  \frac {\pi^2}{4\gamma^2}\Big)  \sum_{n=-\infty}^{+\infty}   \psi_n (z)\\
&=&\EE\Big(  \frac{z^2}2 -\frac{i\pi z} {\gamma} -  \frac{|z|^2}2\Big) \Theta_{{i\pi}/{\gamma^2}}\Big( { \big( {z-\frac{{i\pi}/{\gamma}-\gamma}2}}\big)/\gamma\Big) ,
\end{eqnarray*}
and it was observed in  Remark \ref{rem14} that this  function gives the $M$-stationary wave
$$ \Phi(t,z) = e^{-i\lambda t}\Phi_0(z).$$
Ignoring its phase, this solution is periodic with respect to translations in $\gamma \mathbb{Z}$ and~$\gamma \tau \mathbb{Z}$, and thus corresponds physically to a rectangular lattice.
The aim of this section is to linearize \eqref{LLL} around the rectangular lattice, and describe the spectrum of this linearization.

Linearizing  equation \eqref{LLL} around $\Phi_0$ gives
\begin{equation}
 i\partial_t u = \Pi\big[2|\Phi|^2u+\Phi^2\ov{u}\big] = \Pi\pac{2|\Phi_0|^2u+e^{-2i\lambda t}\Phi_0^2\ov{u}},
\end{equation}
so that the function $v = \EE\big({i\lambda t}\big)u$ satisfies the equation
\begin{equation}\label{Lin_LLL}
 i\partial_t v + \lambda v = \Pi\pac{2|\Phi_0|^2v+\Phi_0^2\ov{v}}.
\end{equation}

\begin{theorem}
\label{thmunstable}
Writing
$$
v(t) = \sum_{n \in \Z} c_n(t) \psi_n \qquad \mbox{and} \qquad 
f(t,\xi) = \sum_{n \in \Z} c_n(t) e^{-2\pi i n\xi},
$$
the equation \eqref{Lin_LLL} becomes
$$
i \partial_t \begin{pmatrix} {f}(t,\xi) \\[2pt] \overline{{f}}(t,-\xi) \end{pmatrix} = A_{\text{rect}}(\xi) \begin{pmatrix} {f}(t,\xi) \\[2pt] \overline{{f}}(t,-\xi) \end{pmatrix},
$$
for a $2\times 2$ matrix $A_{\text{rect}}(\xi)$. The matrix $A_{\text{rect}}(\xi) $ can be diagonalized for $\xi \neq 0$, resulting in
$$
A_{\text{rect}}(\xi) = P(\xi) \begin{pmatrix} \mu(\xi) & 0 \\[2pt] 0 & -\mu(\xi)  \end{pmatrix} P^{-1}(\xi).
$$

Furthermore,
\begin{itemize}
\item For any $\gamma>0$, $\mu(\xi) \in i \mathbb{R} \setminus \{ 0 \}$ if $\xi$ is close to $0$, but $\xi \neq 0$.
\item If $\tau = i$ (square lattice), then $\mu(\xi)  \in i \mathbb{R} \setminus \{ 0 \}$ for any $\xi \neq 0$.
\end{itemize}
In particular, all rectangular lattices are exponentially unstable in $L^2(S_\gamma) \cap \mathcal{F}_\gamma$.
\end{theorem}

We refer to the proof for exact formulas for $A_{\text{rect}}$, $P$ and $\mu$, which are somewhat lengthy and therefore omitted in the statement.

\begin{remark}
	It turns out that for a general rectangular lattice, $\det(A_{\text{rect}}(\xi))$ \emph{can} take negative values in $(0,1)$, leading to the eigenvalues of $A$ being real. Numerical simulations show that this happens for values of $\gamma$ exceeding $\gamma_0 \approx 2.51.$ In any case, we have $\det(A_{\text{rect}}(\xi)) > 0$ for $\xi >0$ close enough to~$0$, leading to exponential growth for some initial data being non-zero close to $\xi = 0$. 
\end{remark}

\begin{proof} \underline{The linearized operator in the Hilbert basis $(\psi_n)$.}
Setting 
$$\dis v(t,z)=\sum_{n=-\infty}^{+\infty} c_n(t)\psi_n(z),$$
 the equation \eqref{Lin_LLL} becomes
\begin{equation*}
	i\partial_t c_n + \lambda c_n = \pa{\frac{\pi}{2}}^{1/2}\gamma \EE\Big( \frac{\pi^2}{2\gamma^2}\Big)\sum_{\substack {k, \ell, m \in \Z\\ k-\ell+m=n}}A_{k,\ell,m}(2c_{k} + \ov{c_{\ell}}),
\end{equation*}
thanks to the  equation \eqref{proj3}, recalling the notation \eqref{def-A}
$$A_{k,\ell,m} = \frac{1}{\gamma\sqrt{\pi}}\EE\Big({-\frac{\pi^2}{\gamma^2}\pa{(\ell-k)^2+(\ell-m)^2}}\Big).$$
 For $n\in \Z$,
\begin{align*}
	\sum_{\substack {k, \ell, m\in \Z\\ k-\ell+m=n}}A_{k,\ell,m}c_{k} &= \frac{1}{\gamma\sqrt{\pi}}\sum_{k \in \Z} \EE\Big( -\frac{\pi^2}{\gamma^2}(n-k)^2\Big)c_{k} \sum_{\substack {\ell,m \in \Z\\ m-\ell=n-k}}       \EE\Big(  -\frac{\pi^2}{\gamma^2}(\ell-k)^2\Big) \\
		&=\frac{1}{\gamma\sqrt{\pi}}\sum_{k \in \Z} \EE\Big(-\frac{\pi^2}{\gamma^2}(n-k)^2\Big)c_{k}\sum_{q\in\Z}\EE\Big(-\frac{\pi^2q^2}{\gamma^2}\Big),
\end{align*}
and 
\begin{align*}
	\sum_{\substack {k, \ell, m\in \Z\\ k-\ell+m=n}}A_{k,\ell,m}\ov{c_{\ell}} &=  \frac{1}{\gamma\sqrt{\pi}}\sum_{\ell \in \Z} \ov{c_{\ell}} \sum_{\substack {k,m \in \Z\\ k+m=n+\ell}}\EE\Big({-\frac{\pi^2}{\gamma^2}\big((n-k)^2+(\ell-k)^2\big)}\Big) \\
	&=  \frac{1}{\gamma\sqrt{\pi}}\sum_{\ell \in \Z} \ov{c_{\ell}} \sum_{k\in\Z}\EE\Big({-\frac{\pi^2k^2}{\gamma^2}}\Big)\EE\Big({-\frac{\pi^2}{\gamma^2}(n-\ell-k)^2}\Big). 
\end{align*}
In other words,
\begin{equation*}
	i\partial_t \begin{pmatrix} c_n \\[2pt] d_n \end{pmatrix} = \begin{pmatrix} \mathcal{L}^{\text{rect}}-\lambda Id & \mathcal{M}^{\text{rect}} \\[2pt] -\mathcal{M}^{\text{rect}} & - (\mathcal{L}^{\text{rect}}-\lambda Id) \end{pmatrix} \begin{pmatrix} c_n \\[2pt] d_n \end{pmatrix},
\end{equation*}
where $(d_n) := (\ov{c_n})$ and
\begin{equation*}  
	\mathcal{L}^{\text{rect}} : u \longmapsto C_L L * u, \quad C_L =\sqrt{2}\EE\Big({\frac{\pi^2}{2\gamma^2}}\Big)\sum_{q\in\Z}\EE\Big(-\frac{\pi^2q^2}{\gamma^2}\Big),  
\end{equation*}	
\begin{equation} \label{def_op_L_rect} 
L_n = \EE\Big(-\frac{\pi^2n^2}{\gamma^2}\Big)
\end{equation}	
	\begin{equation*}
	\mathcal{M}^{\text{rect}} : u \longmapsto C_M M * u, \quad C_M =\frac{1}{\sqrt{2}}\EE\Big({\frac{\pi^2}{2\gamma^2}}\Big), 
		\end{equation*}
		\begin{equation}\label{def_op_M_rect}
	M_n = \sum_{p=-\infty}^{+\infty}\EE\Big({-\frac{\pi^2p^2}{\gamma^2}}\Big) \EE\Big({-\frac{\pi^2}{\gamma^2}(n-p)^2}\Big) = [L * L](n)
\end{equation}	
where $*$ stands for the discrete convolution of sequences.

\medskip 

\noindent \underline{The linearized operator in the Fourier variable $\xi$.}
 Define the discrete Fourier transform $\dis \mathcal{F}:\ell^2(\Z) \to L^2([0,1])$ \big(respectively the inverse Fourier transform $\mathcal{F}^{-1}$\big) for a sequence $\dis u = (u_n)_{n\in\Z} \in \ell^2(\Z)$ \big(respectively a function $f \in L^2([0,1])$\big) by 
\begin{equation}
\label{deffourier}
\mathcal{F}(u)(\xi) = \sum_{k= -\infty}^{+\infty} u_ke^{- 2i\pi k \xi}, \qquad \mathcal{F}^{-1}(f)(k) = \int_0^1 e^{2i\pi k \xi} f(\xi) d\xi.
\end{equation}
As is well-known, $\mathcal{F}(u*v) = \mathcal{F}(u) \mathcal{F}(v)$, so that
\begin{align*}
& \mathcal{F}\mathcal{L}^{\text{rect}}\mathcal{F}^{-1}f=C_L\mathcal{F}(L) f \\
& \mathcal{F}\mathcal{M}^{\text{rect}}\mathcal{F}^{-1}f=C_M\mathcal{F}(M) f = C_M\mathcal{F}(L*L) f = C_M \mathcal{F}(L)^2 f.
\end{align*}
The Fourier transform of $L$ is given by
\begin{multline}\label{def_fct_l}
	\mathcal{F}(L)(\xi) =  \sum_{k= -\infty}^{+\infty} L_ke^{-{2i\pi k}\xi} = \sum_{k = -\infty}^{+\infty} \EE\Big({-{2i\pi k}\xi-\frac{\pi^2 k^2}{\gamma^2}}\Big)= \\
	=1 + 2 \sum_{k=1}^{+\infty}\EE\Big({-\frac{\pi^2k^2}{\gamma^2}}\Big)\cos\pa{2\pi k\xi}=: \ell(\xi).
\end{multline}

All functions and constants above can now be expressed through $\ell$ and $C_M$:
$$
C_L = 2 C_M \ell(0), \quad \lambda = C_M \ell^2(0), \quad \mathcal{F}(M)(\xi) = \ell^2(\xi).
$$

Overall, we find the expression
\begin{equation*}
	\pac{\mathcal{F}\begin{pmatrix} \mathcal{L}^{\text{rect}}-\lambda Id & \mathcal{M}^{\text{rect}} \\[2pt]  -\mathcal{M}^{\text{rect}} & - (\mathcal{L}^{\text{rect}}-\lambda Id) \end{pmatrix}\mathcal{F}^{-1}(X)}\pa{\xi} =:A_{\text{rect}}(\xi)X,
\end{equation*}
where
\begin{equation*}
	A_{\text{rect}}(\xi) = \begin{pmatrix} a(\xi) & b(\xi) \\[2pt] -b(\xi) & -a(\xi) \end{pmatrix} =C_M \begin{pmatrix} \ell(0)(2\ell(\xi)-\ell(0)) & \ell^2(\xi) \\[2pt]  -\ell^2(\xi) & -\ell(0)(2\ell(\xi)-\ell(0)) \end{pmatrix}.
\end{equation*}

\medskip

\noindent \underline{Diagonalizing the matrix $A_{\text{rect}}$.}
The characteristic polynomial of $A_{\text{rect}}(\xi)$ is 
$$
P_{A_{\text{rect}}(\xi)}(X) = X^2 + D(\xi), \quad \mbox{where} \quad D(\xi) = \det(A_{\text{rect}}(\xi)) = b^2(\xi) - a^2(\xi).
$$
Whether the eigenvalues at frequency $\xi$ are real or imaginary (and therefore, stable or unstable) depends on the sign of $D(\xi)$, which can be factorized as follows
\begin{align*}
 D(\xi)  &= C_M^2\pac{\ell^2(\xi)+\ell(0)(2\ell(\xi)-\ell(0))}\pac{\ell^2(\xi)-\ell(0)(2\ell(\xi)-\ell(0))} \\
    &= \underbrace{C_M^2\pac{\ell(\xi)-\ell(0)}^2\pac{\ell(\xi)+(1+\sqrt{2})\ell(0)}}_{\geq 0}\pac{\ell(\xi)-(\sqrt{2}-1)\ell(0)},
\end{align*}
where the above term is non-negative since $|\ell(\xi)|$ is maximum at $\xi=0$. We will denote
$$
\mu(\xi) = \sqrt{-D(\xi)},
$$
with the convention that $\mathfrak{Im} \sqrt x \geq 0$ if $x \leq 0$. The eigenvalues of $A_{\text{rect}}(\xi)$ are $\mu_{\pm} (\xi)= \pm \sqrt{-D(\xi)}$, and the corresponding eigenvectors $e_{\pm}(\xi) = (b(\xi), \mu_{\pm}(\xi) - a(\xi))$. Denoting~$P$ for the change of basis matrix $P(\xi) = (e_+(\xi) | e_-(\xi))$, we obtain the diagonalization formula
$$
A_{\text{rect}}(\xi) = P(\xi) \begin{pmatrix} \mu(\xi) & 0 \\ 0 & -\mu(\xi) \end{pmatrix} P^{-1}(\xi).
$$
 
\medskip

\noindent \underline{The sign of the determinant.}
As demonstrated above, $D(\xi)$ and 
$$F(\xi) :=\ell(\xi) -(\sqrt{2}-1)\ell(0)$$
 have the same sign. By continuity of $F$, it appears that $F(\xi)$, and hence $D(\xi)$, is positive in a neighborhood of $\xi=0$.

To obtain a more precise bound, we resort to the geometric series formula to bound~$F$ from below:
\begin{multline*}
F(\xi) =2 - \sqrt{2} + 2 \sum_{k=1}^{+\infty}\EE\Big(  -\frac{\pi^2 k^2}{\gamma^2}\Big) \pa{\cos\pa{2\pi k\xi}-\sqrt{2}+1}\\
 \geq 2 - \sqrt{2} - 2 \sum_{k=1}^{+\infty}\sqrt{2}\EE\Big(  -\frac{\pi^2 k^2}{\gamma^2}\Big) \geq 2 - \sqrt{2} - 2\sqrt{2}\frac{q}{1-q},
\end{multline*}
where $q := \EE\big(-{\pi^2}/{\gamma^2}\big)$. Then, $F(\xi) > 0$ as soon as $$q < \frac{2-\sqrt{2}}{2\sqrt{2}+2-\sqrt{2}} \approx 0.1716,$$
which is the case for the square lattice, for which $\gamma^2 = \pi$, so that we have \linebreak ${q = \EE\big( {-\pi} \big)\approx 0.0432\dots}$

\medskip

\noindent \underline{Linear instability.} We claim that the equation \eqref{Lin_LLL} is unstable in $L^2(S_\gamma) \cap \mathcal{F}_\gamma$. We note first that
$$
\| v(t,z) \|_{L^2(S_\gamma)} = \| c_n(t) \|_{\ell^2(\mathbb{Z})} = \|  f(t,\xi) \|_{L^2([0,1])}
$$
(where the first equality is a consequence of the fact that $(\psi_n)$ is a Hilbert basis, and the second equality is Parseval's theorem).

As a consequence, it suffices to establish instability in the $ {f}$ variable. Set $f_0(\xi)=f(0,\xi)$. By the above diagonalization formula, $ {f}$ satisfies
$$
\begin{pmatrix} {f}(t,\xi) \\[2pt] \overline{{f}}(t,-\xi) \end{pmatrix} = P(\xi) \begin{pmatrix} e^{-i t \mu(\xi)} & 0 \\[2pt] 0 & e^{i t \mu(\xi)}\end{pmatrix} P^{-1}(\xi) \begin{pmatrix} {f}_0(\xi) \\[2pt] \overline{{f}_0}(-\xi) \end{pmatrix}.
$$

As we saw above, $\mathfrak{Re} (-i \mu(\xi)) > 0$ for $\xi$ close to zero. Therefore, exponential growth will be observed as soon as the first coordinate of $P^{-1}(\xi) \begin{pmatrix} {f}_0(\xi) \\[2pt] \overline{{f}_0}(-\xi) \end{pmatrix}$ is non-zero. This is easy to arrange, and it is even the generic situation!  As an example, we give $v(0) = \psi_0$, leading to $ {f}_0(\xi) = 1$, and $$P^{-1}(\xi) \begin{pmatrix} {f}_0(\xi) \\[2pt] \overline{{f}_0}(-\xi) \end{pmatrix} = \begin{pmatrix}\frac{b(\xi)+a(\xi)}{\mu(\xi)(b(\xi)+a(\xi)-\mu(\xi))} \\[2pt] \cdots \end{pmatrix},$$
whose first coordinate is non-zero. 
\end{proof}
 
\section{Linearized analysis around the hexagonal lattice} 
\label{section5}

\subsection{Writing the problem in the orthonormal basis}

\label{subsect6.1}

We consider the lattice $\mathcal{L}_{\tau,\gamma} = \gamma(\Z \oplus \tau \Z)$ for $N=1$ and $\tau = j = \EE\big({{2i\pi}/{3}}\big)$.  This is the hexagonal lattice (or Abrikosov lattice) $\mathcal{L}_{j,\gamma}$ for which ${\sqrt{3}}/{2} = \tau_2 = {\pi}/{\gamma^2}$ so that $$ \gamma^2 = \frac{2\pi}{\sqrt{3}} \approx 3.6276.$$

We denote again
$$\psi_0(z):=\big(\frac{2}{\pi \gamma^2}\big)^{1/4}\EE\Big( {z^2/2-|z|^2/2}\Big),$$
but for  $k\in \Z$, the family $(\psi_k)$ will be given by
\begin{equation}\label{newdef}
\psi_k(z):=R_{k\tau \gamma}\psi_0(z)=\big(\frac{2}{\pi \gamma^2}\big)^{1/4}\EE\Big( {{2ik\pi z}/ {\gamma} +i\pi \tau k^2+ z^2/2-|z|^2/2}\Big).
\end{equation}

Also recall the definition of the vertical strip of width $\gamma$:
$$
S_{\gamma} = \big\{z \in \mathbb{C}, \quad -{\gamma}/{2} < \Re  z \leq {\gamma}/{2}  \big\}.
$$

The new definition of the $(\psi_k)$ is meant to simplify computations hereafter. With this new definition, lemmas \ref{lem98} and \ref{lem-ortho} need to be adapted, which is the purpose of the following lemma.  

From now on, we denote
\begin{equation}\label{def-B}
B_{k,\ell,m} =  \frac{1}{\gamma\sqrt{\pi}}\EE\Big( -\frac{\pi^2}{\gamma^2}\big((\ell-k)^2+(\ell-m)^2\big)\Big)(-1)^{(\ell-k)(\ell-m)} .
\end{equation}
Then, the following result holds true.

\begin{lemme}\label{lem-ortho-hexa}
\begin{enumerate}[$(i)$]
\item The family $(\psi_k)_{k \in \Z}$ is orthonormal and forms a  Hilbertian basis of $\mathcal{F}_{\gamma} \cap L^{2}(S_{ \gamma})$.
\item If $k,\ell,m, n \in \mathbb{Z}$,
\begin{equation*}
  \int_{S_{\gamma}} \psi_{k}(z)\ov{\psi_{\ell}(z)} \psi_{m}(z)\ov{\psi_{n}(z)} \,dL(z) =
  \begin{cases}
B_{k,\ell,m}   & \mbox{if $k - \ell + m - n = 0$} \\
 0 & \mbox{if $k - \ell + m - n \neq 0$}
 \end{cases}.
\end{equation*}
In other words,
\begin{equation}\label{proj4}
            \Pi\big(\psi_{k}\ov{\psi_{\ell}} \psi_{m}\big)=B_{k,\ell,m}   \psi_{n},
        \end{equation}
with $n = k - \ell + m$.
    \end{enumerate}
\end{lemme}

\begin{proof} $(i)$ For all $k,\ell \in \Z$, and $z=x+iy \in \C$,
    \begin{multline*}
     \psi_{k}(z)\ov{\psi_{\ell}(z)}=\\
   =   \pa{\frac{2}{\pi \gamma^2}}^{1/2}\EE\Big({i\pi(\tau k^2-\overline{\tau}\ell^2)}\Big)  \EE\Big( {-2y^2+\frac{2i\pi}{\gamma}(k-\ell)x-\frac{2\pi}{\gamma}(k+\ell)y}\Big),
          \end{multline*}
    which gives
              \small
        \begin{multline*}
          \int_{S_{\gamma}} \psi_{k}(z)\ov{\psi_{\ell}(z)} dL(z) = \\
    \begin{aligned}
&    =    \pa{\frac{2}{\pi \gamma^2}}^{1/2} e^{i\pi(\tau k^2-\overline{\tau}\ell^2)} \int_{y\in\R}\int_{- \frac \gamma 2 < x < \frac \gamma 2} \EE\Big(  {-2y^2+\frac{2i\pi}{\gamma}(k-\ell)x-\frac{2\pi}{\gamma}(k+\ell)y}\Big)\,dx\, dy \\
        &= \pa{\frac{2}{\pi \gamma^2}}^{1/2}e^{i\pi k^2(\tau -\overline{\tau})} \delta_{k,\ell} \gamma \int_{y\in\R}\EE\Big( {-2y^2-\frac{4\pi}{\gamma}ky}\Big) \, dy \\
        &= \delta_{k,\ell}
    \end{aligned}
            \end{multline*}
               \normalsize
    since $\tau - \ov{\tau} = 2i \tau_2 = 2i{\pi}/{\gamma^2}$ and by the Gaussian integral \eqref{GaussianIntegral}. We can prove that the family is a Hilbertian basis as in Lemma \ref{lem98}. 
    
\medskip

\noindent $(ii)$ For $k,\ell,m,n \in \Z$, and $z=x+iy \in \C$,
    \begin{multline*}
     \psi_{k}(z)\ov{\psi_{\ell}(z)}\psi_{m}(z)\ov{\psi_{n}(z)} = \\
     =\frac{2}{\pi \gamma^2} \EE\Big(    {i\pi\tau (k^2+m^2)-i\pi\overline{\tau}(\ell^2+n^2)}\Big) \qquad  \qquad  \qquad  \\
    \times  \EE\Big(  {-4y^2+\frac{2i\pi}{\gamma}(k-\ell+m-n)x-\frac{2\pi}{\gamma}(k+\ell+m+n)y}\Big).
        \end{multline*}
If $k-\ell+m-n \neq 0$, integrating in $x$ gives $0$. Otherwise, $k-\ell+m-n = 0$ and we obtain 
\small
    \begin{multline*}
            \int_{S_{\gamma}} \psi_{k}(z)\ov{\psi_{\ell}(z)} \psi_{m}(z)\ov{\psi_{n}(z)} \, dL(z) =  \\
    \begin{aligned}
&=\frac{2}{\pi \gamma}  \EE\Big(    i\pi\tau (k^2+m^2)-i\pi\overline{\tau}(\ell^2+n^2) \Big) \int_{y\in\R} \EE\Big(  -4y^2-\frac{2\pi}{\gamma}(k+\ell+m+n)y  \Big)\,dy \\
        &= \frac{2}{\pi \gamma} \EE\Big(    {i\pi\tau (k^2+m^2)-i\pi\overline{\tau}(\ell^2+n^2)} \Big) \frac{\sqrt{\pi}}{2}  \EE\Big( {\frac{\pi^2(k+\ell+m+n)^2}{4\gamma^2}} \Big) \\
        &= \frac{1}{\sqrt{\pi}\gamma}e^{i\pi Q(k,\ell,m,n)}
    \end{aligned}
     \end{multline*}  
     \normalsize 
by \eqref{GaussianIntegral}, where 
$$Q(k,\ell,m,n) := \tau (k^2+m^2) - \overline{\tau}(\ell^2+n^2) -i \dfrac{\pi(k+\ell+m+n)^2}{4\gamma^2}.$$
    We use that $\tau - \ov{\tau} = 2i {\pi}/{\gamma^2}$ and $k+\ell+m+n = 2(k+m) = 2(\ell+n)$ to derive
    \begin{align*}
        Q(k,\ell,m,n) &= \tau (k^2+m^2) - \overline{\tau}(\ell^2+n^2) - \frac{1}{8}\big(k+\ell+m+n\big)^2(\tau-\ov{\tau})\\
        &= \tau\Big({k^2+m^2-\frac12(k+m)^2}\Big) - \ov{\tau}\Big({\ell^2+n^2-\frac12(\ell+n)^2} \Big)\\ 
        &= \frac{\tau}{2}(k-m)^2 - \frac{\ov{\tau}}{2}(\ell-n)^2.
    \end{align*}
 This gives
\begin{equation}\label{proj_4psi_general}
        Q(k,\ell,m,n) = \frac{1}{2}(\tau-\ov{\tau})\big({(\ell-k)^2 + (\ell-m)^2}\big) + \pa{\tau+\ov{\tau}}(k-\ell)(\ell-m),
    \end{equation}
where we used the identities
\begin{align*}
& (k-m)^2 = \big((k-\ell)+(\ell-m)\big)^2 = (\ell-k)^2 + (\ell-m)^2 -2(\ell-k)(\ell-m) \\
& (\ell-n)^2 = \big((\ell-k)+(\ell-m)\big)^2 = (\ell-k)^2 + (\ell-m)^2 +2(\ell-k)(\ell-m).
\end{align*}
We now rely on the specific value $\tau = \EE\big({{2i\pi}/{3}}\big)$, giving $\tau + \ov{\tau} = -1,$ which had not been used so far.

Overall, we proved that
    \begin{multline*}
            \int_{S_{\gamma}} \psi_{k}(z)\ov{\psi_{\ell}(z)} \psi_{m}(z)\ov{\psi_{n}(z)}\,dL(z) =\\
            \begin{aligned}
&= \frac{1}{\sqrt{\pi}\gamma}   \EE\Big( { \frac{i\pi}{2}(\tau-\ov{\tau})\pa{(\ell-k)^2 + (\ell-m)^2}} \Big) \EE \Big({i\pi(\ell-k)(\ell-m)}\Big) \\
        & = \frac{1}{\sqrt{\pi}\gamma} \EE\Big(  -\frac{\pi^2}{\gamma^2}\pa{(\ell-k)^2 + (\ell-m)^2} \Big) (-1)^{(\ell-k)(\ell-m)},
                   \end{aligned}
        \end{multline*}
    which was the claim.
\end{proof}

\begin{remark}\label{rk_tau_general}
	The equation \eqref{proj_4psi_general} does not rely on any specific value of $\tau$. The fact that the projection~\eqref{proj4} is real-valued for any values of the integers $k, \ell, m,n$ is specific to the two cases of rectangular and hexagonal lattices, at least for $|\tau| = 1$. Our conjecture is that other lattices would still give a convolution structure (see hereafter), but with much more complexity in the computations, leading to a complex, non-real matrix~$A_\tau(\xi)$, and in the end, to instability.
\end{remark}

With the help of the above lemma, we can write \eqref{LLL} in the orthonormal basis: if $u$ is expanded as $\sum c_n \psi_n$, then the coefficients $(c_n)$ satisfy the equation
\begin{equation}
\label{eqinck}
i \partial_t c_n = \sum_{\substack{k,\ell,m \in \mathbb{Z} \\ k - \ell + m = n}} B_{k,\ell,m} c_{k} \ov{c_{\ell}} c_{m}, \quad n \in \Z.
\end{equation}

\subsection{Linearizing around the Abrikosov lattice}

By Proposition \ref{prop34}, the following function is a stationary solution of \eqref{LLL}:
\begin{equation*}
\Psi(t,z) = e^{-i\lambda t}\Psi_0(z)
\end{equation*}
\begin{equation}\label{def1.1}
 \Psi_0(z) := \kappa \Phi_0(z)= \kappa  \EE\Big({z^2/2-|z|^2/2-{i\pi z}/{\gamma}} \Big)\Theta_\tau\big({ \pa{z-z_0}/\gamma}\big),
\end{equation}
with $z_0 = {\gamma}(\tau-1)/2$, and where the constants are $\kappa=  \EE\big( {i{\pi}/{4}\tau}\big) \pa{{2}/({\pi\gamma^2})}^{1/4}$ and $\lambda = \lambda_0 |\kappa|^2$. By \eqref{exp1}, $\Phi_0$ can be expanded in the orthonormal basis as
$$
\Psi_0(z) = \sum_{n=-\infty}^{+\infty} \psi_n(z)
$$
(and the reason for the choice of $\kappa$ is now apparent: in the basis $(\psi_k)$,  all the coefficients are $1$). Abusing the terminology slightly, we call this solution the \emph{Abrikosov lattice}, or \emph{hexagonal lattice}.

Linearizing \eqref{LLL} around the Abrikosov lattice $\Psi$ gives
\begin{equation}
 i\partial_t u = \Pi\pac{2|\Psi|^2u+\Psi^2\ov{u}} = \Pi\pac{2|\Psi_0|^2u+e^{-2i\lambda t}\Psi_0^2\ov{u}}.
\end{equation}
Then the function $v = e^{i\lambda t}u$ satisfies the equation
\begin{equation}\label{lin_LLL_hexa}
\tag{LLLhexa}
 i\partial_t v + \lambda v = \Pi\pac{2|\Psi_0|^2v+\Psi_0^2\ov{v}}.
\end{equation}

\begin{lemme}
\label{lemmaformula}
 Writing
$$
v(t) = \sum_{n \in \Z} c_n(t) \psi_n \qquad \mbox{and} \qquad 
f(t,\xi) = \sum_{n \in \Z} c_n(t) e^{-2\pi i n\xi},
$$
the equation \eqref{lin_LLL_hexa} becomes
$$
i \partial_t \begin{pmatrix} {f}(t,\xi) \\[3pt] \overline{{f}}(t,-\xi) \end{pmatrix} =  \begin{pmatrix} a(\xi) & b(\xi) \\[3pt] -b(\xi) & -a(\xi) \end{pmatrix} \begin{pmatrix} {f}(t,\xi) \\[3pt] \overline{{f}}(t,-\xi) \end{pmatrix} = A_{\text{hexa}}(\xi) \begin{pmatrix} {f}(t,\xi) \\[3pt] \overline{{f}}(t,-\xi) \end{pmatrix},
$$
with
$$
\begin{cases}
\dis a(\xi) := \frac{2}{\gamma\sqrt{\pi}}\Big({\ell(\xi)\ell(0)-2h(\xi)h(0)}\Big) - \frac{1}{\gamma\sqrt{\pi}}\pa{\ell^2(0)-2h^2(0)} \\[9pt]
\dis b(\xi) := \frac{1}{\gamma\sqrt{\pi}}\pa{\ell^2(\xi)-2h^2(\xi)}
\end{cases}
$$
and where the functions $\ell$ and $h$ are defined by
$$
\begin{cases}
\dis \ell(\xi) := 1 + 2 \sum_{k=1}^{+\infty}\EE\Big(-\frac{\pi^2k^2}{\gamma^2}\Big)\cos\big({2\pi k\xi}\big) \\[6pt]
\dis h(\xi) := 2 \sum_{k=0}^{+\infty}\EE\Big(-\frac{\pi^2(2k+1)^2}{\gamma^2}\Big)\cos\big({2(2k+1)\pi \xi}\big).
\end{cases}
$$
\end{lemme}

\begin{proof}
Expanding $v$ in the orthonormal basis $\dis v(t,z)=\sum_{n=-\infty}^{+\infty} c_n(t)\psi_n(z)$, the coordinates satisfy the equation
\begin{equation}\label{Lin_coeff_hexa}
	i\partial_t c_n + \lambda c_n = \sum_{\substack {k, \ell, m \in \Z\\ k-\ell+m=n}}B_{k,\ell,m}(2c_{k} + \ov{c_{\ell}}),
\end{equation}
thanks to \eqref{eqinck}.

We will now compute the sums
$$ \sum_{\substack {k, \ell, m\in \Z\\ k-\ell+m=n}}B_{k,\ell,m}c_{k} \qquad \text{and} \qquad  \sum_{\substack {k, \ell, m \in \Z\\ k-\ell+m=n}}B_{k,\ell,m}\ov{c_{\ell}}.$$
by considering separately the cases $H=(\ell-k)(\ell-m)$ even and $H$ odd. Note that~$H$ is odd if and only if ($\ell$ is even and $k,m$ are odd) or ($\ell$ is odd and $k,m$ are even). We have 
\begin{multline*}
\sum_{\substack {k, \ell, m \in \Z\\ k-\ell+m=n}}B_{k,\ell,m}c_{k} = \\
\begin{aligned}
	&=  \frac{1}{\gamma\sqrt{\pi}}\sum_{\substack {k, \ell, m \in \Z\\ k-\ell+m=n, \; H\; \text{even}}}\EE\Big({-\frac{\pi^2}{\gamma^2}\pa{(\ell-k)^2+(\ell-m)^2}}\Big)c_{k} \\
	&\hspace{3cm}- \frac{1}{\gamma\sqrt{\pi}}\sum_{\substack {k, \ell, m \in \Z\\ k-\ell+m=n, \;  H\; \text{odd}}}\EE\Big({-\frac{\pi^2}{\gamma^2}\pa{(\ell-k)^2+(\ell-m)^2}}\Big)c_{k} \\
 &= \frac{1}{\gamma\sqrt{\pi}}\sum_{\substack {k, \ell, m \in \Z \\ k-\ell+m=n }}\EE\Big({-\frac{\pi^2}{\gamma^2}\pa{(\ell-k)^2+(\ell-m)^2}}\Big)c_{k} \\
&\hspace{3cm} - \frac{2}{\gamma\sqrt{\pi}}\sum_{\substack {k, \ell, m \in \Z\\ k-\ell+m=n,  \; H\; \text{odd}}}\EE\Big({-\frac{\pi^2}{\gamma^2}\pa{(\ell-k)^2+(\ell-m)^2}}\Big)c_{k}.
\end{aligned}
\end{multline*}
The first sum is the same as in Section \ref{rectangular_lattice} above, and can be computed accordingly. We then focus on the second sum. In the case $H$ odd, $k$ and $m$ have same parity, so that $\ell$ and $n=k-\ell+m$ also have the same parity.

\smallskip 

\noindent  \underline{First case: $\ell$ even.} The condition $H$ odd means that~$k$ and $m$ are both odd. We compute 
\begin{align*}
	S_1(n) & := \sum_{\substack {\ell \in 2\Z \\k,  m \in \Z \\ k-\ell+m=n,  \; H\; \text{odd}}}\EE\Big({-\frac{\pi^2}{\gamma^2}\pa{(\ell-k)^2+(\ell-m)^2}}\Big)c_{k} \\
	&= \sum_{k\in 2\Z+1}\EE\Big({-\frac{\pi^2}{\gamma^2}(n-k)^2}\Big)c_{k} \sum_{\substack {\ell \in 2\Z \\ m \in 2\Z+1 \\ m-\ell = n-k}}\EE\Big({-\frac{\pi^2}{\gamma^2}(\ell-k)^2}\Big) \\
	&= \sum_{k\in 2\Z+1}\EE\Big({-\frac{\pi^2}{\gamma^2}(n-k)^2}\Big)c_{k} \sum_{q \in 2\Z+1}\EE\Big({-\frac{\pi^2}{\gamma^2}q^2}\Big) \\
	& = \big(T^{\text{odd}}L^{\text{odd}} * c\big)(n),
\end{align*}
where 
\begin{equation*}
T^{\text{odd}} = \sum_{q \in 2\Z+1}\EE\Big({-\frac{\pi^2q^2}{\gamma^2}} \Big)
\end{equation*}
and 
\begin{equation}\label{def-Lodd}
    L_m^{\text{odd}} = \delta_{m\in2\Z+1}\EE\Big({-\frac{\pi^2m^2}{\gamma^2}} \Big) = \left\{
        \begin{array}{ll}
        \EE\Big({-\frac{\pi^2m^2}{\gamma^2}} \Big) & \text{if} \; m \in 2\Z+1, \\
        0 \quad & \text{if} \; m \in 2\Z.
        \end{array}
    \right.
\end{equation}

Similarly, we have
\begin{multline*}
	\sum_{\substack {k, \ell, m \in \Z\\ k-\ell+m=n}}B_{k,\ell,m}\ov{c_{\ell}} 
	= \frac{1}{\gamma\sqrt{\pi}}\sum_{\substack {  k, \ell, m \in \Z \\ k-\ell+m=n }}\EE\Big({-\frac{\pi^2}{\gamma^2}\pa{(\ell-k)^2+(\ell-m)^2}}\Big)\ov{c_{\ell}} \\- \frac{2}{\gamma\sqrt{\pi}}\sum_{\substack { k, \ell, m \in \Z \\ k-\ell+m=n,  \;H\; \text{odd}}}\EE\Big({-\frac{\pi^2}{\gamma^2}\pa{(\ell-k)^2+(\ell-m)^2}}\Big)\ov{c_{\ell}},
\end{multline*}
the first sum being the same as in Section \ref{rectangular_lattice} above. We compute 
\begin{align*}
	S_2(n) &:= \sum_{\substack {\ell \in 2\Z\\ k,  m \in \Z \\ k-\ell+m=n,  \; H\; \text{odd}}}\EE\Big({-\frac{\pi^2}{\gamma^2}\pa{(\ell-k)^2+(\ell-m)^2}}\Big)\ov{c_{\ell}} \\
	&= \sum_{\ell\in 2\Z}\ov{c_{\ell}} \sum_{\substack {k, m  \in 2\Z+1   \\ m-\ell = n-k}}\EE\Big({-\frac{\pi^2}{\gamma^2}(\ell-k)^2}\Big) \EE\Big({-\frac{\pi^2}{\gamma^2}(n-k)^2}\Big) \\
	&= \sum_{\ell\in 2\Z}\ov{c_{\ell}} \sum_{k \in 2\Z+1}\EE\Big({-\frac{\pi^2}{\gamma^2}(\ell-k)^2}\Big) \EE\Big({-\frac{\pi^2}{\gamma^2}(n-k)^2} \Big),
	\end{align*}
Then, with a change of variables, we get 
\begin{align*}
	S_2 (n)&= \sum_{\ell\in 2\Z}\ov{c_{\ell}} \sum_{m \in 2\Z+1}\EE\Big({-\frac{\pi^2}{\gamma^2}m^2}\Big) \EE\Big({-\frac{\pi^2}{\gamma^2}(n-m-\ell)^2}\Big)  \\  
	&= \sum_{\ell\in \Z}\ov{c_{\ell}}\delta_{\ell\in2\Z} \sum_{m \in 2\Z+1}\EE\Big({-\frac{\pi^2}{\gamma^2}m^2}\Big) \EE\Big({-\frac{\pi^2}{\gamma^2}(n-m-\ell)^2} \Big)\\
	&= \sum_{\ell\in \Z}\ov{c_{\ell}} \sum_{m \in 2\Z+1}\delta_{(n-\ell)\in2\Z}\EE\Big({-\frac{\pi^2}{\gamma^2}m^2}\Big) \EE\Big({-\frac{\pi^2}{\gamma^2}(n-m-\ell)^2} \Big)\\
	&=\big( M^{\text{odd}} * \ov{c}\big)(n),
\end{align*}

with 
\begin{align}\label{defModd}
	M^{\text{odd}}(n) &:=\sum_{k \in 2\Z+1} \EE\Big(-\frac{\pi^2}{\gamma^2}k^2\Big) \delta_{n\in2\Z}\EE\Big(-\frac{\pi^2}{\gamma^2}(n-k)^2\Big) \nonumber \\
	&= \sum_{k \in \Z} \delta_{k\in2\Z+1}\EE\Big(-\frac{\pi^2}{\gamma^2}k^2\Big) \delta_{(n-k)\in2\Z+1}\EE\Big(-\frac{\pi^2}{\gamma^2}(n-k)^2\Big) \nonumber \\
	&= \big(L^{\text{odd}} * L^{\text{odd}}\big) (n).
\end{align}

\medskip
\noindent  \underline{Second case: $\ell$ odd.} The condition $H$ odd means that~$k$ and $m$ are both even. We compute just as $S_1$:
\begin{align*}
	S_3(n) &:= \sum_{\substack {  \ell \in 2\Z+1\\k,  m \in \Z \\ k-\ell+m=n,  \;H\; \text{odd}}}\EE\Big(-\frac{\pi^2}{\gamma^2}\pa{(\ell-k)^2+(\ell-m)^2}\Big)c_{k} \\
	&= \sum_{k\in 2\Z}\EE\Big(-\frac{\pi^2}{\gamma^2}(n-k)^2\Big)c_{k} \sum_{\substack {\ell \in 2\Z+1 \\ m \in 2\Z \\ m-\ell = n-k}}\EE\Big(-\frac{\pi^2}{\gamma^2}(\ell-k)^2\Big) \\
	&= T^{\text{odd}}\sum_{k\in \Z}\delta_{(n-k)\in2\Z+1}\EE\Big(-\frac{\pi^2}{\gamma^2}(n-k)^2\Big)c_{k} \\
	&= \big( T^{\text{odd}}L^{\text{odd}} * c\big)(n),
\end{align*}
which is the same expression as in the case $\ell$ even. Now, we compute the last sum:
\begin{align*}
	S_4(n) &:= \sum_{\substack {\ell \in 2\Z+1\\k,  m \in \Z \\ k-\ell+m=n,  \;H\; \text{odd}}}\EE\Big(-\frac{\pi^2}{\gamma^2}\pa{(\ell-k)^2+(\ell-m)^2}\Big)\ov{c_{\ell}} \\
	&= \sum_{\ell\in 2\Z+1}\ov{c_{\ell}} \sum_{\substack {k,m  \in 2\Z   \\ m-\ell = n-k}}\EE\Big(-\frac{\pi^2}{\gamma^2}(\ell-k)^2\Big) \EE\Big(-\frac{\pi^2}{\gamma^2}(n-k)^2\Big) \\
	&= \sum_{\ell\in 2\Z+1}\ov{c_{\ell}}  \sum_{j \in 2\Z+1}\EE\Big(-\frac{\pi^2}{\gamma^2}j^2\Big) \EE\Big(-\frac{\pi^2}{\gamma^2}(n-j-\ell)^2\Big)   
\end{align*}
Therefore we get 
\begin{align*}
	S_4(n) &= \sum_{\ell\in \Z}\ov{c_{\ell}}\delta_{\ell\in2\Z+1} \sum_{j \in 2\Z+1}\EE\Big(-\frac{\pi^2}{\gamma^2}j^2\Big)\EE\Big(-\frac{\pi^2}{\gamma^2}(n-j-\ell)^2\Big) \\
	&= \sum_{\ell\in \Z}\ov{c_{\ell}} \sum_{j \in 2\Z+1}\delta_{(n-\ell)\in2\Z}\EE\Big(-\frac{\pi^2}{\gamma^2}j^2\Big) \EE\Big(-\frac{\pi^2}{\gamma^2}(n-j-\ell)^2\Big)\\
	& =\big( M^{\text{odd}} * \overline{c}\big)(n),
\end{align*}
with $M^{\text{odd}}$ defined in \eqref{defModd}, and this is the same expression as $S_2$ above.

\medskip
Overall, we get 
\begin{equation}\label{convol}
	i\partial_t \begin{pmatrix} c_n \\[2pt] d_n \end{pmatrix} = \begin{pmatrix} \mathcal{L}^{\text{hexa}}-\lambda Id & \mathcal{M}^{\text{hexa}} \\[2pt] -\mathcal{M}^{\text{hexa}} & - (\mathcal{L}^{\text{hexa}}-\lambda Id) \end{pmatrix} \begin{pmatrix} c_n \\[2pt] d_n \end{pmatrix},
\end{equation}
where $(d_n) := (\ov{c_n})$, and
\begin{align*}
	\mathcal{L}^{\text{hexa}} &: u \longmapsto \pa{C_L' L- C_L^{\text{odd}}L^{\text{odd}}}* u, &C'_L=\frac{2}{\gamma\sqrt{\pi}}T, \qquad  &C_L^{\text{odd}}=\frac{4}{\gamma\sqrt{\pi}}T^{\text{odd}},\\
	\mathcal{M}^{\text{hexa}} &: u \longmapsto \pa{C_M' M- C_M^{\text{odd}}M^{\text{odd}}}* u,   &C'_M=\frac{1}{\gamma\sqrt{\pi}}, \qquad  &C_M^{\text{odd}}=\frac{2}{\gamma\sqrt{\pi}};
\end{align*}
here $L$, $M$  are defined in  \eqref{def_op_L_rect} and \eqref{def_op_M_rect} and we introduced the final notation
$$
T=  \sum_{q \in \Z} \EE\Big( -\frac{\pi^2 q^2} {\gamma^2}\Big).
$$ 

Recall that the Fourier transform was defined in \eqref{deffourier}; applied to the convolution kernel $L^{\text{odd}}$, it gives
\begin{align}\label{def_fct_h}
	\mathcal{F}(L^{\text{odd}})(\xi) &= \sum_{k \in 2\Z+1} \EE\Big( -2ik\pi \xi-\frac{\pi^2 k^2}{\gamma^2}\Big) \nonumber \\
	&=2 \sum_{k=0}^{+\infty}\EE\Big(-\frac{\pi^2(2k+1)^2}{\gamma^2}\Big)\cos\big({2(2k+1)\pi \xi}\big)=: h(\xi), 
\end{align}
and we recall \eqref{def_fct_l}
\begin{align}\label{def_fct_lb}
	\mathcal{F}(L)(\xi) &= \sum_{k = -\infty}^{+\infty} \EE\Big(-2ik\pi \xi-\frac{\pi^2 k^2} {\gamma^2}\Big)\nonumber \\
	&=1 + 2 \sum_{k=1}^{+\infty}\EE\Big({-\frac{\pi^2k^2}{\gamma^2}}\Big)  \cos\big({2\pi k\xi}\big)= \ell(\xi).
\end{align}

Then,
\begin{multline}\label{defA}
	\pac{\mathcal{F}\begin{pmatrix} \mathcal{L}^{\text{hexa}}-\lambda Id & \mathcal{M}^{\text{hexa}} \\[2pt] -\mathcal{M}^{\text{hexa}} & - (\mathcal{L}^{\text{hexa}}-\lambda Id) \end{pmatrix}\mathcal{F}^{-1}(X)}\pa{\xi} = \\=\begin{pmatrix} a(\xi) & b(\xi) \\[2pt] -b(\xi) & -a(\xi)  \end{pmatrix}X
	 =:A_{\text{hexa}}(\xi)X,
\end{multline}
with 
\begin{equation}\label{def_a_hexa}
	a(\xi) = \mathcal{F} \pa{C'_L L- C_L^{\text{odd}}L^{\text{odd}}} -\lambda
		= \frac{2}{\gamma\sqrt{\pi}}\big(\ell(\xi)\ell(0)-2h(\xi)h(0)\big) -\lambda
\end{equation}
and
\begin{equation}\label{def_b_hexa}
	b(\xi) = \mathcal{F} \pa{C'_M M- C_M^{\text{odd}}M^{\text{odd}}} 
		= \frac{1}{\gamma\sqrt{\pi}}\big({\ell^2(\xi)-2h^2(\xi)}\big)
\end{equation}
(here, we used that $M^{\text{odd}} = L^{\text{odd}} * L^{\text{odd}}$, see \eqref{defModd}).

The value of $\lambda$ follows from \eqref{lambda-hexa}, taking into account the normalization factor $|\kappa|= \pa{2/{(\pi\gamma^2)}}^{1/4} \EE\big({-{\pi \sqrt{3}}/{8}}\big)$:
\begin{equation}\label{def_lambda_hexa}
\lambda = \lambda_0 |\kappa|^2=\frac{1}{\gamma\sqrt{\pi}}\pa{I^2+2IJ-J^2}=\frac{1}{\gamma\sqrt{\pi}}\pa{\ell^2(0)-2h^2(0)}.
\end{equation}
Gathering the above elements yields the statement of the lemma. \end{proof}

\subsection{Analysis of the matrix $A_{\text{hexa}}(\xi)$}

It follows from the previous lemma that the stability of the linearized equation is equivalent to the stability of the matrix $iA_{\text{hexa}}(\xi)$ for any $\xi$.
To understand the spectrum of $A_{\text{hexa}}(\xi)$, it suffices to find the sign of its determinant; this is the object of the following proposition.
\begin{prop}\label{prop54}
For any $\xi \in (0,1)$, $\det A_{\text{hexa}}(\xi) < 0$.
\end{prop}

\begin{proof}
Since $\det A_{\text{hexa}}(\xi) = b(\xi)^2 - a(\xi)^2$, it suffices to prove that $a(\xi),b(\xi)>0$ and that $a(\xi) > b(\xi)$. This will be achieved in the two following lemmas.\end{proof}

\begin{lemme} For any $\xi \in [0,1]$, $a(\xi) > 0$ and $b(\xi)>0$.
\end{lemme}

\begin{proof} \underline{Sign of $a(\xi)$.} By definition of $a(\xi)$, it is $>0$ if and only if
$$ \ell(0)\pac{2\ell(\xi)-\ell(0)} > 2h(0)\pac{2h(\xi)-h(0)}.$$
This inequality holds if $\ell(0) \geq 2h(0) >0$, $2\ell(\xi)-\ell(0) > 0$ and $2\ell(\xi)-\ell(0) > 2h(\xi)-h(0).$ It is clear from the definition of $h$ that $h(0) > 0$. We define then 
$$ g(\xi) := 1 + 2 \sum_{k=1}^{+\infty}\EE\Big( {-\frac{\pi^2(2k)^2}{\gamma^2}}\Big) \cos\big({4\pi k\xi}\big),$$
so that $\ell(\xi) = g(\xi) + h(\xi).$ Then, 
\begin{multline*}
	\ell(0) - 2h(0) = g(0) - h(0)\\
= 1-2\EE\Big(-\frac{\pi^2}{\gamma^2}\Big) + 2\sum_{k=1}^{+\infty}\pa{\EE\Big({-\frac{\pi^2(2k)^2}{\gamma^2}}\Big)-\EE\Big({-\frac{\pi^2(2k+1)^2}{\gamma^2}}\Big)}.
\end{multline*}
Since $$q:= \EE\Big(-\frac{\pi^2}{\gamma^2}\Big) \approx 0.0658,$$
it is clear that the sum in the above right-hand side is positive, and hence that $\ell(0) > 2 h(0)$.

Next, we compute 
\begin{align*}
2\ell(\xi) - \ell(0) &= 1 + 2 \sum_{k=1}^{+\infty} \EE\Big({-\frac{\pi^2k^2}{\gamma^2}}\Big)\underbrace{\pa{2\cos\big({2\pi k \xi}\big)-1}}_{\geq -3} \geq 1 - 6\sum_{k=1}^{+\infty}\EE\Big({-\frac{\pi^2k^2}{\gamma^2}}\Big) \\
	& \geq 1 - 6\sum_{k=1}^{+\infty}\pa{     \EE\Big(  -\frac{\pi^2}{\gamma^2} \Big)    }^k = 1 - 6\frac{q}{1-q} = \frac{1-7q}{1-q} > 0.
\end{align*}

There remains to prove that $$2h(\xi)-h(0) < 2\ell(\xi)-\ell(0),$$
which is equivalent to $2g(\xi)-g(0) > 0$. By a similar argument,
$$2g(\xi) - g(0) = 1 + 2 \sum_{k=1}^{+\infty}  \EE\Big(  -\frac{4\pi^2k^2}{\gamma^2}\Big) \big({2\cos\big( 4\pi k \xi \big)-1}\big) \geq  \frac{1-7q^4}{1-q^4}> 0,$$
which completes the proof that $a(\xi) > 0.$

\medskip \noindent \underline{Sign of $b(\xi)$.} By definition of $b(\xi)$,
\begin{eqnarray}
b(\xi) &=& \frac{1}{\gamma\sqrt{\pi}}\pac{\ell^2(\xi)-2h^2(\xi)} \nonumber  \\
&=&  \frac{1}{\gamma\sqrt{\pi}}\pac{\pa{\ell(\xi)-2h(\xi)}\pa{\ell(\xi)+2h(\xi)}+2h^2(\xi)}.\label{form_bxi}
\end{eqnarray}
Still denoting $q= \EE\big( {-{\pi^2}/{\gamma^2}}\big)$,
\begin{align*}
	\ell(\xi) + 2 h(\xi) &= 1 + 4q\cos\pa{2\pi\xi} + 2\sum_{k=1}^{+\infty}\pac{q^{k^2}\cos\big({2\pi k\xi}\big)+2q^{(2k+1)^2}\cos\big(2\pi\xi(2k+1)\big)} \\
	&\geq 1 -4q -6\frac{q}{1-q} = \frac{1-11q+4q^2}{1-q} >0.
\end{align*}
To justify the last inequality, we observe that the polynomial $1-11q+4q^2$ is positive as soon as $q < q_- := ({11-\sqrt{105}})/{8}$ or equivalently $\gamma^2 < {\pi^2}/\ln(1/q_-) \approx 4.2$, which can be checked numerically. 

Next, we write 
\begin{align*}
	\ell(\xi) - 2 h(\xi) &= 1 - 4q\cos\pa{2\pi\xi} + 2\sum_{k=1}^{+\infty}\pac{q^{k^2}\cos\big({2\pi k\xi}\big)-2q^{(2k+1)^2}\cos\big(2\pi\xi(2k+1)\big)} \\
	&\geq 1 -4q -6\frac{q}{1-q} = \frac{1-11q+4q^2}{1-q} >0
\end{align*}
by the same token. Gathering the above inequalities, we obtain that $b(\xi)>0$.
\end{proof}

\begin{lemme} 
For any $\xi \in (0,1)$, $a(\xi) > b(\xi)$. 
\end{lemme}

\begin{proof} The inequality $a(\xi) > b(\xi)$ holds if and only if for any $\xi \in (0,1)$
$$ \big(\ell(0) - \ell(\xi)\big)^2 <2\big(h(0)-h(\xi)\big)^2.$$
Since $\ell(0) \geq \ell(\xi)$ and $h(0) \geq h(\xi)$, it suffices to prove that
$$ f(\xi) = \sqrt{2}\pa{h(0)-h(\xi)}-\pa{\ell(0) - \ell(\xi)}  > 0.$$
It follows from the formulas giving $h$ and $\ell$ that
	\begin{align*}
		f(\xi) = \sum_{k=1}^{+\infty}\alpha_k\big(1-\cos({2k\pi\xi})\big), \qquad \mbox{with} \quad
		\begin{cases} \alpha_{2n+1} =2(\sqrt{2}-1)q^{(2n+1)^2}  \\[6pt]    \alpha_{2n} = - 2q^{(2n)^2}. \end{cases}
	\end{align*}
Now we write 
	\begin{equation*}
		f(\xi) =  \sum_{k=1}^{+\infty}\alpha_k\big(1-\cos({2k\pi\xi})\big) = 2 \sum_{k=1}^{+\infty}\alpha_k\sin^2\pa{{k\pi\xi}}
		= 2\sin^2({\pi\xi}) F(\xi)  ,
	\end{equation*}
	where 
	$$\dis F(\xi): = \sum_{k=1}^{+\infty}\alpha_k\frac{\sin^2\pa{{k\pi\xi}}}{\sin^2({\pi\xi})}$$
	and we then need to prove that $F(\xi) \geq 0$. We claim that
	
	$$ \forall x \in \R, \quad \abs{\frac{\sin(kx)}{\sin(x)}} \leq k.$$
Indeed, 
$$\left| \frac{\sin(kx)}{\sin(x)} \right| = \left| \frac{e^{ikx}-e^{-ikx}}{e^{ix}-e^{-ix}} \right| = |e^{i(k-1)x} + e^{i(k-3)x} + \dots + e^{-i(k-1)x}| \leq k.$$
 Therefore,
$$ F(\xi) \geq \alpha_1 - \sum_{k=2}^{+\infty} |\alpha_k|k^2 \geq 2(\sqrt{2}-1)q- \sum_{k=2}^{+\infty}2q^{k^2}k^2.$$
We compute 
$$\sum_{k=2}^{+\infty}q^{k^2}k^2 \leq \sum_{j=4}^{+\infty}jq^{j} = q^3\sum_{j=1}^{+\infty}(j+3)q^j = q^3\pa{\frac{q}{(1-q)^2}+\frac{3q}{1-q}} \leq \frac{4q^4}{(1-q)^2}.$$
Furthermore, we know that $0 < q = \EE\big({-{\pi^2}/{\gamma^2}}\big) < 10^{-1}$ so that 
$$ F(\xi) \geq 2q\pa{\sqrt{2}-1 - \frac{4q^3}{(1-q)^2}} > 0,$$
which completes the proof. 
\end{proof}

\begin{prop} \label{propeitA}
We have
\begin{equation}\label{expr_expo_A}
            e^{-itA_{\text{hexa}}(\xi)} = \begin{pmatrix} \cos(t\mu(\xi)) - i\frac{a(\xi)}{\mu(\xi)}\sin(t\mu(\xi)) & -i\frac{b(\xi)}{\mu(\xi)}\sin(t\mu(\xi)) \\[6pt] i\frac{b(\xi)}{\mu(\xi)}\sin(t\mu(\xi)) & \cos(t\mu(\xi)) + i\frac{a(\xi)}{\mu(\xi)}\sin(t\mu(\xi)) \end{pmatrix},
\end{equation}
where 
$$\mu(\xi)=\big(a^2(\xi)-b^2(\xi)\big)^{1/2}.$$
As a result,
\begin{equation}\label{expr_sg_fg}
    e^{-it A_{\text{hexa}}(\xi)}\begin{pmatrix} {f}_0 \\ {g}_0 \end{pmatrix} = \begin{pmatrix} (k_1^{+}{f}_0+k_{2}^{+}{g}_0)e^{it\mu(\xi)} +  (k_{1}^{-}{f}_0+k_2^{-}{g}_0)e^{-it\mu(\xi)}   \\[6pt]  (k_2^{-}{f}_0+k_1^{-}{g}_0)e^{it\mu(\xi)} +  (k_{2}^{+}{f}_0+k_{1}^{+}{g}_0)e^{-it\mu(\xi)}\end{pmatrix},
\end{equation}
with 
\begin{equation}\label{defK}
 k_1^{\pm}(\xi) = \frac{1}{2\mu(\xi)}\big(\mu(\xi)\mp a(\xi)\big), \qquad k_2^{\pm}(\xi) = \mp\frac{b(\xi)}{2\mu(\xi)}.
 \end{equation}
\end{prop}
\begin{proof} 
In this proof we write $A=A_{\text{hexa}}$.  From Proposition \ref{prop54} we deduce that for all ${0 < \xi < 1}$, the matrix~$A(\xi)$ has two real eigenvalues with opposite signs, $\mu$ and $-\mu$, and thus is diagonalizable. For $\xi =0$ or $\xi = 1$, $\mu$ vanishes, so that $0$ has  algebraic multiplicity $2$. The corresponding matrix is not diagonalizable, and for this reason, we will  trigonalize~$A(\xi)$ for all $\xi \in [0,1]$. 

The eigenspace corresponding to the eigenvalue $\pm \mu(\xi)$ is $$\text{Span}\begin{pmatrix} \delta_{\pm}(\xi)  \\[2pt] 1 \end{pmatrix} \qquad \mbox{with} \qquad 
 \delta_{\pm}(\xi) = -\frac{1}{b(\xi)}(a(\xi) \pm \mu(\xi)) .$$
Note that we have $\delta_{+}(0) = \delta_{-}(0) = -1$. The same relation holds for $\xi = 1$.
We consider the matrix 
$$ P(\xi) = \begin{pmatrix} \delta_-(\xi) & 0 \\[2pt] 1 & 1 \end{pmatrix} \in GL_2(\R),$$
whose inverse is 
$$P^{-1}(\xi) = \frac{1}{\delta_-(\xi)} \begin{pmatrix} 1 & 0  \\[2pt] -1 & \delta_-(\xi) \end{pmatrix}.$$
Then we have the following relation: 
\begin{equation}\label{tridiag_A}
    A(\xi) = P(\xi)B(\xi) P^{-1}(\xi) \qquad \mbox{where} \qquad 
 B(\xi) = \begin{pmatrix} -\mu(\xi) & -(\mu(\xi)+a(\xi)) \\[2pt] 0 & \mu(\xi) \end{pmatrix}.
 \end{equation}
Exponentiating gives
\begin{equation*}
    \forall \xi\in[0,1], \quad e^{-itA(\xi)} = P(\xi)e^{-itB(\xi)} P^{-1}(\xi).
\end{equation*}
Since
\begin{equation*}
    e^{-itB(\xi)} = \cos(t\mu(\xi)) I_2 - \frac{i}{\mu(\xi)}\sin(t\mu(\xi)) B(\xi),
\end{equation*}
we find that
\begin{equation*}
    e^{-itA(\xi)} = \begin{pmatrix} \cos(t\mu(\xi)) - i\frac{a(\xi)}{\mu(\xi)}\sin(t\mu(\xi)) & -i\frac{b(\xi)}{\mu(\xi)}\sin(t\mu(\xi)) \\[6pt] i\frac{b(\xi)}{\mu(\xi)}\sin(t\mu(\xi)) & \cos(t\mu(\xi)) + i\frac{a(\xi)}{\mu(\xi)}\sin(t\mu(\xi)) \end{pmatrix}
\end{equation*}
as claimed.
\end{proof}

In the sequel, we will use the expression \eqref{expr_sg_fg} both for $f$ and $g$ functions of the space variable $w$ and frequency variable $\xi$.
The four functions $k_j^\pm, j=1,2$ are not defined for $\xi=0$, since $\mu$ vanishes at $0$ but the two functions $a$ and $b$ do not. The following proposition gives a finer understanding of the vanishing of $\mu$.

\begin{prop}\label{prop_equiv_mu}
The following holds true: 
\begin{enumerate}[$(i)$]
   \item The functions $a,b$, $\mu$ and $k_j^\pm$ are even.
   \item There exists $c_0 \in \R$ such that
   $$a(\xi)= \lambda +c_0 \xi^2+  \mathcal{O}(\xi^4), \qquad b(\xi)= \lambda +c_0 \xi^2+ \mathcal{O}(\xi^4).$$
   \item Setting $$ C = \sqrt{{\lambda}\big(a^{(4)}(0)-b^{(4)}(0)\big)/12},$$
   there holds
    \begin{equation}\label{equiv_mu}
        \mu(\xi) \simx{\xi}{0} C\xi^2,
    \end{equation}
and 
    \begin{equation}\label{equiv_fcts_k}
        k_j^+(\xi) \simx{\xi}{0} -\frac{\lambda}{2C\xi^2}, \qquad k_j^-(\xi) \simx{\xi}{0} \frac{\lambda}{2C\xi^2}.
    \end{equation}
    \end{enumerate}
\end{prop}
 
\begin{proof}
Recall the expressions \eqref{def_lambda_hexa}, \eqref{def_a_hexa} and \eqref{def_b_hexa} giving $\lambda$, $a$ and $b$ respectively. In particular, $a(0) = b(0) = \lambda$, and $\mu(0) = 0$. Recall also the expressions \eqref{def_fct_h} and \eqref{def_fct_lb} of $h$ and $\ell$. To compute the derivatives of $a$ and~$b$ at $0$, we need the derivatives of $\ell$ and $h$. Since these are even functions, the odd derivatives vanish at $0$ and furthermore
    \begin{align*}
 & \ell(\xi) = 1 + 2\sum_{k=1}^{+\infty} \EE\Big(-\frac{\pi^2k^2}{\gamma^2}\Big) \cos\big({2\pi k \xi} \big) \\ 
 &  \ell''(\xi) = -2\sum_{k=1}^{+\infty} (2\pi k)^2\EE\Big(-\frac{\pi^2k^2}{\gamma^2}\Big) \cos\big({2\pi k \xi} \big) \\
&        \ell^{(4)}(\xi) = 2\sum_{k=1}^{+\infty} (2\pi k)^4\EE\Big(-\frac{\pi^2k^2}{\gamma^2}\Big) \cos\big({2\pi k \xi} \big).
    \end{align*}
This implies that $0< h(0) < \ell(0)$, $\ell''(0) < h''(0) < 0$ and $0 < h^{(4)}(0) < \ell^{(4)}(0)$. We now compute the derivatives of $a$ and $b$ at $0$; once again, since these are even functions, the odd derivatives vanish and
    \begin{align*}
 &       a''(\xi) = \frac{2}{\gamma\sqrt{\pi}} \pa{\ell''(\xi)\ell(0)-2h''(\xi)h(0)}, \\
 &       a^{(4)}(\xi) = \frac{2}{\gamma\sqrt{\pi}} \pa{\ell^{(4)}(\xi)\ell(0)-2h^{(4)}(\xi)h(0)}, \\
 &       b''(\xi) = \frac{2}{\gamma\sqrt{\pi}} \pa{\ell''(\xi)\ell(\xi)+\ell'(\xi)^2-2h''(\xi)h(\xi)-2h'(\xi)^2}, 
    \end{align*}
      \begin{multline*}
     b^{(4)}(\xi) =  \frac{2}{\gamma\sqrt{\pi}} \Big(\ell^{(4)}(\xi)\ell(\xi)+4\ell^{(3)}(\xi)\ell'(\xi)+3\ell''(\xi)^2  \\
         -2h^{(4)}(\xi)h(\xi)-8h^{(3)}(\xi)h'(\xi)-6h''(\xi)^2\Big), 
    \end{multline*}
    giving 
 \begin{align*}
& a''(0) = b''(0) = \frac{2}{\gamma\sqrt{\pi}} \big(\ell''(0)\ell(0)-2h''(0)h(0)\big) \\
&b^{(4)}(0) - a^{(4)}(0) = \frac{6}{\gamma\sqrt{\pi}} \big(\ell''(0)^2 - 2h''(0)^2\big) < 0.
\end{align*}
    Inserting the expansions of $a$ and $b$ up to the fourth order in $\xi$ in the expression of $\mu(\xi)$ gives the result. 
\end{proof}

We recall that $\gamma \approx 1.90$, and we plot the function $\mu$ and its first three derivatives in Figure \ref{fig:graphes_mu}.    
\begin{figure}
    \centering
    \includegraphics[width=\linewidth]{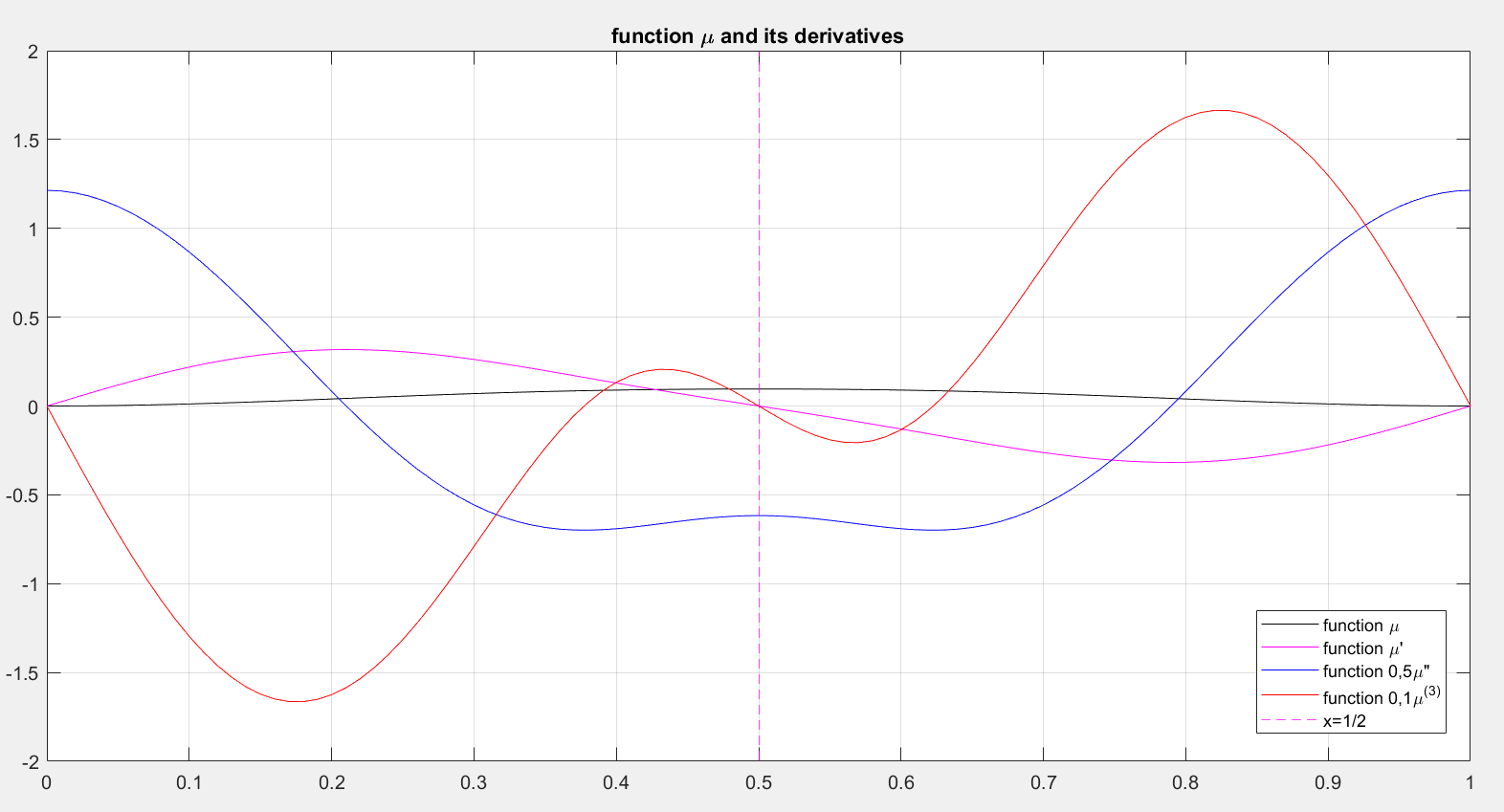}
    \caption{Plot of the function $\mu$ and its derivatives, with multiplicative constants.}
    \label{fig:graphes_mu}
\end{figure}

The equation 
$$i \partial_t {U}(\xi) = A_{\text{hexa}}(\xi)  {U}(\xi)$$
solves in 
$${U}(t,\xi) = e^{-itA_{\text{hexa}}(\xi) }{U_0}(\xi).$$
Denote by $\dis v(t,z)=\sum_{n\in \Z}c_n(t) \psi_n(z)$, $\dis \widetilde{v}(t,z)=\sum_{n\in \Z}\ov{c_n(t)} \psi_n(z)$ and 
$ V= \begin{pmatrix} v\\ \widetilde{v} \end{pmatrix} $, then  
\begin{multline}\label{expr_Utz}
 V(t,z) = \sum_{k=-\infty}^{+\infty} \pa{ \int_0^1 e^{2ik\pi\xi}e^{-it A_{\text{hexa}}(\xi) } {V_0}(\xi)  d\xi }\psi_k(z)   \\
\begin{aligned}
    &= \sum_{k=-\infty}^{+\infty} \pa{ \int_0^1 e^{2ik\pi\xi}e^{-it A_{\text{hexa}}(\xi) } \pac{\sum_{n=-\infty}^{+\infty} e^{-2in\pi\xi}\int_{S_{\gamma}} V_0(w)\ov{\psi_n(w)}dL(w)}  d\xi }\psi_k(z)  \\
    &= \int_{S_{\gamma}} \mathcal{K}_t(w,z) V_0(w) dL(w), 
\end{aligned}
\end{multline}
where
\begin{equation}\label{noyau_LLL}
    \mathcal{K}_t(w,z) = \sum_{k,n\in\Z} \int_0^1 e^{2i\pi(k-n)\xi}\psi_k(z)\ov{\psi_n(w)}e^{-it A_{\text{hexa}}(\xi) }d\xi.
\end{equation}

\subsection{Stability and decay for the linearized equation}

As in Lemma \ref{lemmaformula}, we write
$$
v(t,z) = \sum_{n \in \Z} c_n(t) \psi_n(z) \qquad \mbox{and} \qquad 
f(t,\xi) = \sum_{n \in \Z} c_n(t) e^{-2\pi i n\xi},
$$
and we set furthermore
$$
g(t,\xi) = \overline{f}(t,-\xi) =  \sum_{n \in \Z} \ov{c_n(t)} e^{-2\pi i n\xi}.
$$
 
We  state a first stability result

\begin{theorem}\label{prop_decay_l2_iR} 
	We write 
$$v(t,z) = \sum_{n\in\Z} c_n(t) \psi_n(z), $$ 
and suppose that $ \dis \sum_{n \in \Z} |c_n(0)|^2 < +\infty,$ meaning $v_0 \in L^2(S_{\gamma}).$ We assume that   $ \forall n \in \Z, \, c_n(0) \in i\R, $ 
then 
$$ \|v(t)\|_{L^2(S_\gamma)} \leq C \|v_0\|_{L^2(S_\gamma)}.$$
\end{theorem}

\begin{proof}
	Recall that $ U = \begin{pmatrix} f\\ g\end{pmatrix}$. We write 
    \begin{equation*}
        \|{U}(t,\xi) \|^2_{L_\xi^2} =  2\sum_{k\in\Z} |c_k(t)|^2 =  2 \|v(t,z)\|^2_{L^2_z}
    \end{equation*}
    and recall:
    $${U}(t,\xi) = e^{-itA_{\text{hexa}}(\xi) }{U_0}(\xi).$$
    We denote 
    $$ {U_0}(\xi) = \begin{pmatrix} {f}_0(\xi) \\ {g}_0(\xi)\end{pmatrix}, \qquad {f}_0(\xi) = \sum_{k\in\Z} c_k(0)e^{- 2i\pi k \xi}, \quad {g}_0(\xi) = \sum_{k\in\Z} \overline{c_k(0)}e^{- 2i\pi k \xi}.$$
    We suppose that $\forall n \in \Z, \, c_n(0) \in i\R$, that is to say ${f}_0(\xi) = -{g}_0(\xi).$ Then, from equality \eqref{expr_sg_fg}, we have 
    \begin{multline*}
       {U}(t,\xi)= e^{-itA_{\text{hexa}}(\xi) }{U_0}(\xi) = e^{-it A_{\text{hexa}}(\xi) }\begin{pmatrix} {f}_0(\xi) \\[2pt] -{f}_0(\xi) \end{pmatrix} = \\
       =\begin{pmatrix} (k_1^{+}-k_{2}^{+})e^{it\mu(\xi)} +  (k_{1}^{-}-k_2^{-})e^{-it\mu(\xi)}  \\[4pt]  (k_2^{-}-k_1^{-})e^{it\mu(\xi)} +  (k_{2}^{+}-k_{1}^{+})e^{-it\mu(\xi)}\end{pmatrix}{f}_0(\xi).
    \end{multline*}
    We have
    \begin{align*}
        k_1^{+}(\xi)-k_{2}^{+}(\xi) & = \frac12 + \frac{b(\xi)-a(\xi)}{2\mu(\xi)} \\[4pt]
        & = \frac12 - \frac{\mu(\xi)}{2(b(\xi)+a(\xi))} \in \mathcal{C}_{\xi}^\infty(\R).
    \end{align*}
    Similarly, $k_1^{-}-k_{2}^{-}\in \mathcal{C}_{\xi}^\infty(\R).$
    We deduce that:
    \begin{equation*}
        \| v(t,z)\|_{L_z^2} \leq  \|{U}(t,\xi) \|_{L_\xi^2} \leq C\|{f}_0(\xi)\|_{L_\xi^2} = C \|v_0(z)\|_{L_z^2},
    \end{equation*}
    where $C > 0$ is some absolute constant.
\end{proof}

 The statement of Theorem \ref{prop_decay_l2_iR} is elementary, but it has the drawback that the condition  $ \forall n \in \Z$, $c_n(0) \in i\R, $ is not preserved by the flow of \eqref{lin_LLL_hexa}. A more natural stability condition is given in the result below.

\begin{theorem} 
\label{thmL2stable}
For all $t \geq 0$, the solution $f(t)$ of \eqref{lin_LLL_hexa} satisfies
        \begin{equation*} 
        \Big\|\frac{{f}+{g}}{\mu} \Big\|_{L^2([0, 1])} + \|{f}\|_{L^2([0, 1])} \lesssim  \Big\|\frac{{f}_0+{g}_0}{\mu} \Big\|_{L^2([0, 1])}+  \|{f}_0\|_{L^2([0, 1])} ,
              \end{equation*}
        and more generally, for all $j \in \R$
           \begin{equation}\label{en2}
           \Big\|\frac{{f}+{g}}{\mu^{j+1}} \Big\|_{L^2([0, 1])}+\Big\|\frac{{f}-{g}}{\mu^j} \Big\|_{L^2([0, 1])}  \lesssim \Big\|\frac{{f}_0+{g}_0}{\mu^{j+1}} \Big\|_{L^2([0, 1])}+ \Big\|\frac{{f}_0-{g}_0}{\mu^j} \Big\|_{L^2([0, 1])} .
           \end{equation}     
   \end{theorem}

    Actually, the condition  $\dis \big\|({f}_0+{g}_0)/{\mu} \big\|_{L^2([0, 1])} <\infty$ is equivalent to 
$$\Re\Big(\sum_{n\in\Z}c_n(0)\Big) = \Re\Big(\sum_{n\in\Z}nc_n(0)\Big)=0,$$
 which is propagated by the flow of \eqref{lin_LLL_hexa}, using Lemma \ref{lemme_conserv_hyp1}. More generally, since $\mu$ vanishes at second order at both endpoints and the data are $1$-periodic, for any integer $j \geq 0$ the right-hand side of \eqref{en2} is finite if and only if
$$ \Re\Big(\sum_{n\in\Z}n^p c_n(0)\Big)=  \Im\Big(\sum_{n\in\Z}n^q c_n(0)\Big)=0, $$
for all  $0\leq p\leq 2j+1$ and all $0\leq q\leq 2j-1$.

    \begin{proof}
We deduce from Proposition \ref{propeitA} that
     $${f}\pm{g} =\big ((k_1^{+}\pm k_{2}^{-}) {f}_0+ (k_2^{+}\pm k_{1}^{-}) {g}_0 \big)e^{it\mu(\xi)} + \big ((k_1^{-} \pm k_{2}^{+}) {f}_0+ (k_2^{-} \pm k_{1}^{+}) {g}_0 \big)e^{-it\mu(\xi)}. $$
Setting $\dis F=(a-b)/\mu$ and $\dis G=(a+b)/\mu$, we have
\begin{align*}
& k_1^{+}+k_{2}^{-}=\frac12(1-F), \qquad k_1^{-}+k_{2}^{+}=\frac12(1+F) \\
 & k_1^{+}-k_{2}^{-}=\frac12(1-G), \qquad k_1^{-}-k_{2}^{+}=\frac12(1+G),
 \end{align*}
          so that 
            $${f}+{g} =  ({f}_0+{g}_0)\cos(t \mu)-i ({f}_0-{g}_0) F \sin(t \mu), $$
                     $${f}-{g} =  ({f}_0-{g}_0)\cos(t \mu)- i({f}_0+{g}_0) G \sin(t \mu). $$   
Close to $\xi=0$, Proposition \ref{prop_equiv_mu} gives that $F(\xi) =\mathcal{O}(\xi^2)$ and $G(\xi) =\mathcal{O}(\xi^{-2})$ with a symmetrical behavior at $\xi=1$. As a consequence, for all $j \in \mathbb{R}$,
\begin{align*}
& \Big\|\frac{{f}-{g}}{\mu^j} \Big\|_{L^2([0, 1])}   \lesssim \Big\|\frac{{f}_0-{g}_0}{\mu^j} \Big\|_{L^2([0, 1])}+  \Big\|\frac{{f}_0+{g}_0}{\mu^{j+1}} \Big\|_{L^2([0, 1])} \\[4pt]
& \Big\|\frac{{f}+{g}}{\mu^{j+1}} \Big\|_{L^2([0, 1])}   \lesssim \Big\|\frac{{f}_0+{g}_0}{\mu^{j+1}} \Big\|_{L^2([0, 1])}+  \Big\|\frac{{f}_0-{g}_0}{\mu^j} \Big\|_{L^2([0, 1])},
\end{align*}
hence the result.
 \end{proof}

\begin{theorem}\label{decay_linf}
    Consider a function $v_0 \in \mathcal{F}_\gamma \cap \ET$. We write
    $$v(t,z) = \sum_{n\in\Z} c_n(t) \psi_n(z). $$ 
   Then,
    \begin{enumerate}[$(i)$]
        \item If $\forall n\in \Z, c_n \in i\R$,
        $$ \|v(t)\|_{L^\infty(S_\gamma)} \lesssim {t^{-1/3}} \|v_0\|_{L^1(S_\gamma)}.$$
 \item We have
	\begin{equation*}
		\|v(t)\|_{L^\infty(S_\gamma)} 
	\lesssim {t^{-1/3}} \pac{\absabs{{f}_0}_{H^1([0,1])}+ \absabs{\frac{{f}_0+{g}_0}{\mu} }_{H^1([0,1])}}.
	\end{equation*} 
	 \end{enumerate}
\end{theorem}

\begin{proof}
$(i)$ 
Going back to \eqref{expr_Utz} and \eqref{noyau_LLL}, we study $\EE\big(  {-itA_{\text{hexa}}(\xi) }\big)V_0(w)$. Recall that $ V_0= \begin{pmatrix} v_0\\ \widetilde{v}_0 \end{pmatrix} $. Using the equation \eqref{expr_sg_fg}, we have 
\begin{equation*}
    \mathcal{K}_t(w,z)V_0(w) = \sum_{k,n\in\Z}\psi_k(z) \ov{\psi_n(w)} \begin{pmatrix} I_1^{k,n}(t,w) \\[2pt] I_2^{k,n}(t,w) \end{pmatrix},
\end{equation*}
where for $t \geq 0$ and $w \in \C$
 \begin{multline*}
    I_1^{k,n}(t,w) =  \int_0^1 e^{ {2i\pi(k-n)} \xi}\Big[\big(k_1^+(\xi)v_0(w)+ k_2^+(\xi)\widetilde{v}_0(w)\big)e^{it\mu(\xi)}+\\ +\big(k_1^-(\xi)v_0(w)+   k_2^-(\xi)\widetilde{v}_0(w)\big)e^{-it\mu(\xi)}\Big]d\xi,
\end{multline*}
\begin{multline*}
    I_2^{k,n}(t,w) = \int_0^1 e^{ {2i\pi(k-n)} \xi}\Big[\big(k_2^-(\xi)v_0(w)+ k_1^-(\xi)\widetilde{v}_0(w)\big)e^{it\mu(\xi)}+\\
    +   \big(  k_2^+(\xi)v_0(w)+ k_1^+(\xi)\widetilde{v}_0(w)\big)   e^{-it\mu(\xi)}\Big]d\xi.
\end{multline*}
In the sequel, we will only study the integral $I_1^{k,n}$. We have:
\begin{align*}
    I_1^{k,n}(t,w) &= \frac{1}{2}\int_0^1 e^{2i\pi(k-n)\xi}\Big(e^{it\mu(\xi)}+e^{-it\mu(\xi)}\Big)v_0(w) d\xi  \\ 
    &\quad \qquad - i  \int_0^1 e^{{2i\pi(k-n)}\xi}\Big(b(\xi)\widetilde{v}_0(w)+a(\xi)v_0(w)\Big) \frac{\sin(t\mu(\xi))}{\mu(\xi)} d\xi.
\end{align*}

We define $\Gamma_+(\xi) = \mu(\xi) + {2\pi}(k-n)\xi/t$ and $\Gamma_-(\xi) = -\mu(\xi) + {2\pi}(k-n)\xi/t$. We have $\Gamma_+''(\xi) = \mu''(\xi)$, $\Gamma_-''(\xi) = -\mu''(\xi)$, and $\mu''$ only has two zeros in $[0,1]$, that we denote $\xi_1 \in (0,1/2)$ and $\xi_2\in(1/2,1).$ Close to $\xi_1$ and $\xi_2$, the function $\mu^{(3)}$ is not close to zero, as can be seen by a graphical study (see Figure \ref{fig:graphes_mu}). Then  by the van der Corput lemma (see for example \cite{Stein}), we have 
$$ \abs{ \int_{[\xi_1-\delta, \xi_1+\delta]} \EE\big( {it\Gamma_\pm(\xi)}\big)d\xi } \leq {C_1}{t^{-1/3}} $$ and 
$$ \abs{ \int_{[\xi_2-\delta, \xi_2+\delta]} \EE\big( {it\Gamma_\pm(\xi)}\big)d\xi } \leq {C_1}{t^{-1/3}}, $$
with $C_1 > 0$ a constant independent of $k$ and $n$. Outside the two intervals $[\xi_1-\delta, \xi_1+\delta]$ and $[\xi_2-\delta, \xi_2+\delta]$, the function $\mu''$ is not zero, so that 
$$ \abs{ \int_{[0, \xi_1-\delta]} \EE\big( {it\Gamma_\pm(\xi)}\big)d\xi } \leq {C_2}{t^{-1/2}}, $$
with $C_2 > 0$ a constant independent of $k$ and $n$. We have the same bound on $[\xi_1+\delta, \xi_2-\delta]$ and $[\xi_2+\delta, 1]$. Finally, adding up those five inequalities:
\begin{equation}\label{estimation2}
	\abs{ \int_{0}^1 \EE\big( {it\Gamma_\pm(\xi)}\big)d\xi } \leq {C_3}{t^{-1/3}}.
\end{equation}

We now make the assumption, in the spirit of Theorem \ref{prop_decay_l2_iR}, that $\widetilde{v}_0(w) = -v_0(w)$, meaning that for all $n \in \Z, \, c_n(0) \in i\R$. With this, we obtain 
\begin{equation*}
    - i   \int_0^1 e^{2i\pi(k-n)\xi}\big(b(\xi)\widetilde{v}_0(w)+a(\xi)v_0(w)\big) \frac{\sin(t\mu(\xi))}{\mu(\xi)} d\xi = -v_0(w) \pa{J_{+}^{k,n}(t)-J_{-}^{k,n}(t)},
\end{equation*}
where 
$$ J_\pm^{k,n}(t) :=   \int_0^1 F(\xi)\EE\big( {it\Gamma_\pm(\xi)}\big)d\xi, \qquad F(\xi) := \frac{a(\xi) - b(\xi)}{2\mu(\xi)}.$$
From the computations of Proposition \ref{prop_equiv_mu}, we have 
$$ F(\xi) \simx{\xi}{0} \frac{C}{4\lambda} \xi^2,$$
the constant $C >0$ being defined in Proposition \ref{prop_equiv_mu}. Then $F \in \mathcal{C}^1\big([0,1]\big)$ with $F(0) = F'(0) = 0$. We denote 
$$  G_\pm(\xi) := \int_0^\xi \EE\big( {it\Gamma_\pm(\eta)}\big) d\eta,$$
so that 
$$ J_\pm^{k,n}(t) =  \int_0^1 F(\xi)\frac{d}{d\xi}G_\pm(\xi)d\xi. $$
We recall that 
$$ \forall \xi \in [0,1], \quad |G_\pm(\xi)| \leq {C_3}{t^{-1/3}}.$$
Then, an integration by parts and $F(1) = 0$ give
$$ |J_\pm^{k,n}(t)| =\abs{ F(1)G_\pm(1) -  \int_0^1 F'(\xi)G_\pm(\xi)d\xi} \leq {C_3}{  t^{-1/3}}\int_0^1 |F'(\xi)|d\xi,$$
that is to say
\begin{equation}\label{estimation3}
    |J_\pm^{k,n}(t)| \leq {c}{t^{-1/3}},
\end{equation}
the constant $c>0$ being absolute. Finally, combining \eqref{estimation2} and \eqref{estimation3}: 
\begin{equation*} 
    |I_1^{k,n}(t,w)| \leq {c}{t^{-1/3}}|v_0(w)|,
\end{equation*}
so that we have asymptotic stability:
\begin{equation}\label{estimation_u_l_infty}
    |v(t,z)| \lesssim  {t^{-1/3}}\|\Psi\|_{L^\infty(S_{\gamma}\times S_{\gamma})}\|v_0\|_{L^1(S_{\gamma})},
\end{equation}
where,
\begin{equation*}
    \Psi(z,w) = \sum_{k,n\in\Z}|\psi_k(z)\ov{\psi_n(w)}|.
\end{equation*}
We compute 
$$ \Psi(z,w) =|\psi_0(z)\ov{\psi_0(w)}|  \sum_{k,n\in\Z} \EE\Big( {-{\pi^2(n^2+k^2)}/ {\gamma^2}-{2\pi}(k\mathfrak{Im}(z)+n\mathfrak{Im}(w))/ {\gamma}}\Big), $$
which is a function of $L^\infty(S_{\gamma}\times S_{\gamma})$.

\noindent $(ii)$ As before, we write 
$$ {f}(t,\xi) = \sum_{n\in \Z} c_n(t) e^{-2in\pi\xi}, \qquad {g}(t,\xi) = \sum_{n\in \Z} \ov{c_n(t)} e^{-2in\pi\xi},$$
and ${f}_0(\xi) = {f}(0,\xi)$, ${g}_0(\xi) = {g}(0,\xi).$ Then, 
$$ {f}(t,\xi) = \cos\big(t\mu(\xi)\big) {f}_0(\xi) -i\big(a(\xi){f}_0(\xi)+b(\xi){g}_0(\xi)\big)\frac{\sin(t\mu(\xi))}{\mu(\xi)},$$
so that for all $k\in\Z$,
\begin{equation}\label{expr_ck}
c_k(t) = \int_0^1 e^{2ik\pi\xi}\cos(t\mu){f}_0(\xi) d\xi - i\int_0^1e^{2ik\pi\xi}\big(a{f}_0(\xi)+b{g}_0(\xi)\big)\frac{\sin(t\mu)}{\mu}d\xi. 
\end{equation}
The second integral is well-defined because, for all $t \geq 0, \abs{ \mu^{-1}{\sin(t\mu)}} \leq t.$ Denoting $\dis \Gamma_\pm(\xi) = \pm \mu(\xi) + {2\pi}k\xi/t$ for the phase functions,   
\begin{equation}\label{Integrale1}
\int_0^1 e^{2ik\pi\xi}\cos\big(t\mu(\xi)\big){f}_0(\xi)d\xi = \frac12\int_0^1 \Big(  \EE\big(it\Gamma_+(\xi)\big)+   \EE\big(it\Gamma_-(\xi)\big)   \Big){f}_0(\xi)d\xi.
\end{equation}
Integrating by parts and denoting $\dis G_\pm(\xi) = \int_0^\xi   \EE\big(it\Gamma_\pm(\eta)\big) d\eta$,
\begin{equation}\label{IPP1}
\int_0^1  \EE\big(it\Gamma_\pm(\xi)\big){f}_0(\xi)d\xi = \int_0^1G_\pm'(\xi){f}_0(\xi)d \xi = {f}_0(1)G_\pm(1) - \int_0^1G_\pm(\xi){f}_0'(\xi)d\xi.
\end{equation}
Furthermore, $\Gamma_\pm''(\xi) = \pm\mu''(\xi)$, so that $\abs{\Gamma_\pm''(\xi) } + \abs{\Gamma_\pm'''(\xi) } \geq c > 0, \forall \xi \in [0,1]$. Then, the van der Corput lemma gives the bound 
\begin{equation}\label{estimation01}
	\forall \xi \in [0,1], \qquad |G_\pm(\xi)| \lesssim {t^{-1/3}}.
\end{equation}
Combining \eqref{Integrale1}, \eqref{IPP1} and the estimate \eqref{estimation01}, we get
\begin{equation}\label{Int1}
	\abs{\int_0^1 e^{2ik\pi\xi}\cos(t\mu){f}_0d\xi} \lesssim {t^{-1/3}} \pac{|{f}_0(0)| + \int_0^1|{f}_0'(\xi)|d\xi},
\end{equation}
the implicit constant being independent of $k \in \Z$, and using that ${f}_0(1)={f}_0(0)$. We now turn to the second integral in \eqref{expr_ck}. We write $c_k(0) = \alpha_k + i\beta_k, \alpha_k,\beta_k \in \R$, so that 
\begin{equation*}
\left\{
\begin{aligned}
 {f}_0(\xi) & = {f}_1(\xi) + i{f}_2(\xi), \\[2pt]
 {g}_0(\xi) & = {f}_1(\xi) - i{f}_2(\xi),
\end{aligned}
\right.
\quad \text{with} \quad
\left\{
\begin{aligned}
 {f}_1(\xi) & := \sum_{n\in \Z} \alpha_n e^{-2in\pi\xi}, \\[2pt]
 {f}_2(\xi) & := \sum_{n\in \Z} \beta_n e^{-2in\pi\xi},
\end{aligned}
\right.
\end{equation*}
(observe that $f_1$ and $f_2$ are not real-valued).  In the sequel, we assume that  
$$\absabs{\frac{{f}_0+{g}_0}{\mu} }_{H^1([0,1])}<+\infty$$ which is equivalent to  the condition ${f}_1(0) = {f}'_1(0) = 0$.  Then, 
$$a{f}_0+b{g}_0 = (a+b){f}_1 +i(a-b){f}_2.$$
We then split the second integral in \eqref{expr_ck} into four integrals:
$$
i\int_0^1e^{2ik\pi\xi}(a{f}_0+b{g}_0)\frac{\sin(t\mu)}{\mu}d\xi = \frac12(I_1^+ - I_1^- +I_2^+ - I_2^-),
$$
where 
$$ I_1^\pm = \int_0^1  \EE\big(it\Gamma_\pm \big)    \frac{a+b}{\mu}{f}_1 d\xi, \qquad \text{and} \qquad I_2^\pm = i\int_0^1  \EE\big(it\Gamma_\pm\big)   \frac{a-b}{\mu}{f}_2 d\xi.$$
From Proposition \ref{prop_equiv_mu}, the function $({a-b})/{\mu}$ is bounded with a derivative being also bounded. Then, we get the same decay we obtained in \eqref{Int1} with an analogous method:
\begin{equation*} 
	\abs{I_2^\pm} \leq {c}{t^{-1/3}} \pac{|{f}_2(0)| + \int_0^1|{f}_2'(\xi)|d\xi} \leq {c}{t^{-1/3}} \pac{|{f}_0(0)| + \int_0^1|{f}_0'(\xi)|d\xi}.
\end{equation*}
Similarly, writing 
$$F(\xi) := \frac{a(\xi)+b(\xi)}{\mu(\xi)}{f}_1(\xi),$$
 we show that
\begin{equation}\label{Int3}
	\abs{I_1^\pm} \leq {c}{t^{-1/3}} \pac{|F(0)| + \int_0^1|F'(\xi)|d\xi},
\end{equation}
when those quantities are defined. From $(a+b)(0) = 2\lambda > 0$, and $\mu(\xi) \sim C\xi^2$, the condition ${f}_1(0) = {f}'_1(0) = 0$ (or equivalently $\dis \sum_{n\in\Z}\alpha_n =  \sum_{n\in\Z}n\alpha_n=0$) is necessary. Furthermore, this condition implies that  the function $F$ is well-defined on $[0,1]$, and of class $\mathcal{C}^1$, so that the right hand side of \eqref{Int3} is defined and finite. Then, 
\begin{multline*}
	\int_0^1|F'(\xi)|d\xi \leq \|(a+b)'\|_{L^\infty}\int_0^1\abs{{f}_1/{\mu}}d\xi + \|a+b\|_{L^\infty}\int_0^1\abs{\pa{{{f}_1}/{\mu}}'}d\xi\\
		 \leq c \absabs{\frac{{f}_0+{g}_0}{\mu} }_{W^{1,1}([0,1])}.
\end{multline*}
Overall, 
\begin{equation}\label{Int4}
	\abs{I_1^\pm} \leq {c}{t^{-1/3}} \pac{\absabs{\frac{{f}_0+{g}_0}{\mu} }_{L^\infty([0,1])} + \absabs{\frac{{f}_0+{g}_0}{\mu} }_{W^{1,1}([0,1])}},
\end{equation}
and we conclude that 
\begin{multline*}
	|c_k(t)|  \leq \\
\begin{aligned}
&\leq  {c}{t^{-1/3}} \pac{\absabs{{f}_0 }_{L^\infty([0,1])} + \absabs{{f}_0'}_{L^1([0,1])}+\absabs{\frac{{f}_0+{g}_0}{\mu} }_{L^\infty([0,1])} + \absabs{\frac{{f}_0+{g}_0}{\mu} }_{W^{1,1}([0,1])}} \\
&\leq {c}{t^{-1/3}} \pac{\absabs{{f}_0}_{H^1([0,1])}+ \absabs{\frac{{f}_0+{g}_0}{\mu} }_{H^1([0,1])}}
\end{aligned}
\end{multline*}
where the last line is obtained by Sobolev embeddings and the Cauchy-Schwarz inequality. Now, for all $z \in S_{\gamma},$ we have 
\begin{equation*}
	|v(t,z)| \leq \sum_{k\in\Z}|c_k(t)\psi_k(z)| \leq {C({f}_0,{g}_0)}{t^{-1/3}}\sum_{k\in\Z}|\psi_k(z)|,
\end{equation*}
which implies the result because the function $\dis z \mapsto \sum_{k\in\Z}|\psi_k(z)|$ is bounded.
\end{proof}

\subsection{Growth for the linearized equation}

Recall the equation \eqref{lin_LLL_hexa}
\begin{equation*} 
 i\partial_t v + \lambda v = \Pi\big[2|\Psi_0|^2v+\Psi_0^2\ov{v}\big].
\end{equation*}
 In this paragraph, we will prove the following result which show the possible growth of the solution in the case where $ \dis \big\|    ( {{f}_0+{g}_0})/{\mu} \big\|_{L^2([0, 1])}=+\infty$.

\begin{theorem}\label{thm_borne}
    We write 
    $$v(t,z) = \sum_{n\in\Z} c_n(t) \psi_n(z), $$ 
    with   $\dis v_0(z) = \sum_{n\in\Z} c_n(0) \psi_n(z)$ and suppose that $ \dis \sum_{n \in \Z} |c_n(0)|^2 < +\infty$, meaning that $v_0 \in L^2(S_{\gamma})$. Then, the equation \eqref{lin_LLL_hexa} is globally well-posed in the space $L^2(S_{\gamma})$. \medskip

    Moreover, we have the polynomial bound on the possible growth of the $L^2$-norm:   for all $t\geq 0$
    \begin{equation}\label{bound00}
        \|v(t)\|_{L^2(S_\gamma)}  \leq C(1+t)   \|v_0\|_{L^2(S_\gamma)}. 
            \end{equation}
            The previous bound is optimal in the sense that for all $\eps>0$, there exists $v_0  \in L^2(S_{\gamma})$ such that  for a sequence of times $t \longrightarrow +\infty$
                \begin{equation*} 
        \|v(t)\|_{L^2(S_\gamma)}  \geq C(1+t)^{1-\eps}   \|v_0\|_{L^2(S_\gamma)}. 
            \end{equation*}
         \end{theorem}

\begin{proof}
 We denote again
    $$ {U_0}(\xi) = \begin{pmatrix} {f}_0(\xi) \\ {g}_0(\xi)\end{pmatrix}, \qquad {f}_0(\xi) = \sum_{k\in\Z} c_k(0)e^{- 2i\pi k \xi}, \quad {g}_0(\xi) = \sum_{k\in\Z} \overline{c_k(0)}e^{- 2i\pi k \xi}.$$
 We compute with equality \eqref{expr_sg_fg} that
    \begin{multline*}
       \|{U}(t,\xi)\|_{L^2_\xi}^2 = \big\| e^{-it A_{\text{hexa}}(\xi) }\begin{pmatrix} {f}_0(\xi) \\[2pt] {g}_0(\xi) \end{pmatrix} \big\|_{L^2_\xi}^2\lesssim  \big\| \cos(t\mu){f}_0 -i\big(a{f}_0+b{g}_0\big)\frac{\sin(t\mu)}{\mu}\big\|_{L^2_\xi}^2  \\
	\lesssim \int_0^1 |\cos(t\mu(\xi)){f}_0(\xi)|^2d\xi + \int_0^1 \abs{\pac{a(\xi){f}_0(\xi)+b(\xi){g}_0(\xi)}\frac{\sin(t\mu(\xi))}{\mu(\xi)}}^2d\xi.  
    \end{multline*}
The first integral is bounded by $\dis C\| {f}_0\|_{L^2_\xi}^2 = C \|v_0 \|_{L^2_z}^2$. \medskip

$\bullet$ Let us prove \eqref{bound00}. Here we only use that $a$ and $b$ are bounded functions, and we use the estimate $\dis |\sin(t\mu(\xi))| \leq t\mu(\xi)$, and the result follows. \medskip

$\bullet$ Let $k_0\geq 2$ so that $2^{-k_0}<1 /2$. Define the indicator function $\dis J_k= {\bf 1}_{[{2^{-k-1}}, {2^{-k}}]}$ and consider the function defined on $[0,1]$ by
$${f}_0(\xi)=\sum_{k=k_0}^{+\infty}{k^{-\theta}}{2^{k/2}}J_k(\xi)+ \sum_{k=k_0}^{+\infty}{k^{-\theta}}{2^{k/2}}J_k(1-\xi)$$
which is extended as an even $1-$periodic function. Assume that $\theta>1/2$, then ${f}_0 \in L^2([0,1])$. Since~$f_0$ is a real-valued even function, we have $c_k(0) \in \R$ and thus ${f}_0={g}_0$. There exist two functions $\alpha, \beta$ such that $a(\xi)= \lambda +\alpha(\xi)$ and $b(\xi)=\lambda+\beta(\xi)$, and if $\delta>0$ is small enough, for all $|\xi| \leq \delta$, $|\alpha(\xi)| \leq c |\xi|$ and $|\beta(\xi)| \leq c |\xi|$. As a consequence 
 \begin{equation*} 
    \big|a(\xi){f}_0(\xi)+b(\xi){g}_0(\xi)\big|  \geq 2\lambda |{f}_0(\xi)|  -2c |\xi| |{f}_0(\xi)| \geq \lambda |{f}_0(\xi)|,
        \end{equation*}
so that we have 
 \begin{eqnarray*}
       \|{U}(t,\xi)\|_{L^2_\xi}^2 &\geq &\int_0^1 \abs{\pac{a(\xi){f}_0(\xi)+b(\xi){g}_0(\xi)}\frac{\sin(t\mu(\xi))}{\mu(\xi)}}^2d\xi \\
       &\geq &c   \int_0^\delta | {f}_0(\xi)  |^2 \Big| \frac{\sin(t\mu(\xi))}{\mu(\xi)}\Big|^2d\xi.
    \end{eqnarray*}
Let $k_1\geq k_0$ be such that  $2^{2k_1} \leq t \leq 2^{2k_1+1}$. Therefore, for $\xi \in {[{2^{-k-1}}, {2^{-k}}]}$, we have  $|\sin(t\mu(\xi))| \geq c |t\mu(\xi)| \geq c 2^{2(k_1-k)}$, so that $\dis  \Big| {\sin\big(t\mu(\xi)\big)}/{\mu(\xi)}\Big|^2 \geq c2^{4k_1}$. Hence, 
\begin{eqnarray*}
 \int_0^\delta | f_0(\xi)  |^2 \Big| \frac{\sin(t\mu(\xi))}{\mu(\xi)}\Big|^2d\xi   
&\geq &c 2^{4k_1}\int_0^\delta  \sum_{k=k_1}^{+\infty}k^{-2\theta} {2^{k}}J_k(\xi) d\xi \\
&\geq &c 2^{4k_1}   \sum_{k=k_1}^{+\infty}{k^{-2\theta}} \\
&= &c 2^{4k_1}k_1^{1-2\theta} \geq c t^{2(1-\eps)},
\end{eqnarray*}
which was the claim.
\end{proof}

\appendix

\section{Symmetries of~(\ref{LLL})}
For $\alpha \in \C$, define  the magnetic translation
\begin{equation*} 
R_{\alpha} : u (z) \mapsto u(z+\alpha) \EE\Big( (\overline z \alpha - z \overline{\alpha})/2\Big).
\end{equation*}
Let us recall some material from \cite[Appendix A]{SchTho}. We have 
    \begin{equation*} 
    R_{\alpha}u=e^{i(\alpha \cdot \Gamma)}u,
       \end{equation*}  
 with $\alpha=\alpha_1+i\alpha_2$ and $\alpha \cdot \Gamma := \alpha_1 \Gamma_1 +\alpha_2 \Gamma_2$, where $\Gamma_1$ and $\Gamma_2$ are defined by
  \begin{equation}\label{infinitesimal}
\Gamma_1=i(z -\partial_z-\frac{\ov z}2), \qquad \Gamma_2=(z +\partial_z+\frac{\ov z}2).
\end{equation}
 Notice that the operators $R_{\alpha}$ do not commute in general. Indeed,    for all $\alpha, \beta \in \C$ 
\begin{equation}\label{commut}
R_{\alpha}  R_\beta = e^{ {\ov \alpha} \beta-\alpha \ov{\beta} } R_{\beta} R_{\alpha}. 
\end{equation}
We also record the following formulas
\begin{equation}
\label{propRalpha}
\begin{split}
& R_\alpha R_\beta = R_{\alpha + \beta} \qquad \mbox{if $\alpha,\beta$ are colinear} \\
& (R_\alpha)^{-1} = R_{-\alpha} \\
& R_\alpha [ u(-z) ] = [R_{-\alpha} u](-z).
\end{split}
\end{equation}
Furthermore, 
\begin{align}\label{com0}
& R_\alpha \Pi u = \Pi R_\alpha u \qquad \mbox{if $|u(x)| \lesssim \langle x \rangle^M$} \\
& (R_\alpha)^* = R_{-\alpha} \qquad \mbox{on $L^2$}, \nonumber
\end{align}
and 
\begin{equation}\label{r-1}
R_{-\beta} ( \alpha \cdot \Gamma ) R_\beta =( \alpha \cdot \Gamma )     -2 \Im(\alpha \ov{\beta}).
\end{equation}

\section{Proof of Proposition \ref{prop41}}\label{appendix-B}

\label{App1}
 Let $u \in \mathcal{E}_{\tau,\gamma}$, and find first a fundamental cell $Q$ with respect to $\mathcal{L}_{\tau, \gamma}$ such that $u$ does not vanish on $\partial Q$. Let $\{ z_k \}_{1\leq k\leq N}$ be the zeroes of $u$ on $Q$, and write
    $$
    \dis u(z) = \EE\Big( -|z|^2/2 \Big) \varphi(z) \prod_{k=1}^N \Theta_{\tau}\Big(\frac{1}{\gamma}(z-z_{k})\Big).
    $$
    By construction, $\varphi$ does not vanish on $\mathbb{C}$, and is entire; thus, it can be written $\varphi = \EE\big(\Psi\big)$, with $\Psi$ entire. Furthermore, denoting $A = \mathbb{C} \setminus \cup_{a \in \mathcal{L}_{\tau,\gamma}} B(a,\eps)$ for $\eps>0$ small enough, it is easy to see that $|\Theta_\tau(z)| \geq C \EE\big({-C|z|^2}\big)$ on~$A$. Since $u$ is bounded on $\mathbb{C}$, this implies that $\mathfrak{Re} \Psi (z) \leq C |z|^2$ on~$A$, hence on $\mathbb{C}$ by the maximum principle; but the Borel-Caratheodory theorem implies then that $\Psi$ is actually a polynomial of degree at most 2. Therefore, we can write
\begin{equation}\label{eqA1}
    \dis u(z) = \lambda \EE\Big( {-|z|^2/2 + \alpha z^2 + \beta z}\Big) \prod_{k=1}^N\Theta_{\tau}\Big(\frac{1}{\gamma}(z-z_{k})\Big),
\end{equation}
    where $\alpha,\beta,\lambda \in \mathbb{C}$. We take for simplicity $\lambda =1$ in the following. The first periodicity condition of~$\mathcal{E}_{\tau,\gamma}$ requires that $u(z+\gamma) = \EE\big( {\gamma}(z-\overline{z})/2\big) u(z)$. Given the above form of $u$ and~\eqref{star}, this is equivalent to
    \begin{multline*}
(-1)^N     \EE\Big( -\frac{1}{2}|z|^2 - \frac{\gamma}{2} \ov{z} + \alpha z^2 + z(-\frac{\gamma}{2} + 2 \alpha \gamma+ \beta) + (\alpha \gamma+ \beta - \frac{\gamma}{2})\gamma \Big) \\
     =  \EE\Big( \frac{\gamma}{2}(z-\ov{z})\Big)  \EE\Big({-\frac{1}{2}|z|^2+\alpha z^2 + \beta z}\Big).
    \end{multline*}
    Identifying the coefficients of the polynomials in the exponential, we see that this equality holds if and only if $\alpha = 1/2$ and $\beta = -{iN\pi}/{\gamma} + {2ij\pi}/{\gamma}$, with $j \in \mathbb{Z}$. The second periodicity condition of $\mathcal{E}_{\tau,\gamma}$ demands that $\dis u(z+\gamma\tau) = \EE\big({{\gamma}(\ov{\tau}z-\tau \ov{z})/2}\big)u(z)$. For~$u$ as above, using~\eqref{star}, which gives
    $$\Theta_{\tau}\Big(\frac{1}{\gamma}(z-z_{k}+\gamma \tau)\Big)   =\Theta_{\tau}\Big(\frac{1}{\gamma}(z-z_{k})+ \tau\Big)   =- e^{-i \pi \tau}e^{- {2i\pi} (z-z_k)/ {\gamma}}\Theta_{\tau}\Big(\frac{1}{\gamma}(z-z_{k})\Big) ,$$
     the periodicity condition is equivalent to
    \begin{multline*}
    \EE\Big(-\frac{1}{2}|z|^2 + \frac{1}{2}z^2 -\frac{\gamma}{2} \tau \ov{z} + z(-\frac{\gamma}{2} \ov{\tau} - \frac{2iN\pi}{\gamma} + \gamma\tau + \beta) \Big)\\
  \times   \EE\Big( (-\frac{\gamma^2}{2} |\tau|^2 + Ni\pi - Ni\pi \tau + \frac{2i\pi}{\gamma} S_N + \frac{\gamma^2}{2} \tau^2 + \beta \gamma \tau)\Big)\\
     = \EE\Big( \frac{\gamma}{2}(\ov{\tau} z - \tau \ov{z})\Big) \EE\Big(-\frac{1}{2} |z|^2 + \frac{1}{2} z^2 + \beta z\Big),
      \end{multline*}
    where $S_N=\sum_{k=1}^Nz_k$. Identifying the coefficients in the polynomials, this equality holds if and only if $\tau_2 = {N\pi}/{\gamma^2}$ (this comes from the identification of the factor $z$) and
    $$
    \frac{\gamma^2}{2} \tau^2 - \frac{\gamma^2}{2} |\tau|^2 + Ni\pi - Ni\pi\tau + \frac{2i\pi}{\gamma} S_N + \beta \gamma\tau = 2i \ell  \pi, \quad \mbox{with $\ell \in \mathbb{Z}$},
    $$
    coming from the constant term. Since $\tau_2 = {N\pi}/{\gamma^2}$, this simplifies to $N + {2}S_N/{\gamma} - N\tau + 2j\tau = 2\ell$, with $\ell \in \mathbb{Z}$. Therefore, $S_N = {\gamma N}(\tau-1)/2+\gamma \ell-\gamma \tau j$. Overall, we found that~$u$ reads
    $$
    u(z) = \EE\Big( {-\frac{1}{2}|z|^2 + \frac{1}{2} z^2 + \frac{i\pi z}{\gamma}(-N + 2j)}\Big) \prod_{k=1}^N \Theta_{\tau}\Big(\frac{1}{\gamma}(z-z_{k})\Big),$$
with
$$ \sum_{k=1}^N z_k = \frac{\gamma N}{2}(\tau-1)+\gamma \ell-\gamma \tau j. $$

There remains to show that $\mathcal{E}_{\tau,\gamma}$ is a vector space of dimension $N$. This follows from two observations: on the one hand, it is a vector space by definition, and on the other hand the number of free parameters in the formula for $u$ is $N$ (the multiplicative constant, and the $N$ zeros which have a prescribed sum).

\section{Useful formulas}

\subsection{Gaussian integral} If $a>0$, $b \in \mathbb{R}$,
\begin{equation}
\label{GaussianIntegral}
\int_{\R} \EE\Big( {-at^2+bt}\Big) dt = \sqrt{\frac{\pi}{a}} \EE\Big(    \frac{b^2}{4a}\Big).
\end{equation}

\subsection{Poisson summation}
Let us recall the Poisson summation formula (see for instance \cite[Chapter VII]{SteinWeiss}): for  $g \in \mathscr{S}(\R)$  set 
$$\hat{g}(n)=\int_{-\infty}^{+\infty}g(x)e^{-2i \pi nx}dx,$$
then for all $ x \in \R$
\begin{equation}\label{PF}
\sum_{n \in \Z}g(x+n)=\sum_{n \in \Z}\hat{g}(n)e^{2i \pi n x}.
\end{equation}
In particular, for $\alpha >0$ and  $g : x \longmapsto \EE\big({-\alpha x^2}\big)$, and by an analytical extension,  we get that for all $z \in \C$
\begin{equation}\label{PF1}
\sum_{n \in \Z} \EE\Big({-\alpha (z+n)^2}\Big) =\sqrt{\frac{\pi}{\alpha}} \sum_{n \in \Z} \EE\Big( {-{\pi^2 n^2 }/{\alpha}+2 i \pi n z}\Big).
\end{equation}

\section{Some conservation laws for \eqref{lin_LLL_hexa}}

\begin{lemme}\label{lemme_conserv_hyp1}
	The set 
$$ \dis \mathcal{A} = \paa{ u = \sum_{n\in\Z} c_n\psi_n, \;\; \Re\pa{\sum_{n\in\Z}c_n} = \Re\pa{\sum_{n\in\Z}nc_n} = 0 } $$
is preserved by the flow of \eqref{lin_LLL_hexa}. More precisely, suppose $\dis v_0 = \sum_{n\in\Z} c_n(0)\psi_n \in \mathcal{F}_\gamma$  and consider the solution $\dis v(t)= \sum_{n\in\Z} c_n(t)\psi_n \in \mathcal{F}_\gamma$ of \eqref{lin_LLL_hexa} with initial data~$v_0$. For an integer $j \geq 0$, we denote 
$$\dis K_j(t) = \sum_{n\in\Z}{n^jc_n(t)} = R_j(t) +iI_j(t)$$
 with $R_j,I_j \in \R$. Then:
    \begin{enumerate}[$(i)$]
	\item For $0 \leq j \leq 3$ an integer, the real part of $K_j(t)$ is constant: $R_j(t) = R_j(0).$ Furthermore, if $j\in\{ 0,1\}$ and $R_j(0) = 0$, then the imaginary part $I_j(t)$ is also constant.
	\item The real part of $K_j(t)$ is \emph{not always} constant for $j = 4.$
\end{enumerate}
\end{lemme}
 
Recall that by \eqref{def1.1}, $\dis \sum_{n \in \Z} \psi_n= \kappa \Phi_0$. Then  if $  \dis v=\sum_{n \in \Z}       c_n \psi_n $, for all $j \geq 0$ we have
$$  \Gamma_1^j  v=\sum_{n \in \Z} \big(\frac{2 n \pi}\gamma\big)^j     c_n \psi_n, $$
therefore, 
\begin{equation*} 
\Re\pa{\sum_{n\in\Z}n^jc_n} =\frac{\ov{\kappa} \gamma^j}{2(2\pi)^j} \int_{S_{\gamma}}  \Gamma^j_1  (v+\widetilde{v}) \ov{\Phi_0} \,dL. 
\end{equation*}

\begin{proof}
We multiply equation \eqref{Lin_coeff_hexa} by $n^j$, and take the sum over $n \in \Z$. Thanks to expression \eqref{convol} we obtain
\begin{multline} \label{sum_Kj}
i\partial_t K_j + \lambda K_j =  \\
  =  \sum_{m\in\Z}c_m\sum_{n\in\Z} n^j {(C'_{L} L- C_L^{\text{odd}}L^{\text{odd}})}_{n-m} +  \sum_{m\in\Z}\ov{c_m}\sum_{n\in\Z} n^j {(C'_{M}M-C_M^{\text{odd}}M^{\text{odd}})}_{n-m},
\end{multline}
where $L$, $M$, $L^{\text{odd}}$ and $M^{\text{odd}}$ are defined in \eqref{def_op_L_rect}, \eqref{def_op_M_rect}, \eqref{def-Lodd} and \eqref{defModd} respectively, and which we recall here:
\begin{equation*}  
L_n = \EE\Big(-\frac{\pi^2n^2}{\gamma^2}\Big), \qquad M_n = \sum_{p\in \Z}\EE\Big({-\frac{\pi^2p^2}{\gamma^2}}\Big) \EE\Big({-\frac{\pi^2}{\gamma^2}(n-p)^2}\Big),
\end{equation*}	
\begin{equation*} 
    L_m^{\text{odd}} = \delta_{m\in2\Z+1}\EE\Big({-\frac{\pi^2m^2}{\gamma^2}} \Big), 
    \end{equation*}
\begin{equation*} 
	M_n^{\text{odd}} =\sum_{k \in 2\Z+1} \EE\Big(-\frac{\pi^2}{\gamma^2}k^2\Big) \delta_{n\in2\Z}\EE\Big(-\frac{\pi^2}{\gamma^2}(n-k)^2\Big).
\end{equation*}

For $\ell \in \N$, we set 
$$ T_\ell := \sum_{q\in\Z}q^\ell\EE\Big( {-\frac{\pi^2q^2}{\gamma^2}}\Big), \qquad T_\ell^{\text{odd}}:=\sum_{q\in2\Z+1}q^\ell\EE\Big( {-\frac{\pi^2q^2}{\gamma^2}}\Big),$$
and note that  $T_\ell = T_\ell^{\text{odd}}=0$ if $\ell \in 2\Z+1$. \medskip

Let us first show that we can rewrite \eqref{sum_Kj} as
\begin{equation}\label{result}
 i\partial_t K_j + \lambda K_j = \sum_{\ell=0}^j \dbinom{j}{\ell}\, \Big( d_{j-\ell}\, K_\ell + s_{j-\ell}\, \ov{K_\ell}\Big), 
 \end{equation}
where $d_p := C'_L T_p - C_L^{\text{odd}} T^{\text{odd}}_p$ and $\dis s_p := \sum_{k=0}^p \dbinom{p}{k}\big(C'_M T_{p-k}T_k - C_M^{\text{odd}} T^{\text{odd}}_{p-k}T^{\text{odd}}_k\big)$. We can observe that $d_\ell = s_\ell=0$ if $\ell \in 2\Z+1$. \medskip

We compute by a binomial formula
\begin{eqnarray*}
\sum_{n\in\Z} n^j L_{n-m} &= &\sum_{n\in\Z} (n-m+m)^j \EE\Big( {-\frac{\pi^2(n-m)^2}{\gamma^2}}\Big)\\
&= &\sum_{\ell=0}^j \dbinom{j}{\ell} m^{j-\ell} \sum_{n\in\Z}(n-m)^\ell\EE\Big( {-\frac{\pi^2(n-m)^2}{\gamma^2}}\Big) \\
&=& \sum_{\ell=0}^j \dbinom{j}{\ell} m^{j-\ell} T_\ell.
\end{eqnarray*}
Similarly, 
\begin{equation*}
    \sum_{n\in\Z} n^j L_{n-m}^{\text{odd}} = \sum_{\ell=0}^j \dbinom{j}{\ell} m^{j-\ell} T_\ell^{\text{odd}}.
\end{equation*}
As a consequence we obtain
$$  \sum_{m\in\Z}c_m\sum_{n\in\Z} n^j {(C'_LL-C_L^{\text{odd}}L^{\text{odd}})}_{n-m} = \sum_{\ell=0}^j \dbinom{j}{\ell}   d_{j-\ell}\,K_{\ell},$$ 
which gives the first part of \eqref{result}. For the second sum we argue in the same way, with one extra convolution. For all $r\in\Z$,
$$ M_r = \sum_{k\in\Z} L_k\,L_{r-k}, \qquad M^{\text{odd}}_r = \sum_{k\in\Z} L^{\text{odd}}_k\,L^{\text{odd}}_{r-k}. $$
Expanding $r^p = \big(k+(r-k)\big)^p$ and using $\dis \sum_{q\in\Z} q^b L_q = T_b$, $\dis \sum_{q\in\Z} q^b L^{\text{odd}}_q = T^{\text{odd}}_b$, we obtain for every $p\geq 0$
$$ \sum_{r\in\Z} r^p M_r = \sum_{k=0}^p \dbinom{p}{k} T_k\,T_{p-k}, \qquad \sum_{r\in\Z} r^p M^{\text{odd}}_r = \sum_{k=0}^p \dbinom{p}{k} T^{\text{odd}}_k\,T^{\text{odd}}_{p-k}. $$
Then, exactly as for $L$, the substitution $n=(n-m)+m$ yields
$$ \sum_{n\in\Z} n^j M_{n-m} = \sum_{p=0}^j \dbinom{j}{p}\, m^{j-p}\sum_{k=0}^p \dbinom{p}{k}\, T_k\,T_{p-k}, $$
and similarly for $M^{\text{odd}}$. Multiplying by $\ov{c_m}$ and summing over $m\in\Z$ and combining with \eqref{sum_Kj} and the first sum we get the result \eqref{result}.
\medskip

Let us show that  $d_0=2\lambda$. Evaluating \eqref{def_fct_lb} and \eqref{def_fct_h} at $\xi=0$,
$$
T_0 = \sum_{q\in\Z}\EE\Big(-\frac{\pi^2 q^2}{\gamma^2}\Big) = \ell(0),
\qquad
T_0^{\text{odd}} = \sum_{q\in 2\Z+1}\EE\Big(-\frac{\pi^2 q^2}{\gamma^2}\Big) = h(0).
$$
Since $C'_L=\frac{2}{\gamma\sqrt{\pi}}T_0$ and $C_L^{\text{odd}}=\frac{4}{\gamma\sqrt{\pi}}T_0^{\text{odd}}$, we obtain
$$
d_0=C'_L\, T_0 - C_L^{\text{odd}}\, T_0^{\text{odd}}
= \frac{2}{\gamma\sqrt{\pi}}\,T_0^{2} - \frac{4}{\gamma\sqrt{\pi}}\big(T_0^{\text{odd}}\big)^{2}
= \frac{2}{\gamma\sqrt{\pi}}\Big(\ell^2(0) - 2\,h^2(0)\Big) = 2\lambda,
$$
the last equality being the definition \eqref{def_lambda_hexa} of $\lambda$. Similarly, $s_0=\lambda$.
\medskip

As a consequence, we get the following formulas: \smallskip

\noindent $\bullet$ For $j=0$ or $j=1$ :
$$ i \partial_t K_j  = \lambda (K_j + \ov{K_j}),$$
so that $\partial_tR_j = 0$ and $-\partial_t I_j = 2\lambda R_j.$ \\
$\bullet$ For $j=2$ :
$$ i \partial_t K_2   = \lambda (K_2 + \ov{K_2}) + 2D_0(K_0 + \ov{K_0}), \quad D_0:= \frac{1}{\gamma\sqrt{\pi}}(T_0T_2-2T_0^{\text{odd}}T_2^{\text{odd}}),$$
so that the equation on $K_2=R_2+iI_2$ becomes $\partial_tR_2 =0$ and $-\partial_tI_2 = 2\lambda R_2 + 4D_0R_0$. \\
$\bullet$ For $j=3$ :
$$ i\partial_tK_3  = \lambda (K_3 + \ov{K_3}) + 6D_0(K_1 + \ov{K_1}), $$
so that $\partial_tR_3 =0$ as before. \\
$\bullet$ For $j=4$ :
$$ i\partial_tK_4  = \lambda (K_4 + \ov{K_4}) + 12D_0(K_2 + \ov{K_2}) + D_2(K_0 + \ov{K_0}) + D_3 \ov{K_0},$$
where 
$$ D_2 = \frac{2}{\gamma\sqrt{\pi}}\pa{T_0 T_4-2T_0^{\text{odd}}T_4^{\text{odd}}}, \qquad D_3 = \frac{6}{\gamma\sqrt{\pi}}\pa{T_2^2-2(T_2^{\text{odd}})^2} \neq  0.$$
Then, $\partial_tR_4 = - D_3I_0 $ and $-\partial_tI_4 = 2\lambda R_4 +24D_0 R_2 +(2D_2+D_3)R_0, $ so that $R_4$ is not preserved when $I_0 \neq 0.$ \\
$\bullet$  Finally, for $j\ge5$ equation \eqref{result} reads
$$ i\partial_tK_j  = \lambda (K_j + \ov{K_j}) + j(j-1)D_0(K_{j-2} + \ov{K_{j-2}}) +\sum_{\ell=0}^{j-4} \dbinom{j}{\ell}\, \Big( d_{j-\ell}\, K_\ell + s_{j-\ell}\, \ov{K_\ell}\Big), $$
which we record for completeness.
\end{proof}
   


\end{document}